\documentclass[12pt]{article}
\usepackage{latexsym, amsmath, amscd, bbm}
\usepackage[shortlabels]{enumitem}
\usepackage{amsthm}
\usepackage{graphicx, xspace, ifthen, rotating, pict2e}
\usepackage[svgnames]{xcolor}
\usepackage[charter]{mathdesign}
\usepackage{algpseudocode, algorithmicx}
\usepackage{bbm}
\usepackage{tikz}\pdfpageattr{/Group <</S /Transparency /I true /CS /DeviceRGB>>}
\usetikzlibrary{arrows,calc,intersections,matrix,positioning,through}
\usepackage[normal]{caption}
\usepackage[aligntableaux=center]{ytableau}\ytableausetup{boxsize=16pt}
\usepackage[bookmarksnumbered, bookmarksopen, bookmarksdepth=99, linktocpage=true, linktoc=all, colorlinks=true]{hyperref}
\usepackage[capitalize,noabbrev,compress]{cleveref}

\newcommand{\hidelinks}{\hypersetup{linkcolor=Black}}
\newcommand{\restorelinks}{\hypersetup{allcolors=Navy!75!DodgerBlue,citecolor=DarkGreen!75!Green}}
\restorelinks

\newcommand{\fullcref}[1]{\hyperref[#1]{\cref{#1}}}

\allowdisplaybreaks[1]          

\setlist{topsep=.5ex, itemsep=.5ex, parsep=0ex, partopsep=0ex}

\newcommand{\CC}{\mathbb{C}}
\newcommand{\ZZ}{\mathbb{Z}}

\newcommand{\RR}{\mathbb{R}}

\newcommand{\PP}{\mathbb{P}}

\newcommand{\RP}{\mathbb{RP}}

\newcommand{\calB}{\mathcal{B}}
\newcommand{\calC}{\mathcal{C}}
\newcommand{\calD}{\mathcal{D}}

\newcommand{\calG}{\mathcal{G}}

\newcommand{\calI}{\mathcal{I}}

\newcommand{\calP}{\mathcal{P}}

\newcommand{\calT}{\mathcal{T}}

\newcommand{\calX}{\mathcal{X}}

\newcommand{\bolda}{\mathbf{a}}
\newcommand{\boldb}{\mathbf{b}}

\newcommand{\boldx}{\mathbf{x}}

\newcommand{\boldz}{\mathbf{z}}

\newcommand{\boldlambda}{{\mathchoice
{\mbox{\boldmath{$\lambda$}}}
{\mbox{\boldmath{$\lambda$}}}
{\mbox{\scriptsize\boldmath{$\lambda$}}}
{\mbox{\tiny\boldmath{$\lambda$}}}
}}
\newcommand{\boldmu}{{\mathchoice
{\mbox{\boldmath{$\mu$}}}
{\mbox{\boldmath{$\mu$}}}
{\mbox{\scriptsize\boldmath{$\mu$}}}
{\mbox{\tiny\boldmath{$\mu$}}}
}}

\newcommand{\boldpsi}{{\mathchoice
{\mbox{\boldmath{$\psi$}}}
{\mbox{\boldmath{$\psi$}}}
{\mbox{\scriptsize\boldmath{$\psi$}}}
{\mbox{\tiny\boldmath{$\psi$}}}
}}

\newcommand{\Plucker}{Pl\"ucker\xspace}

\newcommand{\PGL}{\mathrm{PGL}}
\newcommand{\gl}{\mathfrak{gl}}
\newcommand{\GL}{\mathrm{GL}}
\newcommand{\Wr}{\mathrm{Wr}}
\newcommand{\Gr}{\mathrm{Gr}}

\newcommand{\Proj}{\mathop{\mathrm{Proj}}}
\newcommand{\Spec}{\mathop{\mathrm{Spec}}}
\newcommand{\codim}{\mathop{\mathrm{codim}}}

\newcommand{\End}{\mathrm{End}}
\newcommand{\Hom}{\mathrm{Hom}}

\newcommand{\sgn}{\mathrm{sgn}}

\newcommand{\inv}{\mathrm{inv}}

\newcommand{\ev}{\mathrm{ev}}

\newcommand{\numsyt}[1]{\mathsf{f}^{#1}}

\newcommand{\rectangle}{{\mathchoice%
{\scalebox{1.35}{\raisebox{-0.02em}{$\sqsubset\!\!\sqsupset$}}}
{\scalebox{1.35}{\raisebox{-0.02em}{$\sqsubset\!\!\sqsupset$}}}
{\sqsubset\!\!\sqsupset}
{\sqsubset\!\!\sqsupset}
}}

\newtheorem{lemma}{Lemma}[section]
\newtheorem{theorem}[lemma]{Theorem}

\newtheorem{proposition}[lemma]{Proposition}
\newtheorem{corollary}[lemma]{Corollary}
\newtheorem{problem}[lemma]{Problem}
\newtheorem{conjecture}[lemma]{Conjecture}

\theoremstyle{definition}
\newtheorem{example}[lemma]{Example}

\newtheorem{remark}[lemma]{Remark}

\numberwithin{equation}{section}
\numberwithin{table}{section}

\renewcommand{\eqref}[1]{\hyperref[#1]{\textup{(\ref*{#1})}}}

\definecolor{DarkBlue}{rgb}{0, 0.1, 0.55}
\definecolor{DarkRed}{rgb}{0.45, 0, 0}

\newcommand{\defn}[1]{\textbf{\textit{#1}}}

\AddToHook{env/theorem/begin}{\crefalias{lemma}{theorem}}
\AddToHook{env/fact/begin}{\crefalias{lemma}{fact}}
\AddToHook{env/proposition/begin}{\crefalias{lemma}{proposition}}
\AddToHook{env/corollary/begin}{\crefalias{lemma}{corollary}}
\AddToHook{env/problem/begin}{\crefalias{lemma}{problem}}
\AddToHook{env/conjecture/begin}{\crefalias{lemma}{conjecture}}
\AddToHook{env/example/begin}{\crefalias{lemma}{example}}
\AddToHook{env/definition/begin}{\crefalias{lemma}{definition}}
\AddToHook{env/ack/begin}{\crefalias{lemma}{ack}}
\AddToHook{env/remark/begin}{\crefalias{lemma}{remark}}
\AddToHook{env/condition/begin}{\crefalias{lemma}{condition}}
\AddToHook{env/convention/begin}{\crefalias{lemma}{convention}}
\AddToHook{env/algorithm/begin}{\crefalias{lemma}{algorithm}}

\crefname{lemma}{Lemma}{Lemmas}
\crefname{theorem}{Theorem}{Theorems}
\crefname{fact}{Fact}{Facts}
\crefname{proposition}{Proposition}{Propositions}
\crefname{corollary}{Corollary}{Corollaries}
\crefname{problem}{Problem}{Problems}
\crefname{conjecture}{Conjecture}{Conjectures}
\crefname{example}{Example}{Examples}
\crefname{definition}{Definition}{Definitions}
\crefname{ack}{Acknowledgements}{Acknowledgements}
\crefname{remark}{Remark}{Remarks}
\crefname{condition}{Condition}{Conditions}
\crefname{convention}{Convention}{Conventions}
\crefname{algorithm}{Algorithm}{Algorithms}

\title{Universal Pl\"ucker coordinates for the Wronski map and positivity in real Schubert calculus}
\date{}
\author{Steven N.\ Karp and Kevin Purbhoo}
\AtEndDocument{\bigskip{\footnotesize%
\textsc{Department of Mathematics, University of Notre Dame} \par  
E-mail address: \href{mailto:skarp2@nd.edu}{\texttt{skarp2@nd.edu}} \par
\addvspace{\medskipamount}
\textsc{Combinatorics and Optimization Department, University of Waterloo} \par  
E-mail address: \href{mailto:kpurbhoo@uwaterloo.ca}{\texttt{kpurbhoo@uwaterloo.ca}} \par
}}

\newcommand{\scell}[1]{\calX^{#1}}
\newcommand{\svar}[1]{\smash{\overline{\calX}}^{#1}}
\newcommand{\scellnu}{\scell{\nu}}
\newcommand{\svarnu}{\svar{\nu}}

\newcommand{\svarrect}{\svar{\rectangle}}
\newcommand{\monics}{\calP_n}
\newcommand{\bethe}{\calB_n(z_1, \dots, z_n)}
\newcommand{\bethezero}{\calB_n(0, \dots, 0)}
\newcommand{\bethenu}{\calB_\nu(z_1, \dots, z_n)}
\newcommand{\altbethe}{\overline{\calB}_n(z_1, \dots, z_n)}
\newcommand{\altbethenu}{\overline{\calB}_\nu(z_1, \dots, z_n)}
\newcommand{\translatebethe}[1]{\calB_n(z_1+#1, \dots, z_n+#1)}
\newcommand{\translatebethenu}[1]{\calB_\nu(z_1+#1, \dots, z_n+#1)}
\newcommand{\translatealtbethe}[1]{\overline{\calB}_n(z_1+#1, \dots, z_n+#1)}
\newcommand{\custombethe}[1]{\calB_{#1}}
\newcommand{\multibethe}[1]{\calB_{#1}(\boldz_\kappa)}
\newcommand{\fullbethe}{\widehat{\calB}_n}
\newcommand{\fullbethesing}{\fullbethe^\mathrm{sing}}
\newcommand{\fullbetheweight}[1]{\widehat{\calB}^\mathrm{sing}_{n,#1}}
\newcommand{\symgrp}[1]{\mathfrak{S}_{#1}}
\newcommand{\Sn}{\symgrp{n}}
\newcommand{\Snidentity}{\mathbbm{1}_{\Sn}}
\newcommand{\Sidentity}[1]{\mathbbm{1}_{\symgrp{#1}}}
\newcommand{\SP}{\mathrm{SP}}
\newcommand{\specht}[1]{M^{#1}}
\newcommand{\spechtnu}{\specht{\nu}}
\newcommand{\multispecht}[1]{\mathbf{M}^{#1}}
\newcommand{\h}{\mathsf{h}}
\newcommand{\e}{\mathsf{e}}
\newcommand{\m}{\mathsf{m}}
\newcommand{\sm}{\widetilde{\mathsf{m}}}
\newcommand{\p}{\mathsf{p}}
\newcommand{\s}{\mathsf{s}}
\newcommand{\z}{\mathsf{z}}
\newcommand{\cyc}{\mathsf{cyc}}
\newcommand{\sort}{\mathsf{sort}}
\newcommand{\fls}{{\CC((t^{-1}))}}
\newcommand{\E}{\mbox{\small\sf{E}}}
\renewcommand{\H}{\mbox{\small\sf{H}}}
\newcommand{\B}{\mbox{\small\sf{B}}}
\newcommand{\ssum}{\mbox{\small\sf{S}}}
\newcommand{\Bresidue}{[t^{-1}]\B^\perp(t^{-1})}
\newcommand{\alphaminus}[1]{\alpha^{1^{\smash{#1}}}}
\newcommand{\betaminus}[1]{\beta^{1^{\smash{#1}}}}
\newcommand{\glbeta}[1]{\smash{\widehat{\beta}}_{#1}}
\newcommand{\du}{\partial_u}
\newcommand{\RF}{\mathrm{RF}}
\newcommand{\exterior}[1]{\mathsf{\Lambda}^{#1}}
\newcommand{\extalg}{\exterior{\bullet}}
\newcommand{\polyD}{\overline{\calD}}
\newcommand{\translatematrix}[1]{(\begin{smallmatrix} 1 & #1 \\ 0 & 1 \end{smallmatrix})}
\newcommand{\scalematrix}[1]{(\begin{smallmatrix} #1 & 0 \\ 0 & 1 \end{smallmatrix})}
\newcommand{\inversionmatrix}{(\begin{smallmatrix} 0 & 1 \\ 1 & 0 \end{smallmatrix})}
\newcommand{\pker}{\mathop{\mathrm{pker}}}
\newcommand{\multipartition}{\mathrel{\,\Vdash}}

\newcommand{\exspec}{\mathop{\mathrm{ex}}}
\newcommand{\Vmuz}{\mathbf{V}^\boldmu(z_1, \dots, z_s)}
\newcommand{\Vone}{\mathbf{V}^{(1,\dots, 1)}(z_1, \dots, z_n)}
\newcommand{\conjugate}[1]{#1^*}
\newcommand{\trans}[2]{\sigma_{#1,#2}}

\newcommand{\GT}{\calG\calT_n}

\begin{document}

\maketitle

\begin{abstract}
\noindent Given a $d$-dimensional vector space $V \subset \mathbb{C}[u]$ of polynomials, its Wronskian is the polynomial $(u + z_1) \cdots (u + z_n)$ whose zeros $-z_i$ are the points of $\mathbb{C}$ such that $V$ contains a nonzero polynomial with a zero of order at least $d$ at $-z_i$. Equivalently, $V$ is a solution to the Schubert problem defined by osculating planes to the moment curve at $z_1, \dots, z_n$. The \emph{inverse Wronski problem} involves finding all $V$ with a given Wronskian $(u + z_1) \cdots (u + z_n)$. We solve this problem by providing explicit formulas for the Grassmann--Pl\"ucker coordinates of the general solution $V$, as commuting operators in the group algebra $\mathbb{C}[\mathfrak{S}_n]$ of the symmetric group. The Pl\"ucker coordinates of individual solutions over $\mathbb{C}$ are obtained by restricting to an eigenspace and replacing each operator by its eigenvalue. This generalizes work of Mukhin, Tarasov, and Varchenko (2013) and of Purbhoo (2022), which give formulas in $\mathbb{C}[\mathfrak{S}_n]$ for the differential equation satisfied by $V$. Moreover, if $z_1, \dots, z_n$ are real and nonnegative, then our operators are positive semidefinite, implying that the Pl\"ucker coordinates of $V$ are all real and nonnegative. This verifies several outstanding conjectures in real Schubert calculus, including the positivity conjectures of Mukhin and Tarasov (2017) and of Karp (2021), the disconjugacy conjecture of Eremenko (2015), and the divisor form of the secant conjecture of Sottile (2003). The proofs involve the representation theory of $\mathfrak{S}_n$, symmetric functions, and $\tau$-functions of the KP hierarchy.
\end{abstract}


\setcounter{tocdepth}{2}
\hidelinks
\tableofcontents
\restorelinks

\section{Introduction}

For a system of real polynomial equations with finitely many solutions,
we normally expect that some --- but not all --- of the solutions are real,
while the remaining solutions come in complex-conjugate pairs.
The precise number of real solutions usually depends in a complicated way
on the coefficients of the equations.
However, in some rare cases, it is possible to obtain a better
understanding of the real solutions.
A remarkable example occurs in the Schubert calculus of
the Grassmannian $\Gr(d,m)$,
for Schubert problems defined by flags osculating
a rational normal curve.  In 1993, Boris and Michael Shapiro conjectured
that all such Schubert problems with real parameters have only real 
solutions.  The corresponding systems of equations arise in various
guises throughout mathematics, 
from algebraic curves \cite{eisenbud_harris83,kharlamov_sottile03} 
to differential equations \cite{mukhin_varchenko04} to pole-placement problems \cite{rosenthal_sottile98,eremenko_gabrielov02a}.
The conjecture was eventually proved by Mukhin, Tarasov, and Varchenko \cite{mukhin_tarasov_varchenko09a},
using a reformulation in terms of Wronski maps, 
and machinery from quantum integrable systems and representation
theory.

While the details of the Mukhin--Tarasov--Varchenko proof are rather
intricate, the basic idea is relatively straightforward.  
They consider a family of commuting linear operators arising from the Gaudin model,
and show that they satisfy algebraic equations defining a Schubert problem.
Hence, by considering the spectra of
these operators, they are able to infer some basic properties of
the solutions to the Schubert problem.  In this paper we extend these
results, making the connection between the commuting operators and the corresponding solutions more explicit and concrete.
Consequently, we obtain stronger results
in real algebraic geometry, including several generalizations of the Shapiro--Shapiro conjecture. Namely, we resolve the divisor form of the secant conjecture of Sottile (2003), the disconjugacy conjecture of Eremenko \cite{eremenko15}, and the positivity conjectures of Mukhin--Tarasov (2017) and Karp \cite{karp}.


\subsection{The Wronski map and the Bethe algebra}
\label{sec:wronskiintro}

Let $\Gr(d,m)$ denote the Grassmannian of all $d$-dimensional linear subspaces of $\CC^m$. It is often more convenient to work with the $m$-dimensional 
vector space $\CC_{m-1}[u]$, of
univariate polynomials of degree at most $m-1$, rather than $\CC^m$.
We explicitly identify $\CC^m$ with $\CC_{m-1}[u]$, via the isomorphism
\begin{equation}
\label{eq:isomorphism}
     (a_1, \dots, a_m)  \leftrightarrow  \sum_{j=1}^m a_j \frac{u^{j-1}}{(j-1)!}
\,.
\end{equation}
In particular, we also view $\Gr(d,m)$ as the space of 
all $d$-dimensional subspaces of $\CC_{m-1}[u]$.

Now fix a nonnegative integer $n$, and let $\nu$ be a 
partition of $n$ with at most $d$ parts;
that is, $\nu = (\nu_1, \dots, \nu_d)$ is a tuple of nonnegative 
integers such that $\nu_1 \geq \dots \geq \nu_d \geq 0$, and 
$|\nu| := \nu_1 + \dots + \nu_d = n$.  
The \defn{Schubert cell}
$\scellnu$ is the space of all $d$-dimensional linear subspaces of $\CC[u]$
that have a basis $(f_1, \dots, f_d)$, with
$\deg(f_i) = \nu_i+d-i$.  As a scheme, $\scellnu$ is isomorphic to
$n$-dimensional affine space.
We take $m \geq d+\nu_1$, so that $\scellnu \subseteq \Gr(d,m)$.

Let $\monics \subseteq \CC[u]$ denote the $n$-dimensional affine space of
monic polynomials of degree $n$.
Given $V \in \scellnu$, choose any basis $(f_1, \dots, f_d)$ for $V$.
We define $\Wr(V)$ to be the unique monic polynomial which is a scalar
multiple of the Wronskian $\Wr(f_1, \dots, f_d)$.  It is not hard to see that 
$\Wr(V) \in \monics$ is a polynomial of degree $n$, and 
is independent of the choice of basis.  Thus we
obtain a map $\Wr : \scellnu \to \monics$, called the \defn{Wronski map}
on $\scellnu$.  Abstractly, this is a finite morphism 
from $n$-dimensional affine
space to itself.

Suppose $g(u) = (u+z_1)\dotsm (u+z_n) \in \monics$,
where $z_1, \dots, z_n$ are complex numbers.   The 
\defn{inverse Wronski problem} is to compute the fibre 
$\Wr^{-1}(g) \subseteq \scellnu$.

In their study of the Gaudin model for $\gl_n$,
Mukhin, Tarasov, and 
Varchenko \cite{mukhin_varchenko04,mukhin_varchenko05,mukhin_tarasov_varchenko06,mukhin_tarasov_varchenko09a,mukhin_tarasov_varchenko09b} discovered a connection between the inverse Wronski problem,
and the problem of diagonalizing the Gaudin Hamiltonians \cite{gaudin76}. 
We will focus on the version of this story from \cite{mukhin_tarasov_varchenko13}, in which the Gaudin Hamiltonians generate the \emph{Bethe algebra (of Gaudin type)} $\bethe \subseteq \CC[\Sn]$, which is a commutative subalgebra of the group algebra of the symmetric group.

Let $\spechtnu$ be the Specht module (i.e.\ irreducible $\Sn$-representation) associated to the partition $\nu$. 
Then $\bethe$ acts on $\spechtnu$, and the image of 
this action defines a commutative subalgebra 
$\bethenu \subseteq \End(\spechtnu)$.
The following result is
stated more precisely as \cref{thm:precise}:

\begin{theorem}[Mukhin, Tarasov, and Varchenko \cite{mukhin_tarasov_varchenko13}]
\label{thm:vague}
The eigenspaces $E \subseteq \spechtnu$ of the algebra 
$\bethenu$ are in one-to-one 
correspondence with the points $V_E \in \Wr^{-1}(g)$.
The eigenvalues of the generators of $\bethenu$ are
coordinates for $V_E$ in some coordinate system.
\end{theorem}

(There are also scheme-theoretic analogues of \cref{thm:vague}, which we discuss in \cref{sec:schemeintro}.)
Unfortunately, \cref{thm:vague}/\ref{thm:precise} 
is poorly suited to studying certain properties of the Wronski map.  This is 
because the generators of $\bethe$ 
correspond to a somewhat unusual coordinate system for $\scellnu$.  
Namely, given $V \in \scellnu$, there is a unique \defn{fundamental differential operator} 
\[
    D_V = \partial_u^d + \psi_1(u) \partial_u^{d-1} + \dots + \psi_d(u)
\]
with coefficients $\psi_j(u) \in \CC(u)$,
such that $V$ is the space of solutions to the differential equation 
$D_V f(u) = 0$.
The coefficients of $D_V$ can be 
regarded as a coordinate system on $\scellnu$.  In the precise formulation (see \cref{thm:precise}), the point
$V_E\in\Wr^{-1}(g)$ is computed in these coordinates. In order to express $V_E$ in standard coordinates, we need to solve a differential equation, resulting in highly non-linear formulas.

The main result of this paper is \cref{thm:main} below, which is a new version of \cref{thm:vague}. Rather than using the fundamental differential operator coordinates, it computes $V_E\in\Wr^{-1}(g)$ in the \defn{\Plucker coordinates}, which are the $d\times d$ minors of a $d\times m$ matrix whose rows form a basis for $V_E$.  We introduce (by explicit formulas)
a new set of generators $\beta^\lambda$ for $\bethe$, 
which are indexed by partitions $\lambda$.
For any eigenspace 
$E \subseteq \spechtnu$, the corresponding eigenvalues of the $\beta^\lambda$'s
are the \Plucker coordinates of $V_E$.

There are three major advantages of this formulation.  First, we obtain a more direct description of $V_E$ which does not require solving a differential equation; the implicit part of our construction lies entirely in understanding the representation theory of $\Sn$. Second, many
natural objects of interest are given by \emph{linear} 
functions of the \Plucker 
coordinates. For example, we readily obtain explicit bases for $V_E$; the Wronskian and
the fundamental differential operator coordinates are given as 
linear functions of the \Plucker coordinates; and Schubert varieties and 
Schubert intersections are defined by linear equations in the \Plucker 
coordinates.
Third, basic properties of the operators $\beta^\lambda$ imply positivity results about the \Plucker coordinates of $V_E$. This 
enables us to resolve several conjectures in real algebraic geometry, as we explain in \cref{sec:conjectures}.

We mention that after a preliminary version of this paper appeared, John Harnad informed us of work of Alexandrov, Leurent, Tsuboi, and Zabrodin \cite{alexandrov_leurent_tsuboi_zabrodin14} which studies operators $T_\lambda$ acting on $(\CC^d)^{\otimes n}$, called \emph{higher Gaudin Hamiltonians}. The definition of $T_\lambda$ involves taking iterated matrix derivatives of $d\times d$ matrices, and at first glance appears quite different from the definition of $\beta^\lambda$. However, since $\beta^\lambda$ and $T_\lambda$ satisfy some of the same algebraic properties, it was natural to consider if they are related. The follow-up paper \cite{karp_mukhin_tarasov25} shows that, despite the disparate formulas, the operators $\beta^\lambda$ and $T_\lambda$ acting on $(\CC^d)^{\otimes n}$ are equal (under a certain choice of auxiliary parameters), which leads to a generalization of some of the main results of this paper from spaces of polynomials to spaces of \emph{quasi-exponentials} (polynomials rescaled by exponential functions). The two formulas each have their own advantages: our formula for $\beta^\lambda$ is concrete and combinatorial, and is useful for doing calculations, while the definition of $T_\lambda$ is natural from the perspective of integrable systems.
An important benefit of our formula for $\beta^\lambda$ is that it makes manifest that $\beta^\lambda$ is positive semidefinite (see \cref{prop:betapsd}), which is entirely opaque from the definition of $T_\lambda$. This positivity is key to the applications in \cref{sec:conjectures}, and was one of our main motivations for studying $\beta^\lambda$.
For further discussion on \cite{alexandrov_leurent_tsuboi_zabrodin14}, we refer to \cite[Section 8]{karp_mukhin_tarasov25}.


\subsection{Universal \Plucker coordinates}
\label{sec:universalintro}

We now state our main theorem.
For every partition $\lambda$, define 
\begin{equation}
\label{eq:betadef}
  \beta^\lambda(t) := \sum_{\substack{X \subseteq [n], \\ |X| = |\lambda|}} \,
   \sum_{\sigma \in \symgrp{X}}
     \chi^\lambda(\sigma) \sigma \prod_{i \in [n] \setminus X} (z_i + t)
\,.
\end{equation}
Here $[n] = \{1, \dots, n\}$, $\symgrp{X} \subseteq \Sn$ is the group of 
permutations of $X$, and
$\chi^\lambda : \symgrp{X} \to \CC$ is the character of the Specht module $\specht{\lambda}$.
We note that $\chi^\lambda$ is integer-valued, so $\beta^\lambda(t)$ is in fact defined over $\ZZ$.
Also, $\beta^\lambda(t)$ is nonzero if and only if $|\lambda| \leq n$.
Set $\beta^\lambda := \beta^\lambda(0)$.
\begin{example}\label{ex:betadef}
If $\lambda = (1,1)$, then $\chi^\lambda$ is the sign character on $\symgrp{2}$. When $n=3$, we get
\[
\beta^{11} = (\Sidentity{3} - \trans{1}{2})z_3 + (\Sidentity{3} - \trans{1}{3})z_2 + (\Sidentity{3} - \trans{2}{3})z_1\,,
\]
where $\Sidentity{3}$ denotes the identity element of $\symgrp{3}$, 
and $\trans{i}{j} := (i \;\; j)$ is the transposition swapping $i$ and $j$.
\end{example}

\begin{theorem}
\label{thm:main}
Let $z_1, \dots, z_n \in \CC$, and set $g(u) := (u+z_1) \dotsm (u+z_n) \in \CC[u]$.
The operators $\beta^\lambda(t) \in \CC[\Sn]$ satisfy the following 
algebraic identities:
\begin{enumerate}[(i)]
\item\label{main_commutativity} Commutativity relations:
\begin{equation}
\label{eq:commutativity}
      \beta^\lambda(s) \beta^\mu(t) = \beta^\mu(t) \beta^\lambda(s)
   \qquad \text{for all partitions $\lambda$ and $\mu$}\,.
\end{equation}
\item\label{main_translation} Translation identity:
\begin{equation}
\label{eq:translationidentity}
    \beta^\mu(s+t) = \sum_{\lambda \supseteq \mu} 
      \frac{\numsyt{\lambda/\mu}}{|\lambda/\mu|!} t^{|\lambda/\mu|} 
      \beta^\lambda(s) \qquad \text{for all partitions } \mu\,,
\end{equation}
where $\numsyt{\lambda/\mu}$ denotes the number of standard Young tableaux 
of shape $\lambda/\mu$.
\item\label{main_pluckers} The quadratic \Plucker relations \eqref{eq:prelspartitions}.
\end{enumerate}
Furthermore:
\begin{enumerate}[(i)] \setcounter{enumi}{3}
\item\label{main_generates} For every partition $\lambda$ and $t \in \CC$, 
      we have $\beta^\lambda(t) \in \bethe$.
      The set $\{\beta^\lambda \mid |\lambda| \leq n\}$ generates
      $\bethe$ as an algebra.
\item\label{main_eigenspace} If $E \subseteq \spechtnu$ is any eigenspace 
      of $\bethenu$, then the corresponding eigenvalues 
      of the operators $\beta^\lambda$ are the \Plucker coordinates of 
      a point $V_E \in \scellnu \subseteq \Gr(d,m)$ such that $\Wr(V_E) = g$.  
      Every point of $\Wr^{-1}(g)$ 
      corresponds in this way to some eigenspace $E \subseteq \spechtnu$ 
      of $\bethenu$.  
\item\label{main_multiplicity} The multiplicity of $V_E$ as a point of 
      $\Wr^{-1}(g)$ is equal to $\dim \widehat{E}$, where 
      $\widehat{E} \subseteq \spechtnu$ is the generalized eigenspace of
      $\bethenu$ containing $E$.
\end{enumerate}
\end{theorem}

We note that while the translation identity in part \ref{main_translation} is linear, 
parts \ref{main_commutativity} and \ref{main_pluckers} both involve quadratic expressions in 
$\bethe$,
making them intractable to prove directly.  In both of these cases we proceed by reducing the problem to --- and then proving --- an easier identity, using a diverse set of algebraic tools. For part \ref{main_commutativity}, we use properties of $\bethe$ and combinatorial ideas which appeared in \cite{purbhoo}. For part \ref{main_pluckers}, we employ the translation identity, properties of the exterior algebra, new combinatorial identities of symmetric functions, and the theory of $\tau$-functions of the KP hierarchy.  Once
identities \ref{main_commutativity}--\ref{main_pluckers}
are established, parts \ref{main_generates}--\ref{main_multiplicity} are relatively straightforward
consequences.

While the proof of \cref{thm:main} uses some of the same mathematical 
constructions and foundational results as the proof of \cref{thm:vague}/\ref{thm:precise} in \cite{mukhin_tarasov_varchenko13}, it does not use
the result itself.  In fact, our proof of \cref{thm:main} simultaneously proves \cref{thm:precise} (see \cref{rmk:newproof}).
The two arguments are interconnected, which plays a key role in
establishing the \Plucker relations.
\begin{example}
\label{ex:main}
We illustrate \cref{thm:main} in the case $n=2$, for the Grassmannian $\Gr(2,4)$. Writing $\symgrp{2} = \{\Sidentity{2}, \trans{1}{2}\}$, we have
\[
\beta^0 = \Sidentity{2} \, z_1z_2\,, \qquad \beta^1 = \Sidentity{2} (z_1 +  z_2)\,, \qquad \beta^2 = \Sidentity{2} + \trans{1}{2}\,, \qquad \beta^{11} = \Sidentity{2} - \trans{1}{2}\,,
\]
and $\beta^\lambda = 0$ for all other partitions $\lambda$. Note that the $\beta^\lambda$'s satisfy the equation
\[
- \beta^0\beta^{22} + \beta^1\beta^{21} - \beta^{11}\beta^2 = 0\,,
\]
which is the first non-trivial \Plucker relation \eqref{eq:firstplucker}.

There are two Specht modules for $\symgrp{2}$, namely $\specht{2}$ and $\specht{11}$, which are both $1$-dimensional. In $\specht{2}$, both $\Sidentity{2}$ and $\trans{1}{2}$ act with eigenvalue $1$, and so
\begin{equation}
\label{eq:examplepluckers}
\beta^0 \rightsquigarrow z_1z_2\,, \qquad \beta^1 \rightsquigarrow z_1 + z_2\,, \qquad \beta^2 \rightsquigarrow 2\,, \qquad \beta^{11} \rightsquigarrow 0\,.
\end{equation}
These are the \Plucker coordinates (under the identification \eqref{eq:identification}) of the element
\[
V = \left\langle 1,\, z_1z_2u + \frac{z_1 + z_2}{2}u^2 + \frac{1}{3}u^3 \right\rangle \in \scell{2} \subseteq \Gr(2,4)\,;
\]
see \cref{ex:pluckers} for further explanation. On the other hand, in $\specht{11}$, the element $\Sidentity{2}$ acts with eigenvalue $1$ and $\trans{1}{2}$ acts with eigenvalue $-1$, giving the solution
\[
V = \left\langle \frac{z_1 + z_2}{2} + u,\, -z_1z_2 + u^2 \right\rangle \in \scell{11} \subseteq \Gr(2,4)\,.
\]

We can check that both elements of $\Gr(2,4)$ have Wronskian $g(u) = (u+z_1)(u+z_2)$.
\end{example}

\begin{example}
\label{ex:22}
We illustrate parts \ref{main_commutativity} and \ref{main_pluckers} of \cref{thm:main} in the case $n=4$. Consider the $2$-dimensional representation $\spechtnu$ of $\symgrp{4}$, $\nu = (2,2)$. Following the conventions used by \texttt{Sage} \cite{sagemath}, the simple transpositions $\trans{1}{2}$ and $\trans{3}{4}$ both act as $(\begin{smallmatrix}1 & 0 \\ 1 & -1\end{smallmatrix})$, and $\trans{2}{3}$ acts as $(\begin{smallmatrix}0 & -1 \\ -1 & 0\end{smallmatrix})$. Let $\beta^\lambda_\nu \in \End(\spechtnu)$ denote the operator $\beta^\lambda$ acting on $\spechtnu$, which we regard as a $2\times 2$ matrix. Then
\begin{gather*}
\beta^0_\nu = z_1z_2z_3z_4\,\scalebox{0.85}{$\begin{pmatrix}1 & 0 \\ 0 & 1\end{pmatrix}$}\,,
\qquad
\beta^1_\nu = (z_1z_2z_3 + z_1z_2z_4 + z_1z_3z_4 + z_2z_3z_4)\,\scalebox{0.85}{$\begin{pmatrix}1 & 0 \\ 0 & 1\end{pmatrix}$}\,,
\\[2pt]
\beta^2_\nu = \scalebox{0.85}{$\begin{pmatrix}
2z_1z_2 + z_1z_4 + z_2z_3 + 2z_3z_4 & z_1z_3 - z_1z_4 - z_2z_3 + z_2z_4 \\
z_1z_2 - z_1z_4 - z_2z_3 + z_3z_4 & 2z_1z_3 + z_1z_4 + z_2z_3 + 2z_2z_4
\end{pmatrix}$}\,,
\\[2pt]
\beta^{11}_\nu = \scalebox{0.85}{$\begin{pmatrix}
2z_1z_3 + z_1z_4 + z_2z_3 + 2z_2z_4 & -z_1z_3 + z_1z_4 + z_2z_3 - z_2z_4 \\
-z_1z_2 + z_1z_4 + z_2z_3 - z_3z_4 & 2z_1z_2 + z_1z_4 + z_2z_3 + 2z_3z_4
\end{pmatrix}$}\,,
\\[2pt]
\beta^{21}_\nu = 3(z_1 + z_2 + z_3 + z_4)\,\scalebox{0.85}{$\begin{pmatrix}1 & 0 \\ 0 & 1\end{pmatrix}$}\,,
\qquad
\beta^{22}_\nu = 12\,\scalebox{0.85}{$\begin{pmatrix}1 & 0 \\ 0 & 1\end{pmatrix}$}\,,
\end{gather*}
and $\beta^\lambda_\nu = 0$ for all other partitions $\lambda$ 
(see \cref{prop:betapsd}\ref{betapsd4}). 
We can see that the $\beta^\lambda_\nu$'s pairwise commute and satisfy the \Plucker relation \eqref{eq:firstplucker}: $- \beta^0_\nu\beta^{22}_\nu + \beta^1_\nu\beta^{21}_\nu - \beta^{11}_\nu\beta^2_\nu = 0\,.$
\end{example}

We now briefly discuss several additional results related to \cref{thm:main}, before going into detail about applications to real algebraic geometry in \cref{sec:conjectures}.

\subsubsection{Scheme-theoretic results}
\label{sec:schemeintro}

The eigenspaces of $\bethenu$ can also be regarded as the points
of the spectrum of the algebra, $\Spec \bethenu$. 
\cref{thm:vague,thm:main} both set-theoretically identify $\Spec \bethenu$ 
with the fibre of the Wronski map $\Wr^{-1}(g) \subseteq \scellnu$.
We discuss a more precise scheme-theoretic version of this 
correspondence in \cref{sec:scheme}.

Namely, in the case where $z_1, \dots, z_n$ are distinct, 
Mukhin, Tarasov, and Varchenko \cite{mukhin_tarasov_varchenko13}
prove that $\Spec \bethenu$ and $\Wr^{-1}(g)$ are in fact 
isomorphic as schemes; equivalently, $\bethenu$
is isomorphic to the coordinate ring of $\Wr^{-1}(g)$.  
We reformulate this result as \cref{thm:distinct}, expressing the isomorphism 
in terms of \Plucker coordinates and
the operators $\beta^\lambda$.

If $z_1, \dots, z_n$ are not distinct, then $\Spec \bethenu$ and
$\Wr^{-1}(g)$ are not necessarily isomorphic as schemes.
Instead, $\Spec \bethenu$ is scheme-theoretically identified
with a particular union of Schubert intersections, involving
flags osculating a rational normal curve; the latter
is set-theoretically (but not always scheme-theoretically)
the same as $\Wr^{-1}(g)$. 
We formulate this scheme-theoretic 
isomorphism precisely as \cref{thm:nondistinct}, 
and prove it using results from 
\cite{mukhin_tarasov_varchenko09b}.

Together, \cref{thm:distinct,thm:nondistinct} give the precise
scheme-theoretic formulation of \cref{thm:main}\ref{main_eigenspace}.

\subsubsection{Bases for \texorpdfstring{$V$}{V} in a fibre of the Wronski map}
\label{sec:basesintro}

Using \cref{thm:main}, we obtain two explicit bases for any element $V \in \Wr^{-1}(g)$, in terms of our operators $\beta^\lambda(t)$ acting on the associated eigenspace $E$. The first basis (see \cref{thm:betabasis}) depends on the Schubert cell $\scellnu$ containing $V$, corresponding to a matrix representative in reduced row-echelon  form. The second basis (see \cref{thm:generalbasis}) is independent of the Schubert cell, and only involves the operators $\beta^k(t)$ associated to single-row partitions (i.e.\ where the corresponding character $\chi^k$ is trivial).

\subsubsection{Geometric transformations}
\label{sec:transformationsintro}

The Grassmannians $\Gr(d,m)$ and $\Gr(m-d,m)$ are dual to each other.
In \cref{sec:transformations}, we use \cref{thm:main} to show that this 
duality corresponds to an automorphism of the algebra $\bethe$
(see \cref{prop:dualityautomorphism}), 
which recovers a result from \cite{purbhoo}.  We also show that when
the partition $\nu$ is a rectangle, the 
Bethe algebra $\bethenu \subseteq \End(\spechtnu)$ is invariant 
under the action of $\PGL_2$ on the parameters $z_1, \dots, z_n$
(see \cref{cor:invariance}).  This also follows from \cref{thm:main},
and corresponds to the fact that there is a natural $\PGL_2$-action
on $\Gr(d,m)$.
This action characterizes the Wronski map \cite[Section 3.2]{gillespie_levinson_purbhoo23} and is also related to the combinatorics of promotion and evacuation on standard Young tableaux \cite{purbhoo13}.

\subsubsection{A \texorpdfstring{$\tau$}{tau}-function of the KP hierarchy}

The \emph{KP hierarchy} is a system of differential equations which arose out of the study of solitary waves. As we recall in \cref{sec:kpbackground}, its solutions are encoded by \emph{$\tau$-functions}. These are symmetric functions satisfying the Hirota equation \eqref{eq:KP}, or equivalently, functions whose coefficients in the Schur basis satisfy the \Plucker relations \eqref{eq:prelspartitions}. We can then equivalently rephrase \cref{thm:main}\ref{main_pluckers} as the statement that the symmetric function
\begin{equation}
\label{eq:tauintro}
     \sum_{X \subseteq [n]} \, \sum_{\sigma \in \symgrp{X}}
     \sigma \otimes \p_{\mu_1}\dotsm \,\p_{\mu_s}
     \cdot\prod_{i \in [n] \setminus X}z_i\, \in \CC[\Sn] \otimes \Lambda
\end{equation}
is a $\tau$-function of the KP hierarchy, where $\mu_1, \dots, \mu_s$ denote the lengths of the cycles of $\sigma \in \symgrp{X}$, $\Lambda$ denotes the $\CC$-algebra of symmetric functions, and $\p_k\in\Lambda$ is the $k$th power sum. (See \cref{thm:taufunction,thm:tauinfinity} for precise statements.) As we mentioned above, our proof of \cref{thm:main}\ref{main_pluckers} in fact uses the theory of symmetric functions and $\tau$-functions.


\subsection{Conjectures in real algebraic geometry}
\label{sec:conjectures}

We now discuss several applications of our results in real algebraic geometry. We encourage the reader to consult \cref{fig:implications} to keep track of the implications between the various main results discussed here.

We continue to assume that $\nu$ is a partition of 
$n$ with at most $d$ parts, and $m \geq d+\nu_1$, so that the
Schubert cell $\scellnu$ is contained in $\Gr(d,m)$.
The Schubert variety $\svarnu \subseteq \Gr(d,m)$ is
the closure of $\scellnu$.
We write $\rectangle$ for the rectangular partition $(m-d)^d = (m-d, \dots, m-d)$.
In this case, $\svarrect = \Gr(d,m)$.

We will be mainly concerned with the following Schubert problem. 
Given $W_1, \dots, W_n$ in $\Gr(m-d,m)$, determine
all $d$-planes $V$ such that
\begin{equation}
\label{eq:schubertproblem}
V \in \svarnu \qquad \text{and} \qquad
  V \cap W_i \neq \{0\}\; \text{ for all $i =1, \dots, n$}
\,.
\end{equation}
When $W_1, \dots, W_n$ are sufficiently general, the number of distinct
solutions $V$ to the Schubert problem \eqref{eq:schubertproblem} is exactly $\numsyt{\nu} = \dim \spechtnu$.

We will be concerned with solving \eqref{eq:schubertproblem} over the real numbers when $W_1, \dots, W_n$ are real, and especially with instances for which \emph{all} the solutions are real. The interest in algebraic problems with only real solutions dates back at least to Fulton \cite[Section 7.2]{fulton84}, who wrote, ``The question of how many solutions of real equations can be real is still very much open, particularly for enumerative problems.'' Note that the property of having only real solutions is extremely rare; for example, for a `random' Schubert problem on $\Gr(d,m)$ defined over $\RR$, the number of real solutions is roughly the square root of the number of complex solutions \cite{burgisser_lerario20}. We refer to \cite{sottile11} for a detailed survey of real enumerative geometry.
\begin{figure}[t]
\hidelinks
\begin{center}
\begin{tikzpicture}[baseline=(current bounding box.center),implies/.style={double,double equal sign distance,-implies},equiv/.style={double,double equal sign distance,implies-implies}]
\tikzstyle{block}=[rectangle,draw,text width=48mm,align=flush center,rounded corners]
\tikzstyle{out1}=[inner sep=0,minimum size=1.2mm,circle,draw=black,fill=black]
\tikzstyle{in1}=[inner sep=0,minimum size=1.2mm,circle,draw=black,fill=white]
\pgfmathsetmacro{\unit}{0.922};
\pgfmathsetmacro{\r}{7.5};
\pgfmathsetmacro{\s}{3};
\pgfmathsetmacro{\t}{4.6};
\node[block](main)at(0,0){{\bfseries\fullcref{thm:main}}\\{\footnotesize Algebraic properties of $\beta^\lambda$}};
\node[block](positiveSS)at(0,-\s){{\bfseries\fullcref{thm:positive}}\\[2pt]{\footnotesize Positive Shapiro--Shapiro\\ conjecture}};
\node[block](positivesecant)at(0,-\s-\t){{\bfseries\fullcref{thm:positivesecant}}\\[2pt]{\footnotesize Positive secant conjecture\\ (divisor form)}};
\node[block](SS)at(\r,-\s){{\bfseries\fullcref{thm:ssc}}\\[2pt]{\footnotesize Shapiro--Shapiro\\ conjecture \cite{mukhin_tarasov_varchenko09a}}};
\node[block](secant)at(\r,-\s-\t){{\bfseries\fullcref{thm:secant}}\\[2pt]{\footnotesize Secant conjecture\\ (divisor form)}};
\node[block](disconjugacy)at(\r,-\s-0.42*\t){{\bfseries\fullcref{thm:disconj}}\\{\footnotesize Disconjugacy conjecture}};
\draw(main)edge[implies] node[right=1pt,align=left]{{\footnotesize positivity properties of $\beta^\lambda$}\\{\footnotesize (\fullcref{prop:betapsd})}}(positiveSS);
\draw(positiveSS)edge[implies] node[above=1pt]{{\footnotesize $\PGL_2$-action}}(SS);
\draw(positivesecant)edge[implies] node[above=1pt]{{\footnotesize $\PGL_2$-action}}(secant);
\draw(disconjugacy)edge[implies] node[right=1pt,align=left]{{\footnotesize \cite{eremenko15}, using}\\{\footnotesize \fullcref{thm:ssc}}}(secant);
\draw(positiveSS)edge[equiv] node[right=1pt]{{\footnotesize \cite{karp}}}(positivesecant);
\draw(positiveSS.south east)edge[equiv] node[below=4pt]{{\footnotesize \cite{karp}}}(disconjugacy.north west);
\end{tikzpicture}
\caption{Implications between various results of the paper.}\label{fig:implications}
\end{center}
\restorelinks
\end{figure}

\subsubsection{The Shapiro--Shapiro conjecture}
\label{sec:shapiroshapiro}
The \defn{moment curve}
$\gamma : \CC \to \CC_{m-1}[u]$ is the parametric curve
\begin{equation}
\label{eq:moment}
   \gamma(t) := \frac{(u+t)^{m-1}}{(m-1)!}
\,.
\end{equation}
The closure of the image of $\gamma$ in $\PP^{m-1}$ is 
a rational normal curve.
A $d$-plane $V \in \Gr(d,m)$ \defn{osculates} $\gamma$ at 
$w \in \CC$ if 
$(\gamma(w), \gamma'(w), \gamma''(w), \dots, \gamma^{(d-1)}(w))$
is a basis for $V$. 
Osculating planes to the moment curve are related to the Wronski map 
by the following fact (see \cref{prop:schubertwronskian2} for a 
more detailed formulation):
\begin{proposition}
\label{prop:schubertwronskian}
Suppose $W \in \Gr(m-d,m)$ and $V \in \Gr(d,m)$.
If $W$ osculates $\gamma$ at $w$, then $V \cap W \neq \{0\}$ 
if and only if $-w$ is a zero of $\Wr(V)$.
\end{proposition}

The \defn{Shapiro--Shapiro conjecture} can be stated as follows:
\begin{theorem}[Mukhin, Tarasov, and Varchenko \cite{mukhin_tarasov_varchenko09a}]
\label{thm:ssc}
Let $z_1, \dots, z_n$ be distinct real numbers.
For $i =1, \dots, n$, let $W_i \in \Gr(m-d,m)$ be the osculating 
$(m-d)$-plane to $\gamma$ at $z_i$.
Then there are exactly $\numsyt{\nu}$ distinct solutions to the Schubert
problem \eqref{eq:schubertproblem}, and all solutions are real.
\end{theorem}

\cref{thm:ssc} was conjectured by Boris and Michael Shapiro in 1993, and extensively tested and popularized by Sottile \cite{sottile00}. It was proved in the cases $d \le 2$ and $m-d \le 2$ by Eremenko and Gabrielov \cite{eremenko_gabrielov02}, and in general by Mukhin, Tarasov, and Varchenko \cite{mukhin_tarasov_varchenko09a}. Their proof was later restructured and simplified in \cite{purbhoo}.  A very different proof, based on geometric and topological arguments, is given in \cite{levinson_purbhoo21}.

A notable feature of \cref{thm:ssc} is that the Schubert problem
has $\numsyt{\nu}$ distinct solutions, despite the fact that $W_1, \dots, W_n$ are
explicitly specified, and hence not assumed to be general.  
In fact, this follows from the claim that all 
solutions are real (see \cite[Theorem 13.2]{sottile11}).
A more general form of \cref{thm:ssc} describes how the story changes 
in the limit as $z_1, \dots, z_n$ become non-distinct,
as we discuss in \cref{sec:dimensions,sec:generalssc}.

For the situation described in \cref{thm:ssc}, the 
Schubert problem~\eqref{eq:schubertproblem} is
equivalent (by 
\cref{prop:schubertwronskian}/\ref{prop:schubertwronskian2}) to $V\in\scellnu$
and $\Wr(V) = g$, where $g(u) = (u+z_1) \dotsm (u+z_n)$. 
Thus, \cref{thm:ssc} can be rephrased as follows:
if $V \subseteq \CC[u]$ is a finite-dimensional vector space of polynomials,
and $\Wr(V)$ has only real roots, then $V$ is real. 
Mukhin, Tarasov and Varchenko deduce \cref{thm:ssc} from 
\cref{thm:vague} using this reformulation.  
The key point is that when $z_1, \dots, z_n \in \RR$,
the Gaudin Hamiltonians are self-adjoint operators with respect
to a Hermitian inner product.  Hence their eigenvalues,
which determine the points of $\Wr^{-1}(g)$, are all real.

Using \cref{thm:main}, we obtain a number of generalizations of \cref{thm:ssc}:

\subsubsection{The divisor form of the secant conjecture}
\label{sec:secantconjecture}
Let $I \subseteq \RR$ be an interval.  An $(m-d)$-plane $W \in \Gr(m-d,m)$ 
is a \defn{secant} to $\gamma$ along $I$ if there exist distinct points 
$w_1, \dots, w_{m-d} \in I$ such that $(\gamma(w_1), \dots, \gamma(w_{m-d}))$
is a basis for $W$.
More generally, $W$ is a \defn{generalized secant} to $\gamma$ along $I$ 
if there
exist distinct points $w_1, \dots, w_k \in I$ and positive integers 
$m_1, \dots, m_k$, such that $m_1+ \dots + m_k = m-d$ and
\[
   \big(\gamma(w_1), \gamma'(w_1), \dots, \gamma^{(m_1-1)}(w_1),
    \ \dots\ , \gamma(w_k), \gamma'(w_k), \dots, \gamma^{(m_k-1)}(w_k)\big)
\]
is a basis for $W$.  Working projectively, these definitions naturally
extend to cyclic intervals of $\RP^1 = \RR \cup \{\infty\}$, 
where the interval is allowed to wrap around infinity.  When
one of the points $w_i\in I$ is $\infty$, 
$\gamma^{(j)}(w_i)$ is replaced by $u^j$.

Around 2003, Frank Sottile formulated the \defn{secant conjecture}, which asserts in particular that \cref{thm:ssc} remains true when $W_1, \dots, W_n$ are generalized secants to $\gamma$ along disjoint intervals of $\RR$. This statement is what we call the \defn{divisor form} of the secant conjecture, since it arises from intersecting Schubert varieties of codimension one, i.e., \emph{Schubert divisors}; the general form of the secant conjecture involves intersecting Schubert varieties of arbitrary codimension, as we discuss in \cref{sec:generalsecant}. Note that this case of the secant conjecture is a generalization of the Shapiro--Shapiro conjecture, since an osculating plane to $\gamma$ is a special case of a generalized secant.

The secant conjecture appeared in \cite{ruffo_sivan_soprunova_sottile06} (cf.\ \cite[Section 13.4]{sottile11}), and it was extensively tested experimentally in a project led by Sottile \cite{garcia-puente_hein_hillar_martin_del_campo_ruffo_sottile_teitler12}, as described in \cite{hillar_garcia-puente_martin_del_campo_ruffo_teitler_johnson_sottile10}. It has also been proved in special cases: Eremenko, Gabrielov, Shapiro, and Vainshtein \cite[Section 3]{eremenko_gabrielov_shapiro_vainshtein06} established the case $m-d\le 2$; and Mukhin, Tarasov, and Varchenko \cite{mukhin_tarasov_varchenko09c} (cf.\ \cite[Section 3.1]{garcia-puente_hein_hillar_martin_del_campo_ruffo_sottile_teitler12}) verified the case of the divisor form when there exists $r > 0$ such that every $W_i$ is a (non-generalized) secant where $w_1, \dots, w_{m-d} \in I_i$ are an arithmetic progression of step size $r$.

We show that the divisor form of the secant conjecture is true in general:
\begin{theorem}[Secant conjecture, divisor form]
\label{thm:secant}
Let $I_1, \dots, I_n  \subseteq \RR$ be pairwise disjoint real intervals.
For $i =1, \dots, n$, let $W_i \in \Gr(m-d,m)$ be a generalized secant
to $\gamma$ along $I_i$.
Then there are exactly $\numsyt{\nu}$ distinct solutions to the Schubert
problem \eqref{eq:schubertproblem}, and all solutions are real.
\end{theorem}

This verifies the secant conjecture in the first non-trivial case of interest for a Schubert problem on an arbitrary Grassmannian. As we discuss in \cref{sec:generalsecant}, we do not yet know how to address the general form of the secant conjecture with our methods.

We mention that when $\nu = \rectangle$ (and $n = d(m-d)$), \cref{thm:secant} remains true 
if $I_1, \dots, I_n$ are cyclic intervals of
$\RP^1$, i.e., one of the intervals is allowed to wrap
around infinity.  The secant conjecture is sometimes phrased in this way.
This follows from \cref{thm:secant} as stated, using the $\PGL_2$-action on $\Gr(d,m)$.

\subsubsection{The disconjugacy conjecture}

Suppose that $V$ is a $d$-dimensional vector space of real analytic functions, 
defined on an interval $I\subseteq\RR$.  Disconjugacy is concerned with 
the question of how many zeros a function in $V$ can have.
By linear algebra, there always exists a nonzero function 
$f \in V$ such that $f$ has at least $d-1$ zeros on $I$.  We say that $V$ is \defn{disconjugate}
on $I$ if every nonzero function in $V$ has at most $d-1$ 
zeros on $I$ (counted with multiplicities). Disconjugacy has long been studied because it is related to explicit solutions for linear differential equations; see \cite{coppel71}, as well as \cite[Section 4.1]{karp} and the references therein.

It is not always straightforward to decide if $V$ is disconjugate on $I$.
However, a necessary condition is that $\Wr(V)$ has no zeros on $I$.
This is because $\Wr(V)$ has a zero at $w$ if and only if
there exists a nonzero $f \in V$ such that $f$ has a zero at $w$ of 
multiplicity at least $d$.  In general, the converse is false; for example, $V = \langle \cos u, \sin u \rangle$ is not disconjugate on $I = \RR$, and $\Wr(V) = 1$. Eremenko \cite{eremenko15,eremenko19} conjectured that the converse statement is actually correct
under very special circumstances.  This is known as the 
\defn{disconjugacy conjecture}, which we state now as a theorem:

\begin{theorem}[Disconjugacy conjecture]
\label{thm:disconj}
Let $V \subseteq \RR[u]$ be a finite-dimensional vector space of polynomials such that
$\Wr(V)$ has only real zeros.  Then $V$ is disconjugate on
every interval which avoids the zeros of $\Wr(V)$.
\end{theorem}

The disconjugacy conjecture was previously verified in the case that $\dim V \le 2$ \cite{eremenko_gabrielov_shapiro_vainshtein06} (cf.\ \cite[p.\ 341]{eremenko15}). Eremenko furthermore showed that the disconjugacy conjecture (along with the Shapiro--Shapiro conjecture) implies \cref{thm:secant}; in fact, his motivation was to generalize the argument used to prove the $m-d \le 2$ case of the secant conjecture \cite[Section 3]{eremenko_gabrielov_shapiro_vainshtein06}.
The main idea is encapsulated in \cref{lem:topological}, and explained
further in \cref{sec:conjectureproofs}.

\subsubsection{Positivity conjectures}\label{sec:positivity_conjectures}

A $d$-plane $V \in \Gr(d,m)$ is called \defn{totally nonnegative} if all of its Pl\"ucker coordinates are real and nonnegative (up to rescaling). Similarly, $V$ is called \defn{totally positive in $\scellnu$} if $V\in\scellnu$ and all of its \Plucker coordinates which are not trivially zero on $\scellnu$ are positive, i.e.,
\begin{equation}
\label{eq:positivenupluckers}
\Delta^\lambda > 0 \;\text{ for all } \lambda\subseteq\nu \qquad \text{and} \qquad \Delta^\lambda = 0 \;\text{ for all } \lambda\not\subseteq\nu\,.
\end{equation}
For example, each element $V\in\Gr(2,4)$ from \cref{ex:main} is totally nonnegative if and only if $z_1, z_2 \ge 0$, and is totally positive in its Schubert cell if and only if $z_1, z_2 > 0$.

The totally nonnegative part of $\Gr(d,m)$ is a totally nonnegative partial flag variety in the sense of Lusztig \cite{lusztig94,lusztig98} (see \cite[Section 1]{bloch_karp23} for further discussion), and was studied combinatorially by Postnikov \cite{postnikov06}. Total positivity in Schubert cells was considered by Berenstein and Zelevinsky \cite{berenstein_zelevinsky97}. These and similar totally positive spaces have been extensively studied in the past few decades, with connections to representation theory \cite{lusztig94}, combinatorics \cite{postnikov06}, cluster algebras \cite{fomin_williams_zelevinsky}, soliton solutions to the KP equation \cite{kodama_williams14}, scattering amplitudes \cite{arkani-hamed_bourjaily_cachazo_goncharov_postnikov_trnka16}, positive geometries \cite{arkani-hamed_bai_lam17}, Schubert calculus \cite{knutson14}, topology \cite{galashin_karp_lam22}, and many other topics. Total positivity also provided one of the original motivations for the Shapiro--Shapiro conjecture, since the moment curve $\gamma$ is an example of a \emph{totally positive} (or \emph{convex}) \emph{curve}; see \cref{sec:totalreality} and cf.\ \cite[Section 4]{sottile00}.

Mukhin--Tarasov and Karp conjectured that the reality statements discussed in \cref{sec:shapiroshapiro,sec:secantconjecture} have totally positive analogues. We verify these conjectures in slightly greater generality:
\begin{theorem}[Positive Shapiro--Shapiro conjecture]
\label{thm:positive}
Let $z_1, \dots, z_n$ and $W_1, \dots, W_n$ be as in \cref{thm:ssc}.
\begin{enumerate}[(i)]
\item\label{positive1} If $z_1, \dots, z_n \in [0,\infty)$, then all solutions to the Schubert
problem \eqref{eq:schubertproblem} are real and totally nonnegative.

\item\label{positive2} If $z_1, \dots, z_n \in (0,\infty)$, then
all solutions to the Schubert problem \eqref{eq:schubertproblem} 
are real and totally positive in $\scellnu$.
\end{enumerate}
\end{theorem}

\begin{theorem}[Positive secant conjecture, divisor form]
\label{thm:positivesecant}
Let $I_1, \dots, I_n$ and $W_1, \dots, W_n$ be as in \cref{thm:secant}.
\begin{enumerate}[(i)]
\item\label{positivesecant1} If $I_1, \dots, I_n  \subseteq [0,\infty)$, 
then there are exactly $\numsyt{\nu}$ distinct solutions to the Schubert
problem \eqref{eq:schubertproblem}, and all solutions are real and totally nonnegative.

\item\label{positivesecant2} If $I_1, \dots, I_n \subseteq (0,\infty)$, 
then there are exactly $\numsyt{\nu}$ distinct solutions to the Schubert
problem \eqref{eq:schubertproblem}, and all solutions are real and totally positive in $\scellnu$.
\end{enumerate}

\end{theorem}

In the special case $\nu = \rectangle$, \cref{thm:positive}\ref{positive1} was conjectured by Evgeny Mukhin and Vitaly Tarasov in 2017, and \cref{thm:positive,thm:positivesecant} were conjectured independently in \cite{karp}. It was shown in \cite{karp} that the four statements in \cref{thm:positive,thm:positivesecant} in the case $\nu=\rectangle$ are all pairwise equivalent, and that they are moreover equivalent to the disconjugacy conjecture (\cref{thm:disconj}).

We briefly recall from \cite{karp} why \cref{thm:positive}\ref{positive1} implies the disconjugacy conjecture (the converse is much more subtle, but we do not need it here). Let $V\in\Gr(d,m)$ be such that $\Wr(V)$ has only real zeros, and let $I \subseteq \RR$ be an interval which avoids the zeros of $\Wr(V)$, which we may assume is closed. We apply the $\PGL_2$-action so that $I\subseteq (0,\infty)$ and the zeros of $\Wr(V)$ are all negative. Then by \cref{thm:positive}\ref{positive1}, $V$ is totally nonnegative. Equivalently, by a classical result of Gantmakher and Krein \cite[Theorem V.3]{gantmaher_krein50}, the sequence of coefficients of every $f\in V$ changes sign at most $d-1$ times. By Descartes's rule of signs, $f$ has at most $d-1$ zeros on $(0,\infty)$, as required.

As we have mentioned, the disconjugacy conjecture in turn implies the divisor form of the secant conjecture (\cref{thm:secant}). Therefore to prove all of these statements, it suffices to establish \cref{thm:positive,thm:positivesecant}. We now explain how to do so.

\subsubsection{Proof of conjectures}
\label{sec:conjectureproofs}

We give the proofs of \cref{thm:positive,thm:positivesecant}. We begin with the former, which is a direct corollary of our main result (\cref{thm:main}). We need the following properties of the operators $\beta^\lambda$, which we will prove in \cref{sec:repthy}. They are straightforward consequences of the definitions and some
well-known results in representation theory.

Recall that $\spechtnu$ can be equipped with a Hermitian inner product,
such that every $\sigma \in \Sn$ acts as a unitary operator.

\begin{proposition}
\label{prop:betapsd}
Let $\beta^\lambda_\nu \in \End(\spechtnu)$ denote the operator 
$\beta^\lambda$ acting on $\spechtnu$.  
\begin{enumerate}[(i)]
\item\label{betapsd1}
If $z_1, \dots, z_n \in \RR$, then $\beta^\lambda_\nu$ is a self-adjoint operator.
\item\label{betapsd2}
If $z_1, \dots, z_n \in [0,\infty)$, then $\beta^\lambda_\nu$ is positive semidefinite
for all $\lambda, \nu$.
\item\label{betapsd3}
If $z_1, \dots, z_n \in (0,\infty)$ and $\lambda \subseteq \nu$, then 
$\beta^\lambda_\nu$ is positive definite.
\item\label{betapsd4}
If $\lambda \not \subseteq \nu$, then $\beta^\lambda_\nu = 0$.
\end{enumerate}
\end{proposition}

\begin{remark}
We point out that \cref{prop:betapsd}\ref{betapsd1} is consistent with \cref{ex:22}, despite the fact that $\beta^2_\nu$ and $\beta^{11}_\nu$ are not Hermitian matrices (for $z_1, \dots, z_4 \in \RR$). This is because in \cref{ex:22}, we are not working with an orthonormal basis of $\spechtnu$. We can change to such a basis in which $\trans{1}{2}$ and $\trans{3}{4}$ act as the unitary matrix $(\begin{smallmatrix} 1 & 0 \\ 0 & -1\end{smallmatrix})$, and $\trans{2}{3}$ acts as $\frac{1}{2}(\begin{smallmatrix}-1 & \sqrt{3} \\ \sqrt{3} & 1\end{smallmatrix})$; then every $\beta^\lambda_\nu$ is Hermitian.
\end{remark}

\begin{proof}[Proof of \cref{thm:positive}]
Let $V$ be a solution to the Schubert problem \eqref{eq:schubertproblem}.
Equivalently, by \cref{prop:schubertwronskian}, we have
$V \in \scellnu$ and $\Wr(V) = g$.  
By \cref{thm:main}\ref{main_eigenspace}, we can write 
$V = V_E$ for some eigenspace $E \subseteq \spechtnu$ of $\bethenu$.
This means that the \Plucker coordinates 
$[\Delta^\lambda : \lambda \subseteq \rectangle]$ of $V$ 
are the eigenvalues of the operators $\beta^\lambda_\nu$ 
on $E$. If $z_1, \dots, z_n \in [0,\infty)$, then \cref{prop:betapsd}\ref{betapsd2} implies that the eigenvalues of $\beta^\lambda_\nu$ are real and nonnegative, so $V$ is totally nonnegative. This proves part \ref{positive1}. Similarly, if $z_1, \dots, z_n\in (0,\infty)$, then parts \ref{betapsd3} and \ref{betapsd4} of \cref{prop:betapsd} imply that \eqref{eq:positivenupluckers} holds, so $V$ is totally positive in $\scellnu$. This proves part \ref{positive2}.
\end{proof}

\cref{thm:positivesecant} now follows from topological arguments used in \cite{eremenko15,karp}, which we apply in the following form:
\begin{lemma}
\label{lem:topological}
Let $I_1, \dots, I_n$ and $W_1, \dots, W_n$ be as in \cref{thm:secant}, and suppose that the disconjugacy conjecture is true. Then there are exactly $\numsyt{\nu}$ distinct solutions to the Schubert problem \eqref{eq:schubertproblem}, and all solutions are real. Moreover, for each solution $V$, we can write $\Wr(V) = (u+z_1) \dotsm (u+z_n)$ for some real numbers $z_1\in I_1, \dots, z_n\in I_n$.
\end{lemma}

\begin{proof}
In the case that $\nu = \rectangle$, this is precisely \cite[Lemma 4.15]{karp}. In fact, the proof of \cite[Lemma 4.15]{karp} applies to an arbitrary $\nu \subseteq \rectangle$, using \cref{thm:ssc}.
\end{proof}

\begin{proof}[Proof of \cref{thm:positivesecant}]
We have proved \cref{thm:positive}, so the disconjugacy conjecture is true. Hence we can apply \cref{lem:topological}, which along with \cref{thm:positive} yields the result.
\end{proof}


\subsection{Outline}

The remainder of this paper is organized as follows.
In \cref{sec:background}, we recall background required for the
proof of \cref{thm:main}.  Our discussion spans several topics, 
including: \Plucker coordinates and the \Plucker relations; the Wronski map,
and its relationship to Schubert varieties, 
the $\PGL_2$-action on $\Gr(d,m)$,
and differential operators; 
applications of symmetric function theory, 
including representation theory of symmetric groups, 
the proof of \cref{prop:betapsd}, 
$\tau$-functions of the KP hierarchy, and several
symmetric function identities and their proofs;
and Bethe subalgebras of $\CC[\Sn]$, including the definition
of $\bethe$ and the precise statement
of \cref{thm:vague} (\cref{thm:precise}).

\cref{sec:commutativitytranslation,sec:pr} are devoted to the
proof of \cref{thm:main}, which we structure as follows.  In
\cref{sec:commutativitytranslation} we establish some basic properties
of the operators $\beta^\lambda(t)$.  We prove part 
\ref{main_translation} (the translation identity), followed by part \ref{main_commutativity} (the commutativity relations).
Combining these parts and some of the arguments involved in their proofs, 
we also establish \cref{lem:comparealgebras},
which is related to --- but slightly weaker than ---
part \ref{main_generates}.  
We use all of these basic properties, 
in \cref{sec:pr}, to prove the remaining parts of \cref{thm:main}.
\cref{sec:part1proof,sec:part2} contain the proof of part \ref{main_pluckers} (the \Plucker relations),
which is the most technical part
of the proof of \cref{thm:main}.
We then deduce parts \ref{main_eigenspace} and \ref{main_multiplicity},
which we use to establish part \ref{main_generates}, in 
\cref{sec:finalsteps}.

Finally, in \cref{sec:discussion}, we discuss several results related
to \cref{thm:main} and its consequences in real algebraic geometry,
as well as a variety of open problems.  We give the precise
scheme-theoretic version of \cref{thm:main}\ref{main_eigenspace},
as discussed in \cref{sec:schemeintro};
in particular, this yields a general formula for the dimension of 
the Bethe algebras $\bethe$ and $\bethenu$.   
As discussed in \cref{sec:basesintro}, we use \cref{thm:main} to
exhibit two different bases of the solutions $V_E \in \Wr^{-1}(g)$
to the inverse Wronski problem.
We explain how Grassmann duality and the $\PGL_2$-action on
$\Gr(d,m)$ are reflected in the structure of $\bethe$ and $\bethenu$.
We discuss the combinatorial meaning of the commutativity relations 
\eqref{eq:commutativity}, and an extension of the $\tau$-function of
the KP hierarchy \eqref{eq:tauintro} to the infinite symmetric group
$\symgrp{\infty}$.  Finally, we discuss open problems and longstanding
conjectures relating to \cref{thm:ssc,thm:secant,thm:positive},
including the general form of the secant conjecture and 
the total reality conjecture for convex curves.

\paragraph*{Acknowledgements.}
We thank Evgeny Mukhin and Frank Sottile for helpful discussions, John Harnad for informing us of the paper \cite{alexandrov_leurent_tsuboi_zabrodin14}, David Speyer for providing \cref{ex:zerobasis} and allowing us to include it in our paper, and the reviewers for constructive feedback (including suggesting the addition of \cref{fig:implications} and \cref{thm:comparemultiplicity}).  Calculations for
this project were carried out using \texttt{Sage} \cite{sagemath}, including
verification of \cref{thm:main}\ref{main_pluckers} up to $n=7$; 
we thank Mike Zabrocki for assistance with some parts of the code.


\section{Background}
\label{sec:background}

We recall some background on \Plucker coordinates \cite[Chapter 9]{fulton97}, Schubert varieties \cite[Chapter 9]{fulton97}, Wronskians \cite{karp,purbhoo10,sottile11}, symmetric functions \cite[Chapter 7]{stanley24}, representation theory \cite{sagan01,serre98}, and Bethe algebras \cite{mukhin_tarasov_varchenko13,purbhoo}. See the listed references for further details.


\subsection{\Plucker coordinates}
\label{sec:plucker}

For a $d$-plane $V \in \Gr(d,m)$, we can represent $V$ as the row
space of a $d \times m$ matrix $A$, which is unique up to left multiplication by $\GL_d$. Recall that $[m] = \{1, 2, \dots, m\}$, and define $\binom{[m]}{d}$ to be the set of $d$-element subsets of $[m]$.  We also write such subsets as tuples $(i_1, \dots, i_d)$, with
$1 \leq i_1 < \dots <i_d \leq m$.  For
each $I \in {[m] \choose d}$, let $\Delta_I$
denote the $d \times d$ minor of $A$ with column set $I$.
The projective coordinates 
$\big[\Delta_I : I \in {[m] \choose d}\big]$
are (up to a scalar multiple) independent of the choice of matrix $A$, 
and are called
the \defn{\Plucker coordinates} of $V$.

\begin{example}
\label{ex:pluckers}
Let $V := \left\langle 1,\, z_1z_2u + \frac{z_1 + z_2}{2}u^2 + \frac{1}{3}u^3 \right\rangle \in \Gr(2,4)$, as in \cref{ex:main}. Recalling the isomorphism \eqref{eq:isomorphism}, we can represent $V$ by the $2\times 4$ matrix
\[
A := \begin{pmatrix}
1 & 0 & 0 & 0 \\
0 & z_1z_2 & z_1+z_2 & 2
\end{pmatrix}\,.
\]
The \Plucker coordinates of $V$ are the $2\times 2$ minors of $A$:
\[
\Delta_{(1,2)} = z_1z_2, \qquad \Delta_{(1,3)} = z_1+z_2, \qquad \Delta_{(1,4)} = 2, \qquad \Delta_{(2,3)} = \Delta_{(2,4)} = \Delta_{(3,4)} = 0\,,
\]
in agreement with \cref{ex:main}. (In general, we can construct a matrix $A$ from the \Plucker coordinates of $V$ using \cref{prop:pluckerbasis}.)
\end{example}

\subsubsection{\Plucker relations}

The \Plucker coordinates define
an embedding $\delta : \Gr(d,m) \hookrightarrow \PP^{{m \choose d} -1}$
of the Grassmannian into projective space.
To describe the image of $\delta$,
it is useful to extend the indexing set 
for $\Delta_I$ from ${[m] \choose d}$
to $[m]^d$ by the alternating property.  That is, if $i_1 < \dots < i_d$,
put 
\[ \Delta_{(i_{\sigma(1)}, \dots, i_{\sigma(d)})} 
:= \sgn(\sigma) \Delta_{(i_1, \dots, i_d)} \qquad \text{for all } \sigma \in \symgrp{d};
\]
if $j_1, \dots, j_d$ are
not distinct, put $\Delta_{(j_1, \dots, j_d)} := 0$.
Thus for every $I \in [m]^d$, $\Delta_I$ is either zero, or plus or minus
some \Plucker coordinate.

If $I = (i_1, \dots, i_{d+1}) \in [m]^{d+1}$, 
$J = (j_1, \dots, j_{d-1}) \in [m]^{d-1}$, and $k \in [m]$,
write 
\begin{equation}
\label{eq:deletesubscript}
    \Delta_{I-k} 
     := \begin{cases}
        (-1)^{d+1-s}\Delta_{(i_1, \dots, \widehat{i_s},  \dots i_{d+1})},
       &\quad \text{if $k = i_s$ for a unique $s \in [d+1]$;}
     \\
       0,
       &\quad \text{otherwise}
     \end{cases}
\end{equation}
and
\begin{equation}
\label{eq:addsubscript}
\Delta_{J+k} := \Delta_{(j_1, \dots, j_{d-1},k)}
\,.
\end{equation}
The \defn{\Plucker relations} for $\Gr(d,m)$ \cite[Section 9.1]{fulton97} are the equations
\begin{equation}
\label{eq:prels}
     \sum_{k = 1}^m \Delta_{I-k} \Delta_{J+k} = 0
\qquad \text{for $I \in [m]^{d+1}$ and $J \in [m]^{d-1}$}\,.
\end{equation}
The equations \eqref{eq:prels} define the image of $\Gr(d,m)$ in
$\PP^{{m \choose d} -1}$ under the embedding $\delta$, as a scheme.
In particular, for every $V \in \Gr(d,m)$, the \Plucker coordinates
of $V$ satisfy the equations \eqref{eq:prels}.

\begin{example}
Taking $I = (1,2,3)$ and $J = (4)$, we obtain the unique non-trivial
\Plucker relation for $\Gr(2,4)$:
\begin{equation}
\label{eq:gr24plucker}
    - \Delta_{(1,2)}\Delta_{(3,4)}
   + \Delta_{(1,3)}\Delta_{(2,4)}
   - \Delta_{(2,3)}\Delta_{(1,4)} = 0
\,.
\end{equation}
Other choices for $I,J$ give the same equation (up to sign), 
or the trivial equation $0=0$.  
\end{example}

The preceding facts can be reformulated in terms of the exterior algebra
of $\CC^m$.  Let $(e_1, \dots, e_m)$ denote the standard basis for
$\CC^m$, and for $I = (i_1, \dots, i_d) \in [m]^d$, 
write $e_I := e_{i_1} \wedge \dots \wedge e_{i_d} \in \exterior{d} \CC^m$.
If $(v_1, \dots, v_d)$ is a basis for $V\in\Gr(d,m)$, then we have
\[
v_1 \wedge \dots \wedge v_d = \sum_{I \in {[m] \choose d}} \Delta_I e_I,
\]
where the $\Delta_I$'s are the \Plucker coordinates of $V$. 

\begin{proposition}[{\cite[Section 9.1]{fulton97}}]
\label{prop:exteriorplucker}
Let $\omega := \sum_{I \in {[m] \choose d}} \Delta_I e_I \in \exterior{d} \CC^m$, where $\Delta_I\in\CC$ for $I\in\binom{[m]}{d}$. Then the following are equivalent:
\begin{enumerate}[(a)]
\item\label{exteriorplucker1} the coefficients $\Delta_I$ of $\omega$ 
satisfy the \Plucker relations \eqref{eq:prels};
\item\label{exteriorplucker2} $\omega = v_1 \wedge \dots \wedge v_d$
for some $v_1, \dots, v_d \in \CC^m$.
\end{enumerate}
Furthermore, if \ref{exteriorplucker2} holds and $\omega \neq 0$, then 
$\big[\Delta_I : I \in {[m] \choose d}\big]$
are the \Plucker coordinates of $V = \langle v_1, \dots, v_d\rangle$.
\end{proposition}

\subsubsection{A basis from the \Plucker coordinates}

Given a basis for $V \in \Gr(d,m)$, by definition we obtain the \Plucker coordinates
as minors of the matrix of coefficients in the standard basis
$(e_1, \dots, e_m)$.  Conversely,
if we know the \Plucker coordinates of $V$,
there is a straightforward way to obtain a basis. Namely, for $I= (i_1, \dots, i_d) \in [m]^d$ and $j, k \in [m]$, we define
\[
    \Delta_{(I-j)+k} 
     := \begin{cases}
             (-1)^{d-s}\Delta_{(i_1, \dots, \widehat{i_s},  \dots i_d, k)},
            &\quad \text{if $j = i_s$ for a unique $s \in [d]$;}
          \\
            0,
            &\quad \text{otherwise.}
          \end{cases}
\]
Then we have:
\begin{proposition}
\label{prop:pluckerbasis}
Suppose that $V \in \Gr(d,m)$ has \Plucker coordinates
$\big[\Delta_I : I \in {[m] \choose d}\big]$, and take $J \in {[m] \choose d}$ such that $\Delta_J \neq 0$. Then
    $\big( \sum_{k=1}^m \Delta_{(J-j)+k} e_k : j \in J \big)$
is a basis for $V$.
\end{proposition}

\begin{proof}
Take a $d\times m$ matrix $A$ whose row span is $V$, and let $g$ be the $d\times d$ submatrix of $A$ with column set $J$. Since $\Delta_J\neq 0$, we have $g\in\GL_d$, so $g^{-1}A$ also represents $V$. We can verify that the rows (up to rescaling) of $g^{-1}A$ form the desired basis.
\end{proof}

\subsubsection{Indexing by partitions}
\label{sec:partitions}

For our purposes, it is also useful to index \Plucker coordinates
by partitions.  A \defn{partition} $\lambda = (\lambda_1, \dots, \lambda_s)$
is a weakly decreasing sequence of positive integers
$\lambda_1 \geq \dots \geq \lambda_s > 0$.  The \defn{length} and
\defn{size} of $\lambda$ are $\ell(\lambda)=s$ and 
$|\lambda| = \lambda_1 + \dots + \lambda_s$, respectively, and
the notation $\lambda \vdash k$ means $|\lambda| = k$.
When convenient, we use exponential notation, e.g.,
$4^321^4 = (4,4,4,2,1,1,1,1)$.
We adopt the convention that $\lambda_j = 0$ for all $j > \ell(\lambda)$, 
and partitions may be written with any number of trailing zeros, e.g., $(3,3,1)$ and $(3,3,1,0,0,0,0)$ are considered to be
the same partition. 

The \defn{diagram} of a partition $\lambda$ is the array of $|\lambda|$ left-justified boxes with $\lambda_i$ boxes in row $i$ for all $i \ge 1$; see \cref{fig:partition}.
If $\lambda$ and $\mu$ are partitions,
we write $\mu \subseteq \lambda$ if the diagram of $\mu$ is contained in the diagram of $\lambda$, i.e., $\mu_i \leq \lambda_i$ for all $i \ge 1$.
\begin{figure}[t]
\begin{center}
\begin{tikzpicture}[baseline=(current bounding box.center)]
\tikzstyle{out1}=[inner sep=0,minimum size=1.2mm,circle,draw=black,fill=black]
\tikzstyle{in1}=[inner sep=0,minimum size=1.2mm,circle,draw=black,fill=white]
\pgfmathsetmacro{\unit}{0.922};
\pgfmathsetmacro{\s}{0.86}
\useasboundingbox(0,1*\unit)rectangle(5*\unit,-3*\unit);
\coordinate (vstep)at(0,-0.22);
\coordinate (hstep)at(0.16,0);
\coordinate (vepsilon)at(0,-0.02*\unit);
\coordinate (hepsilon)at(0.02*\unit,0);
\draw[thick](0,0)--(5*\unit,0)--(5*\unit,-3*\unit)--(0,-3*\unit)--cycle;
\node[inner sep=0]at(0,0){\scalebox{1.6}{\begin{ytableau}
\none \\
\none \\
\none \\
\none & \none & \none & *({black!20!}) & *({black!20!}) & *({black!20!}) \\
\none & \none & \none & *({black!20!}) & *({black!20!}) \\
\none & \none & \none
\end{ytableau}}};
\node[inner sep=0]at($(0,-2.5*\unit)+(hstep)$){\scalebox{\s}{$1$}};
\node[inner sep=0]at($(2*\unit,-1.5*\unit)+(hstep)$){\scalebox{\s}{$4$}};
\node[inner sep=0]at($(3*\unit,-0.5*\unit)+(hstep)$){\scalebox{\s}{$6$}};
\node[inner sep=0]at($(0.5*\unit,-2*\unit)+(vstep)$){\scalebox{\s}{$2$}};
\node[inner sep=0]at($(1.5*\unit,-2*\unit)+(vstep)$){\scalebox{\s}{$3$}};
\node[inner sep=0]at($(2.5*\unit,-1*\unit)+(vstep)$){\scalebox{\s}{$5$}};
\node[inner sep=0]at($(3.5*\unit,0)+(vstep)$){\scalebox{\s}{$7$}};
\node[inner sep=0]at($(4.5*\unit,0)+(vstep)$){\scalebox{\s}{$8$}};
\node[inner sep=0]at(0,-1.5*\unit)[label={[left=2pt]$d=3$}]{};
\node[inner sep=0]at(2.5*\unit,0)[label={[above=2pt]$m-d=5$}]{};
\end{tikzpicture}
\caption{The partition $\lambda = (3,2)$ corresponds to the set $I = (1,4,6)\in\binom{[8]}{3}$, where $d = 3$ and $m = 8$. When we label the edges of the border of the diagram of $\lambda$ by $1, \dots, m$ from southwest to northeast, the elements of $I$ are the labels of the vertical edges.}\label{fig:partition}
\end{center}
\end{figure}

We introduce a variable $\Delta^\lambda$ for every 
partition $\lambda$ (using a superscript to distinguish $\Delta^\lambda$ from $\Delta_I$).  If $I = (i_1, \dots, i_d) \in {[m] \choose d}$,
we identify
\begin{equation}
\label{eq:identification}
   \Delta_I \equiv \Delta^\lambda, \qquad \text{where } \lambda = (i_d-d, \dots, i_2-2, i_1-1)
\,;
\end{equation}
see \cref{fig:partition}. We emphasize that this identification depends on the choice of $d$ (but not of $m$). Thus, a partition $\lambda$ indexes a \Plucker coordinate of
$\Gr(d,m)$ if and only if $\lambda \subseteq \rectangle$, where
$\rectangle = (m-d)^d$. For example, indexing by partitions, the \Plucker relation
\eqref{eq:gr24plucker}
can be rewritten as
\begin{equation}
\label{eq:firstplucker}
   - \Delta^0 \Delta^{22} + \Delta^{1} \Delta^{21} - \Delta^{11}\Delta^{2} = 0
\,.
\end{equation}

When indexed by partitions, the \Plucker relations 
\eqref{eq:prels} are stable.
This means that if $d' \geq d$ and $m'-d' \geq m-d$, then the \Plucker
relations for $\Gr(d,m)$ are a subset of the \Plucker relations for
$\Gr(d',m')$.  Furthermore, if we set $\Delta^\lambda$ to $0$ for all $\lambda \nsubseteq \rectangle = (m-d)^d$, then every \Plucker relation for 
$\Gr(d',m')$ becomes a (possibly trivial) \Plucker relation for $\Gr(d,m)$.
Hence taking the union of all non-trivial \Plucker relations for all 
Grassmannians gives
the complete list of \defn{all \Plucker relations}, which
are valid for all $\Gr(d,m)$'s.  

Explicitly, for a partition $\lambda$ and $i \geq 1$, let $c = 0$ if $i > \lambda_1$,
and otherwise let $c$ be the unique
positive integer such that $\lambda_c \geq i > \lambda_{c+1}$.
Let $\lambda^{(i)}$ and $\lambda^{(-i)}$ denote the partitions
\begin{align*}
  \lambda^{(i)} &:= (\lambda_1-1, \dots, \lambda_c-1,\, i-1,\,
                  \lambda_{c+1}, \lambda_{c+2}, \dots ) \,,
\\
  \lambda^{(-i)} &:= (\lambda_1+1, \dots, \lambda_{i-1}+1, 
                  \, \lambda_{i+1}, \lambda_{i+2}, \dots)
\,;
\end{align*}
see \cref{fig:operations}. With this notation, the complete list of \Plucker relations can be written as follows \cite[Theorem 4.1]{carrell_goulden10}:
\begin{figure}[t]
\begin{center}
$\lambda = \,\ydiagram{3,2}$ \qquad\qquad $\lambda^{(3)} = \,\ydiagram{2,2,2}$ \qquad\qquad $\lambda^{(-2)} = \,\ydiagram{4}$
\caption{When $\lambda = (3,2)$, we have $\lambda^{(3)} = (2,2,2)$ and $\lambda^{(-2)} = (4)$. In general, $\lambda^{(i)}$ changes the $i$th horizontal edge from the left into a vertical edge, and $\lambda^{(-j)}$ changes the $j$th vertical edge from the top into a horizontal edge.}\label{fig:operations}
\end{center}
\end{figure}
\begin{equation}
\label{eq:prelspartitions}
   \sum_{\substack{i,j \geq 1, \\ 
    |\lambda^{(-i)}| + |\mu^{(j)}| = |\lambda| + |\mu|+ 1}}
    (-1)^{|\mu|-|\mu^{(j)}|+i+j}
      \Delta^{\lambda^{(-i)}} \Delta^{\mu^{(j)}} = 0
  \qquad{\text{for all partitions $\lambda$ and $\mu$}}
\,.
\end{equation}
(The condition $|\lambda^{(i)}| + |\mu^{(-j)}| = |\lambda| + |\mu|+ 1$ implies that the sum is finite.) For example, taking $\lambda = 0$ and $\mu = 3$ yields the \Plucker relation \eqref{eq:firstplucker}.

\cref{thm:main}\ref{main_pluckers} asserts that for all $t \in \CC$,
the operators
$\beta^\lambda(t)$ satisfy all of the equations \eqref{eq:prelspartitions}.
Since $\beta^\lambda(t) = 0$ for $|\lambda| > n$, this is equivalent
to asserting that they satisfy the \Plucker relations \eqref{eq:prels} 
for $\Gr(n,2n)$.
Furthermore, it suffices to prove these for $t=0$, since
$\beta^\lambda \mapsto \beta^\lambda(t)$ under the change of parameters
$(z_1, \dots, z_n) \mapsto (z_1+t, \dots, z_n+t)$.

\subsubsection{Single-column and single-row \Plucker relations}
\label{sec:scprels}

The \Plucker relations~\eqref{eq:prels} corresponding
to $I = (1,2,\dots, d+1)$ for some Grassmannian $\Gr(d,m)$ will
play a special role. We refer to these as the \defn{single-column \Plucker relations}, since when they are rewritten in terms of partitions, the first of the two indexing partitions has at most one column.
Equivalently, these are the \Plucker relations \eqref{eq:prelspartitions}
corresponding to $\lambda = 0$.
For example, \eqref{eq:firstplucker} is a
single-column relation.  One of the main steps in the proof of 
\cref{thm:main}\ref{main_pluckers} will be to show explicitly that the operators 
$\beta^\lambda$ satisfy all of the single-column \Plucker relations.

Similarly, the \defn{single-row \Plucker relations} are the relations
\eqref{eq:prels} corresponding to
$J = (1,2, \dots, d-1)$ for some Grassmannian $\Gr(d,m)$, or equivalently, 
the relations \eqref{eq:prelspartitions} corresponding to $\mu = 0$.
The single-row \Plucker relations will play a role in our discussion
of bases for the spaces $V_E$ in \cref{sec:bases}.


\subsection{Schubert varieties and the Wronski map}

For partitions $\mu \subseteq \lambda$, put $|\lambda/\mu| := |\lambda| - |\mu|$, and let $\lambda/\mu$ be the array of $|\lambda/\mu|$ boxes formed by the set difference of the diagrams of $\lambda$ and $\mu$. We define $\numsyt{\lambda/\mu}$ as the number of \defn{standard Young tableaux} of shape $\lambda/\mu$, that is, the number of ways to fill the boxes of $\lambda/\mu$ with the numbers $1, \dots, |\lambda/\mu|$ (each used exactly once) such that numbers increase along both rows (left to right) and columns (top to bottom). In particular, $\numsyt{\lambda} := \numsyt{\lambda/0}$ is the number of standard Young tableaux of shape $\lambda$. These numbers play a prominent role in describing 
the geometry of the Wronski map,
arising as both degrees of Wronski maps, 
and as coefficients in explicit formulas in
terms of \Plucker coordinates.

For any $d \ge \ell(\lambda)$, we have the following formula for $\numsyt{\lambda}$ \cite[Exercise 3.20]{sagan01}:
\begin{equation}
\label{eq:hookformula}
\frac{\numsyt{\lambda}}{|\lambda|!} = \frac{\prod_{1 \le i < j \le d}(\lambda_i - i - \lambda_j + j)}{\prod_{i=1}^d(\lambda_i - i + d)!}\,.
\end{equation}
We also have a determinantal formula for $\numsyt{\lambda/\mu}$ \cite[Corollary 7.16.3]{stanley24}:
\begin{equation}
\label{eq:numsyt}
   \frac{\numsyt{\lambda/\mu}}{|\lambda/\mu|!} = \det 
   \left( \frac{1}{(\lambda_i -i -\mu_j+j)!} \right)_{1 \leq i,j \leq d}
\,,
\end{equation}
where by convention, 
$\frac{1}{k!} := 0$ if $k$ is a negative integer.

\subsubsection{Schubert varieties}
\label{sec:schubert}

A \defn{complete flag} in $\CC_{m-1}[u]$ is a tuple $F_\bullet : F_0 \subsetneq \dots \subsetneq F_m$ of nested subspaces, where $\dim F_j = j$ for all $j$.  
For each partition $\lambda \subseteq \rectangle$,
we have a \defn{Schubert variety} of $\Gr(d,m)$ relative to 
$F_\bullet$:
\[
    X_\lambda F_\bullet := 
   \{V \in \Gr(d,m) \mid \dim(V \cap F_{m-d+i-\lambda_i}) \geq i) 
     \text{ for all $i \in [d]$}\}
\,.
\]

A \defn{Schubert problem} on $\Gr(d,m)$ is to find the points
(or just the number of points) in an intersection of the form
\begin{equation}
\label{eq:generalschubertintersection}
     \svarnu
     \cap X_{\mu_1} F^{(1)}_\bullet \cap \dots \cap X_{\mu_s} F^{(s)}_\bullet \,,
\end{equation}
where $F^{(1)}_\bullet, \dots, F^{(s)}_\bullet$ are flags,
and $\nu, \mu_1, \dots, \mu_s$ are partitions such that $|\nu| = |\mu_1| + \cdots + |\mu_s|$. 
When the intersection \eqref{eq:generalschubertintersection} is transverse, 
the number of points in the intersection
is the coefficient of the Schur function $\s_\nu$ in the
product $\s_{\mu_1} \dotsm \s_{\mu_s}$ \cite[Section 9.4]{fulton97}. 
The Schubert problem \eqref{eq:schubertproblem} corresponds to the
intersection
\begin{equation}
\label{eq:schubertintersection}
    \svarnu
     \cap X_1 F^{(1)}_\bullet \cap \dots \cap X_1 F^{(n)}_\bullet,
\end{equation}
where $n = |\nu|$ and the flags $F^{(1)}_\bullet, \dots, F^{(n)}_\bullet$ are 
such that $F^{(i)}_{m-d} = W_i$.  When the intersection 
\eqref{eq:schubertintersection} is transverse, it contains $\numsyt{\nu}$ points.

For $w \in \CC$, let 
\[
F_j(w) := \langle \gamma(w), \gamma'(w), \dots, \gamma^{(j-1)}(w) \rangle = \langle (u+w)^{m-1}, (u+w)^{m-2}, \dots, (u+w)^{m-j} \rangle
\]
be the osculating $j$-plane to the moment curve \eqref{eq:moment} at $w$.
Then $F_\bullet(w) : F_0(w) \subsetneq \dots \subsetneq F_m(w)$ is a
complete flag, called the \defn{osculating flag} to $\gamma$ at $w$.
For these flags, we use the shorthand notation
\[
    X_\lambda(w) := X_\lambda F_\bullet(w)
\,.
\]
We also put
$F_\bullet(\infty) := \lim_{w \to \infty} F_\bullet(w)$, which
is just the \defn{standard flag}, with
$F_j(\infty) = \langle 1, u, \dots, u^{j-1} \rangle$.
The Schubert variety $\svarnu$ is in fact
\[
    \svarnu = X_{\nu^\vee} (\infty),
\]
where $\nu^\vee := (m-d-\nu_d, \dots, m-d-\nu_1)$ is the complement of $\nu$ inside $\rectangle$.
Our conventions are such that $\codim X_\lambda F_\bullet = |\lambda|$
for any flag $F_\bullet$, whereas 
$\dim \svarnu = \dim \scellnu = |\nu|$.

In terms of \Plucker coordinates, $\scellnu$ and $\svarnu$
have straightforward descriptions:
\begin{proposition}[{\cite[Section 9.4]{fulton97}}]
\label{thm:schubertplucker}
The Schubert variety $\svarnu$ is the closed subscheme of $\Gr(d,m)$ defined by 
$\Delta^\lambda = 0$ for all $\lambda \not \subseteq \nu$. The Schubert cell $\scellnu$ is the open subscheme of $\svarnu$ defined by 
$\Delta^\nu \neq 0$.
\end{proposition}

Henceforth, when we refer to the \Plucker coordinates of 
a $d$-plane $V \in \svarnu$, we will frequently consider only 
the \Plucker coordinates 
indexed by partitions $\lambda \subseteq \nu$, and
write these as $[\Delta^\lambda : \lambda \subseteq \nu]$.
By \cref{thm:schubertplucker} all other \Plucker coordinates 
are zero.  If $V \in \scellnu$, define the 
\defn{normalized \Plucker coordinates} of $V$
to be the unique scaling 
$(\Delta^\lambda : \lambda \subseteq \nu)$
of the \Plucker coordinates such that 
$\Delta^\nu = \frac{|\nu|!}{\numsyt\nu}$.

\begin{remark}
\label{rmk:0schubertconditions}
More generally, every Schubert variety $X_\lambda F_\bullet$ is defined
as a scheme by a system of linear equations in the \Plucker
coordinates (we omit the proof).  For example, $X_\mu(0)$ is defined by the
equations $\Delta^\lambda = 0$ for all $\lambda \not \supseteq \mu$.
\end{remark}

\subsubsection{The Wronski map on \texorpdfstring{$\Gr(d,m)$}{Gr(d,m)}}

Recall that if $V$ is a finite-dimensional vector space of polynomials, $\Wr(V)$ is defined 
to the monic polynomial which is a scalar multiple of
\[
    \Wr(f_1, \dots, f_d) := 
   \begin{vmatrix}
    f_1 & f_1' & f_1'' & \dots & f_1^{(d-1)} \\
    f_2 & f_2' & f_2'' & \dots & f_2^{(d-1)} \\
    \vdots & \vdots & \vdots & \ddots & \vdots \\\
    f_d & f_d' & f_d'' & \dots & f_d^{(d-1)} \\
   \end{vmatrix}
\,,
\]
where $(f_1, \dots, f_d)$ is any basis for $V$.
If $V \in \Gr(d,m)$, then $\Wr(V)$ is a polynomial of degree
at most $d(m-d)$, and we obtain
a projective morphism $\Wr : \Gr(d,m) \to \PP^{d(m-d)}$
called the \defn{Wronski map} on $\Gr(d,m)$.
 Here $\PP^{d(m-d)}$ is identified 
with the projectivization of $\CC_{d(m-d)}[u]$ via \eqref{eq:isomorphism}.
Eisenbud and Harris \cite[Theorem 2.3]{eisenbud_harris83} showed that this is a finite morphism 
of degree $\numsyt{\rectangle}$, the number of standard Young tableaux
of shape $\rectangle$.

Restricting the Wronski map to the Schubert cell $\scellnu$, we obtain
a finite morphism of affine schemes $\Wr : \scellnu \to \monics$,
where $\monics$ is the space of monic polynomials of degree $n = |\nu|$.
In this case, both the domain and codomain are 
isomorphic to $n$-dimensional affine space, and the restricted
Wronski map has degree $\numsyt{\nu}$.

We have an explicit formula for $\Wr(V)$ in terms of \Plucker
coordinates:
\begin{proposition}
\label{prop:pluckerwronskian}
Let $V \in \scellnu$, where $\nu \vdash n$.  
If $(\Delta^\lambda : \lambda \subseteq \nu)$
are the normalized \Plucker coordinates of $V$, then
\begin{equation}
\label{eq:pluckerwronskian}
   \Wr(V) = \sum_{\lambda \subseteq \nu} 
    \frac{\numsyt{\lambda}}{|\lambda|!} \Delta^\lambda u^{|\lambda|}  
\,.
\end{equation}
\end{proposition}

\begin{proof}
By \cite[Proposition 2.3]{purbhoo10}, we have
\[
   \Wr(V) = \sum_{\lambda \subseteq \nu}\frac{\prod_{1 \le i < j \le d}(\lambda_i - i - \lambda_j + j)}{\prod_{i=1}^d(\lambda_i - i + d)!}
    \Delta^\lambda u^{|\lambda|}  
\,.
\]
(The denominator $\prod_{i=1}^d(\lambda_i - i + d)!$ above accounts for the fact that \cite{purbhoo10} uses a different convention for the isomorphism $\CC^m \xrightarrow{\simeq} \CC_{m-1}[u]$; namely, the denominator $(j-1)!$ in \eqref{eq:isomorphism} is omitted.) We obtain \eqref{eq:pluckerwronskian} by applying \eqref{eq:hookformula}.
\end{proof}

The Wronski map is closely related to Schubert varieties for
osculating flags of $\gamma$:

\begin{proposition}
\label{prop:schubertwronskian2}
Let $V \in \Gr(d,m)$.
\begin{enumerate}[(i)]
\item\label{schubertwronskian2_divides} For $w \in \CC$, $(u+w)^k$ divides $\Wr(V)$ if and only if
    $V \in X_\lambda(w)$ for some $\lambda \vdash k$, 
     $\lambda \subseteq \rectangle$.
\item\label{schubertwronskian2_infinity} We have $\deg \Wr(V) = n$ if and only if $V \in \scellnu$ 
       for some $\nu \vdash n$, $\nu \subseteq \rectangle$.
\end{enumerate}
\end{proposition}

\begin{proof}
By translation, it suffices to prove part \ref{schubertwronskian2_divides} when $w = 0$. This case follows from \eqref{eq:pluckerwronskian}, using \cref{rmk:0schubertconditions}. Part \ref{schubertwronskian2_infinity} follows directly from \eqref{eq:pluckerwronskian}.
\end{proof}

The $k=1$ case of part \ref{schubertwronskian2_divides} of \cref{prop:schubertwronskian2} is equivalent to 
\cref{prop:schubertwronskian}.
Part \ref{schubertwronskian2_infinity} should be regarded as the $w=\infty$ case of part \ref{schubertwronskian2_divides}.

As a consequence of \cref{prop:schubertwronskian2}, 
every fibre of the Wronski map can be described
as a union of Schubert intersections.  
Namely, suppose
$g(u) = (u+z_1)^{\kappa_1} \dotsm (u+z_s)^{\kappa_s}$,
where $z_1, \dots, z_s$ are distinct complex numbers and
$\kappa_1+ \dots + \kappa_s = n$.
Then for the Wronski map $\Wr : \scellnu \to \monics$, we have the equality of sets
\begin{equation}
\label{eq:fibre}
   \Wr^{-1}(g) = \bigcup_{\mu_1 \vdash \kappa_1, \dots, \mu_s \vdash \kappa_s}
       \svarnu \cap X_{\mu_1}(z_1) \cap \dots \cap X_{\mu_s}(z_s)
\,.
\end{equation}
This equality holds scheme-theoretically if and only if $\nu$ equals $n$ or $1^n$, or 
$\kappa_1, \dots, \kappa_s \leq 2$ (i.e.\ $g(u)$ has no roots of multiplicity
greater than $2$); see \cref{cor:schemeequality}.
As we will discuss in \cref{sec:dimensions}, 
this accounts for the discrepancy between the scheme structures 
of $\Wr^{-1}(g)$ and $\Spec \bethenu$.

\subsubsection{\texorpdfstring{$\PGL_2(\CC)$}{PGL\textunderscore{}2(C)}-action and translation action}

For all $k \geq 0$, the group $\GL_2(\CC)$ acts on 
$\CC_k[u]$ by M\"obius transformations.
If $\phi = 
\big(\begin{smallmatrix} 
\phi_{11} & \phi_{12} \\ \phi_{21} & \phi_{22}
\end{smallmatrix}\big) \in \GL_2(\CC)$ and
$f(u) \in \CC_k[u]$,
the action is given by
\[
\phi f(u) := (\phi_{21} u + \phi_{11})^k
f\Big(\frac{\phi_{22} u + \phi_{12}}{\phi_{21} u + \phi_{11}}\Big)
\,.
\]
This induces a $\PGL_2(\CC)$-action on
$\Gr(d,m)$ and $\PP^{d(m-d)}$.  The Wronski map on $\Gr(d,m)$ is
$\PGL_2(\CC)$-equivariant with respect to these actions.

The additive group $(\CC,+)$ is isomorphic the
unipotent subgroup consisting of matrices of the form
$\translatematrix{t}$, and therefore also acts on $\CC_k[u]$ and $\Gr(d,m)$. We call this the \defn{translation action}, since $\translatematrix{t} f(u) = f(u+t)$. 
For $V \in \Gr(d,m)$, we write 
\[
    V(t) := \translatematrix{t} V
\,.
\]
Note that $\translatematrix{t} X_\lambda(w) = X_\lambda(w+t)$. 
Furthermore, if
$\Wr(V) = g(u)$, then $\Wr(V(t)) = g(u+t)$,
and so the translation action of gives an isomorphism 
between the fibres of the Wronski map over $g(u)$ and $g(u+t)$.

The translation action also preserves the Schubert cell $\scellnu$.
Explicitly, in terms of \Plucker coordinates, we have:

\begin{proposition}
\label{prop:geometrictranslation}
Let $V \in \scellnu$, where $\nu \vdash n$.  For $t \in \CC$, let 
$(\Delta^\lambda(t) : \lambda \subseteq \nu)$ be the 
normalized \Plucker coordinates of $V(t)$, and let $\Delta^\lambda(t) = 0$ for $\lambda \not \subseteq \nu$. Then for all $s, t \in \CC$ and all partitions $\mu$, we have
\begin{equation}
\label{eq:geometrictranslation}
    \Delta^\mu(s+t) = \sum_{\lambda \supseteq \mu} 
      \frac{\numsyt{\lambda/\mu}}{|\lambda/\mu|!}
      \Delta^\lambda(s) t^{|\lambda/\mu|}
\,.
\end{equation}
\end{proposition}

\begin{proof}
Let $\phi_t\in\End(\CC_{m-1}[u])$ act by translation by $t$. We can represent $\phi_t$ as an $m\times m$ matrix, whose $(i,j)$-entry we calculate via \eqref{eq:isomorphism}:
\[
(\phi_t)_{i,j} = \Big[\frac{u^{i-1}}{(i-1)!}\Big]\,\frac{(u+t)^{j-1}}{(j-1)!}\, = \frac{1}{(j-i)!}\,,
\]
where the operator $[\frac{u^{i-1}}{(i-1)!}]$ extracts the coefficient of $\frac{u^{i-1}}{(i-1)!}$. If $A$ is a $d\times m$ matrix which represents $V(s)$, then $V(s+t)$ is represented by $A\phi_t^T$. We obtain \eqref{eq:geometrictranslation} by applying the Cauchy--Binet identity and \eqref{eq:numsyt}.
\end{proof}

Note that \eqref{eq:geometrictranslation} has the same form as the translation identity \eqref{eq:translationidentity}.  This
explains the significance of the translation identity: it asserts that the
operators $\beta^\lambda(t)$ behave exactly like \Plucker coordinates
under the translation action. We also note the similarity between equations \eqref{eq:pluckerwronskian}
and \eqref{eq:geometrictranslation}.  In the notation of 
\cref{prop:geometrictranslation}, equation
\eqref{eq:pluckerwronskian} says that $\Wr(V) = \Delta^0(u)$.

\subsubsection{Inclusions of Grassmannians}

Let $\nu \vdash n$ be a partition such that $\nu \subseteq \rectangle = (m-d)^d$.  Then 
$\scellnu$ can be regarded as Schubert cell of $\Gr(d,m)$, or of any
other Grassmannian $\Gr(d',m')$ with $d' \geq d$ and $m'-d' \geq m-d$.
We now argue that it does not matter which Grassmannian we choose to
work inside, as far as the geometry of the
Schubert variety, \Plucker coordinates, total positivity, the Wronski map, and 
the translation action are concerned.

In the following proposition, we write $\svarnu_{d',m'}$
to specifically mean the Schubert variety 
$\svarnu \subseteq \Gr(d',m')$,
and $\Wr_{d',m'}$ to mean the Wronski map on $\Gr(d',m')$.

\begin{proposition}
\label{prop:changegr}
Let $\imath_1 : \Gr(d,m) \hookrightarrow \Gr(d, m+1)$ and 
$\imath_2 : \Gr(d,m) \hookrightarrow \Gr(d+1, m+1)$ denote the
inclusions of Grassmannians defined by
\[
   \imath_1(V) := V
   \qquad\text{and}\qquad \imath_2(V) := \{f(u) \in \CC[u] \mid f'(u) \in V\}
\,.
\]
\begin{enumerate}[(i)]
\item\label{changegr1}
The maps $\imath_1$ and $\imath_2$ restrict to isomorphisms 
of Schubert varieties
\[
  \imath_1 : \svarnu_{d,m} \xrightarrow{\simeq} \svarnu_{d,m+1}
\qquad\text{and}\qquad
  \imath_2 : \svarnu_{d,m} \xrightarrow{\simeq} \svarnu_{d+1,m+1}
\,.
\]
\item\label{changegr2} Both isomorphisms in \ref{changegr1} are, in terms of
\Plucker coordinates, defined by the identity map:
$\Delta^\lambda \mapsto \Delta^\lambda$, $\lambda \subseteq \nu$.
That is, $V$, $\imath_1(V)$, and $\imath_2(V)$ all have the same
\Plucker coordinates.
\item\label{changegr3} If any one of the three elements $V$, $\imath_1(V)$, and $\imath_2(V)$ is totally nonnegative (respectively, totally positive in its Schubert cell), then so are the other two.
\item\label{changegr4} $\Wr_{d,m+1} \circ \imath_1 = \Wr_{d,m}$ and 
$\Wr_{d+1,m+1} \circ \imath_2 = \Wr_{d,m}$.
\item\label{changegr5} Both $\imath_1$ and $\imath_2$ are equivariant with respect 
to the translation action. 
\end{enumerate}
\end{proposition}

\begin{proof}
We can verify parts \ref{changegr4} and \ref{changegr5} directly. Parts \ref{changegr1}--\ref{changegr3} follow from the fact that if $V\in\Gr(d,m)$ is represented by the $d\times m$ matrix $A$, then $\imath_1(V)$ and $\imath_2(V)$ are represented by the matrices
\[
\begin{pmatrix}
& & & 0 \\
& A & & \vdots \\
& & & 0
\end{pmatrix} \qquad\text{and}\qquad \begin{pmatrix}
1 & 0 & \cdots & 0 \\
0 & & & \\
\vdots & & A & \\
0 & & &
\end{pmatrix},
\]
respectively.
\end{proof}

Note that the maps $\imath_1$ and $\imath_2$ are not 
$\PGL_2(\CC)$-equivariant, 
and therefore if we are concerned with the full 
$\PGL_2(\CC)$-action, the choice of Grassmannian is still relevant.
However, for the most part, we will only be concerned with the
translation action.  We can then regard $\scellnu$ as a subvariety of
any sufficiently large Grassmannian $\Gr(d,m)$, and the geometry we
consider will independent of this choice.  When $\nu \vdash n$, it will
be particularly convenient to work inside $\Gr(n,2n)$.

\subsubsection{Fundamental differential operators}

We now recall some background from \cite{purbhoo}. Consider linear differential operators of the form
\[
 D = \psi_0(u) \du^d + \psi_1(u)\du^{d-1} + \dots + \psi_{d-1}(u)\du + \psi_d(u)
\,,
\]
where $\psi_0(u), \dots, \psi_d(u) \in \CC(u)$ and $\du$ is the differentiation operator.  
The operator $D$ acts on functions $f(u)$ as
\[
   Df(u) := \psi_0(u) f^{(d)}(u) + \psi_1(u) f^{(d-1)}(u)
   + \dots + \psi_{d-1}(u) f'(u) + \psi_d(u) f(u)
\,.
\]
If $\psi_0(u) \neq 0$,
we say that $D$ has \defn{order} $d$, and  $D$ is \defn{monic} if
$\psi_0(u) = 1$.

We will only be concerned with solutions to
$Df = 0$ where $f$ is a polynomial.  Define
\[
    \pker D := \{f(u) \in \CC[u] \mid Df(u) = 0\}
\,.
\]
A basic fact about a linear differential operator is that its kernel has dimension less than or equal to its order
(see, e.g., {\cite[Section 3.3.2]{ince27}}), so in particular, we have:

\begin{proposition}
\label{prop:Dnullity}
If $D$ is a nonzero linear differential operator of order $d$ with coefficients
in $\CC(u)$, then
\[
    \dim \pker D  \leq d
\,.
\]
\end{proposition}

Now suppose that $V \in \scellnu \subseteq \Gr(d,m)$, and let $(f_1, \dots, f_d)$
be any basis for $V$.  Consider
\[
     D_Vf := \frac{\Wr(f_1, \dots, f_d, f)}{\Wr(f_1, \dots, f_d)}
\,.
\]
Expanding the determinant for $\Wr(f_1, \dots, f_d, f)$ along the
bottom row, we see that $D_V$
is a monic linear differential operator of order $d$
with coefficients in $\CC(u)$.  Specifically, 
we can write $D_V$ in the form
\[
  D_V = \frac{1}{\Wr(V)} 
       \sum_{k=0}^d \sum_{\ell =0}^{n-k} (-1)^k \psi_{k,\ell} u^{n-k-\ell} \du^{d-k}
\,,
\]
where $\psi_{k,\ell} \in \CC$, and $n = |\nu|$.
$D_V$ is called the \defn{fundamental differential operator} of $V$.

\begin{proposition}
\label{prop:FDOunique}
$D_V$ is the unique monic linear operator of order $d = \dim V$ 
with the property
that $V = \pker D_V$.
\end{proposition}

\begin{proof}
Let $(f_1, \dots, f_d)$ be a basis for $V$. We have
\[
D_V f = 0 \qquad\iff\qquad \Wr(f_1, \dots, f_d, f) = 0 \qquad\iff\qquad f\in V,
\]
so $V = \pker D_V$.
If there were another monic operator $D_V' \neq D_V$
of order $d$ such that $V = \pker D_V'$,
then we would have $V \subseteq \pker(D_V - D_V')$,  contradicting \cref{prop:Dnullity}.
\end{proof}

Thus, the numbers 
$\boldpsi := (\psi_{k,\ell} : 0 \le k \leq d,\ 0 \le \ell \leq n-k)$ 
uniquely determine (and are
uniquely determined by) $V$.  
We may regard $\boldpsi$ as a system
of affine coordinates on the Schubert cell $\scellnu$.
We refer to these as \defn{FDO coordinates}, since they come
from the fundamental differential operator.
Unlike the \Plucker coordinates, it is difficult to state
the precise relations among FDO coordinates, though in
principle it is possible. 
The main fact we will need is that we can express the normalized \Plucker coordinates of $V$ as 
polynomial functions of $\boldpsi$:

\begin{lemma}
\label{lem:FDOplucker}
Let $\nu \vdash n$.
For $V \in \scellnu$, 
let $(\Delta^\lambda_V : \lambda \subseteq \nu)$
denote the normalized \Plucker coordinates of $V$,
and let $\boldpsi_V$ denote the FDO coordinates of $V$.  
Then there exist multivariate polynomials $p^\lambda_\nu$ with rational coefficients
such that for all $V \in \scellnu$, we have
\[
     \Delta^\lambda_V = p^\lambda_\nu(\boldpsi_V)
\qquad \text{for all $\lambda \subseteq \nu$.}
\]
\end{lemma}

\begin{proof}
By \cite[Lemma 7.5]{purbhoo}, there exist formulas for FDO 
coordinates as polynomial functions of coefficients of a 
basis for $V$; this change of coordinates is manifestly 
triangular and therefore polynomially invertible.  Hence there exist
polynomials which give coefficients of a basis for $V$ in terms of 
the FDO coordinates.
The normalized \Plucker coordinates are minors of the coefficient matrix for any
basis (up to a global scalar), and hence are also given by polynomials in the FDO coordinates.
\end{proof}

\begin{remark}
The polynomials $p^\lambda_\nu$ have high degree and (as described above)
depend on $\nu$. 
By contrast, the formula
for changing from \Plucker coordinates to
FDO coordinates is linear and independent of $\nu$.  Specifically, let
$(\Delta^\lambda(t) : \lambda \subseteq \nu)$ be the 
normalized \Plucker coordinates of $V(t)$,
which are given by \eqref{eq:geometrictranslation} with $s=0$.
Then the FDO coordinates of $V$ are given by the formula
\[
    \psi_{k,\ell} = [u^{n-k-\ell}]\Delta^{1^k}(u)
\,.
\]
This can be proved directly from the definition; it also follows from
\cref{lem:FDOcoords}.
\end{remark}

\begin{remark}
An explicit polynomial formula for \Plucker coordinates in terms 
FDO coordinates, which does not depend on $\nu$ (but does depend on
the choice of Grassmannian $\Gr(d,m)$), can be obtained 
using the basis in \cref{thm:generalbasis} and \cref{prop:dualityautomorphism}.
This construction is discussed in detail (and in greater generality)
in the follow-up paper \cite{karp_mukhin_tarasov25};
the resulting formula is the ``dual Jacobi--Trudi identity'' (Theorem 3.1(ii)) 
therein.
A special case of this formula (corresponding to
$\Gr(d,m)$ with $m-d = \lambda_1$) was previously obtained in 
\cite[(4.29)]{alexandrov_leurent_tsuboi_zabrodin14}.
\end{remark}


\subsection{Symmetric functions}

Let $\Lambda = \Lambda_\CC(\boldx)$ denote the algebra of symmetric
functions over $\CC$, in infinitely many variables $x_1, x_2, \dots$. We assume the reader is familiar with symmetric functions; we refer to \cite[Chapter 7]{stanley24} for background.
There are several standard bases for $\Lambda$, indexed by partitions: 
the monomial
basis $(\m_\lambda)$, the elementary basis $(\e_\lambda)$,
the homogeneous basis $(\h_\lambda)$, the power sum basis $(\p_\lambda)$,
and the Schur basis $(\s_\lambda)$.  
It is sometimes useful to regard $\Lambda$ as being a free polynomial
algebra in the power sum generators,
$\Lambda = \CC[\p_1, \p_2, \p_3, \dots]$.
We will also sometimes consider
the formal completion $\widehat \Lambda$ of $\Lambda$, 
in which infinite linear combinations of the basis elements 
of $\Lambda$ are allowed; equivalently,
$\widehat \Lambda = \CC[[\p_1, \p_2, \p_3, \dots]]$.

The Hall inner product 
$\langle \cdot, \cdot \rangle$ on $\Lambda$ is the unique Hermitian inner 
product such that the Schur functions form an orthonormal basis:
$\langle \s_\lambda, \s_\mu \rangle = \delta_{\lambda,\mu}$.
We also have
$\langle \m_\lambda, \h_\mu \rangle = \delta_{\lambda,\mu}$,
and
$\langle \p_\lambda, \p_\mu \rangle = 0$ for
$\lambda \neq \mu$.  The inner product 
$\langle \p_\lambda, \p_\lambda \rangle$ is a positive integer, denoted 
$\z_\lambda$.
We take the convention that $\langle \cdot, \cdot \rangle$ is linear
in the first component and antilinear in the second component (unlike in \cite[Section 7.9]{stanley24}, where $\langle\cdot,\cdot\rangle$ is bilinear). However, in practice, the symmetric functions appearing in the second
component will always have rational coefficients.

\subsubsection{Representation theory of \texorpdfstring{$\Sn$}{S\textunderscore{}n}}
\label{sec:repthy}

We recall some background from \cite{sagan01}. For a permutation $\sigma \in \Sn$, let $\cyc(\sigma)$ denote
the cycle type of $\sigma$, which is a partition of $n$.  For
$\lambda \vdash n$, write 
$\calC_\lambda = \{\sigma \in S_n \mid \cyc(\sigma) = \lambda\}$ for
the $\Sn$-conjugacy class consisting of all permutations of cycle 
type $\lambda$.
We have 
$\z_\lambda 
= \langle \p_\lambda, \p_\lambda \rangle = \frac{n!}{|\calC_\lambda|}$.

For $\lambda \vdash n$, we have the 
Specht module $\specht{\lambda}$, which is the irreducible
representation of $\Sn$ associated to $\lambda$.
The character of $\specht{\lambda}$ is denoted
$\chi^\lambda : \Sn \to \CC$, and $\dim \specht{\lambda} = \chi^\lambda(\Snidentity) = \numsyt{\lambda}$.
The Frobenius character formula states that 
$\chi^\lambda(\sigma) = \langle \s_\lambda, \p_{\cyc(\sigma)} \rangle \in \ZZ$.
We write $\chi^\lambda_\mu := \langle \s_\lambda, \p_\mu \rangle 
= \chi^\lambda(\sigma)$ for any
$\sigma \in \calC_\mu$, so that we have the change of basis formulas
\begin{equation}
\label{eq:spchangeofbasis}
     \s_\lambda
     = \sum_{\mu \vdash |\lambda|} \frac{\chi^\lambda_\mu}{\z_\mu}\p_\mu
\qquad \text{and} \qquad
     \p_\mu
     = \sum_{\lambda \vdash |\mu|} \chi^\lambda_\mu
        \s_\lambda
\,.
\end{equation}

More generally, if $X \subseteq [n]$ and $\lambda \vdash |X|$, we define
$\chi^\lambda : \symgrp{X} \to \CC$, by identifying $\symgrp{X}$ with 
$\symgrp{|X|}$.  We denote the Specht module of $\symgrp{X}$
corresponding to $\lambda$ as $\specht{\lambda}_X$,
which is identified
with $\specht{\lambda}$ under the isomorphism 
$\symgrp{X} \simeq \symgrp{|X|}$.
Let
\begin{equation}
\label{eq:alphadef}
  \alpha^\lambda_X := 
   \sum_{\sigma \in \symgrp{X}} \chi^\lambda(\sigma) \sigma
   \in \CC[\symgrp{X}]
\,.
\end{equation}
Properties of the $\alpha^\lambda_X$'s will allow us to prove \cref{prop:betapsd}, using the fact that
\begin{equation}
\label{eq:alphatobeta}
\beta^\lambda = \sum_{\substack{X \subseteq [n], \\ |X| = |\lambda|}}
   \alpha^\lambda_X \cdot \prod_{i \in [n]\setminus X} z_i
\,.
\end{equation}

Recall that every finite-dimensional representation $N$ of a finite group $G$ admits a Hermitian inner product such that every $g\in G$ acts as a unitary operator \cite[Section 1.3]{serre98}. The following is a result for general $G$ applied to the case $G = \symgrp{X}$ (using the fact that $\chi^\lambda$ is real-valued):
\begin{proposition}[{\cite[Theorem 8]{serre98}}]
\label{prop:projection}
Let $N$ be a finite-dimensional representation of $\symgrp{X}$ equipped with a unitary inner product. Then $\frac{\numsyt{\lambda}}{|\lambda|!}\alpha^\lambda_X$ acts on $N$ as orthogonal projection onto its $\lambda$-isotypic component (i.e.\ the sum of all submodules isomorphic to $\specht{\lambda}_X$).
\end{proposition}

\begin{corollary}
\label{cor:alphapsd}
For all partitions $\nu \vdash n$ and $\lambda \vdash |X|$, the element
$\frac{\numsyt{\lambda}}{|\lambda|!}\alpha^\lambda_X$ acts on $\spechtnu$ as a real orthogonal projection. In particular, $\alpha^\lambda_X$ is positive semidefinite. Also, it is 
nonzero if and only if $\lambda \subseteq \nu$.
\end{corollary}

\begin{proof}
Since $\alpha^\lambda_X \in \CC[\symgrp{X}]$, we may regard $\spechtnu$ as a representation of $\symgrp{X}$ by restriction. According to the branching rule \cite[Theorem 2.8.3]{sagan01}, this restriction decomposes into irreducibles as $\bigoplus_{\mu\vdash |X|,\mu\subseteq\nu} (\specht{\mu}_X)^{\oplus\numsyt{\nu/\mu}}$. The result follows from \cref{prop:projection}.
\end{proof}

\begin{proof}[Proof of \cref{prop:betapsd}]
Let $\alpha^\lambda_{X,\nu}$ denote the operator $\alpha^\lambda_X$ on
$\spechtnu$.  By \cref{cor:alphapsd}, $\alpha^\lambda_{X,\nu}$ is self-adjoint and positive semidefinite.
By \eqref{eq:alphatobeta}, if $z_1, \dots, z_n \in \RR$, then $\beta^\lambda_\nu$ is a real
linear combination of self-adjoint operators, and hence is itself self-adjoint.  This proves part \ref{betapsd1}. Similarly, if $z_1, \dots z_n \geq 0$, then $\beta^\lambda_\nu$ is a 
nonnegative linear combination of positive-semidefinite operators, and hence is itself positive semidefinite. This proves part \ref{betapsd2}.

If $\lambda \not \subseteq \nu$,
then by \cref{cor:alphapsd} we have
$\alpha^\lambda_{X,\nu} = 0$.  Hence $\beta^\lambda_\nu = 0$, which proves part \ref{betapsd4}.

Finally, for part \ref{betapsd3}, suppose that $z_1, \dots, z_n > 0$ and $\lambda \subseteq\nu$.
In this case, by \cref{cor:alphapsd},
the operators $\alpha^\lambda_{X,\nu}$ are nonzero.
Since $\beta^\lambda_\nu$ is positive semidefinite, it is positive
definite if and only if $\ker \beta^\lambda_\nu = \{0\}$.
Furthermore, since $\beta^\lambda_\nu$ is a positive linear combination of
positive-semidefinite operators, we have
\[
   \ker \beta^\lambda_\nu = 
    \bigcap_{\substack{X \subseteq [n], \\ |X| = |\lambda|}} 
     \ker \alpha^\lambda_{X,\nu} 
\,.
\]
Since the operators $\alpha^\lambda_{X,\nu}$ are nonzero, we have
$\beta^\lambda_\nu \neq 0$.
Note that the right-hand side above is independent of $z_1, \dots, z_n$.  
Therefore, it 
is sufficient to compute $\ker \beta^\lambda_\nu$ for $z_1 = \dots = z_n = 1$. 
In this case $\beta^\lambda$ belongs to the centre of $\CC[\Sn]$,
and so by Schur's lemma \cite[Proposition 4]{serre98}, $\beta^\lambda_\nu$ is a multiple of the identity operator
on $\spechtnu$.
Since $\beta^\lambda_\nu \neq 0$, we deduce that
$\ker \beta^\lambda_\nu = \{0\}$, as required.
\end{proof}

\subsubsection{The KP hierarchy}
\label{sec:kpbackground}

The \emph{KP equation}
\[
\frac{\partial}{\partial x}\left(-4\frac{\partial u}{\partial t} + 6u\frac{\partial u}{\partial x} + \frac{\partial^3 u}{\partial x^3}\right) + 3\frac{\partial^2 u}{\partial y^2} = 0
\]
was introduced by Kadomtsev and Petviashvili \cite{kadomtsev_petviashvili70} as a ($2$+$1$)-dimensional generalization of the KdV equation, to model solitary waves under the effect of weak transverse perturbations (such as shallow water waves). It has since been widely studied; we refer to \cite{kodama17} for further background. The KP equation is the first equation in a system known as the \emph{KP hierarchy}. Sato \cite{sato81} discovered that solutions of the KP hierarchy correspond to points in an infinite-dimensional Grassmannian. We recall an algebraic version of this correspondence, following Carrell and Goulden \cite{carrell_goulden10} (cf.\ \cite{miwa_jimbo_date00}).

We identify a symmetric function $f \in \Lambda$ 
with the multiplication operator $f : \Lambda \to \Lambda$, $g \mapsto fg$.
The linear operator $f^\perp : \Lambda \to \Lambda$ is the adjoint of $f$ under $\langle \cdot, \cdot \rangle$,
and is called a \emph{skewing operator}.
Viewing $\Lambda = \CC[\p_1, \p_2, \p_3, \dots]$ as a polynomial algebra in the power sum generators,
$f^\perp$ is a differential operator 
with constant coefficients; for example, $\p_k^\perp = k\frac{\partial}{\partial p_k}$. 
Equations involving skewing operators are 
partial differential equations.

We consider the multiplication 
operators $\H(t)$, $\E(t)$ and their adjoints
$\H^\perp(t)$, $\E^\perp(t)$, defined as follows:
\begin{align*}
  \H(t) &:= \sum_{k \geq 0} \h_k t^k\,,
&
  \E(t) &:= \sum_{k \geq 0} \e_k t^k\,,
\\[4pt]
  \H^\perp(t) &:= \sum_{k \geq 0} \h_k^\perp t^k\,, 
&
  \E^\perp(t) &:= \sum_{k \geq 0} \e_k^\perp t^k
\,,
\end{align*}
where $t$ is a formal indeterminate.  Note that
\begin{equation}
\label{eq:E}
 \E(t) = \sum_{\substack{S\subseteq\ZZ_{>0},\\ S \text{ finite}}} t^{|S|} \cdot \prod_{b\in S} x_b,
\end{equation}
where the sum is taken over all finite subsets $S$ of the positive integers.
The operator $\H^\perp(t)$ has
a particularly nice alternate description:

\begin{proposition}[cf.\ {\cite[(14)]{bergeron_haiman13}}]
\label{prop:Hperp}
For every symmetric function $f = f(x_1, x_2, \dots ) \in \Lambda$, we have
\[
  \H^\perp(t)f = f(t, x_1, x_2, \dots)
\,.
\]
In particular, $\H^\perp(t)$ is a $\CC$-algebra homomorphism.
\end{proposition}

The \defn{Bernstein operator} $\B(t)$ and its adjoint $\B^\perp(t)$ are 
the following operators:
\[
  \B(t) := \H(t)\E^\perp(-t^{-1}) \qquad \text{and} \qquad
   \B^\perp(t) := \E(-t^{-1})\H^\perp(t)
\,.
\]
Let $\tau \in \widehat \Lambda$ be a series in the formal completion
of the algebra of symmetric functions.
Then $\B(t) \tau$ and $\B^\perp(t^{-1}) \tau$
are series with coefficients in the field $\fls$ of
formal Laurent series in $t^{-1}$, that is,
$\B(t) \tau,\,\B^\perp(t^{-1}) \tau \in \widehat \Lambda_{\fls}$.
We say that $\tau$ is a 
\defn{$\tau$-function of the KP hierarchy} if the following \emph{Hirota bilinear equation} holds:
\begin{equation}
\label{eq:KP}
    [t^{-1}]\, \big( \B(t) \tau \big) \otimes_\fls  \big( \B^\perp(t^{-1}) \tau \big) = 0
\,,
\end{equation}
where $[t^{-1}]$ extracts the coefficient of $t^{-1}$. We can interpret the elementary tensor $f\otimes g$ as the product $f(x_1, x_2, \dots)g(y_1, y_2, \dots)$ of $f$ and $g$ in two different sets of variables.

\begin{theorem}[Sato {\cite[Section 4]{sato81}; cf.\ \cite[Chapter 10]{miwa_jimbo_date00}}]
\label{thm:KPplucker}
For $\tau \in \widehat \Lambda$, the following are equivalent:
\begin{enumerate}[(a)]
\item the coefficients of $\tau$ in the Schur basis, 
$\Delta^\lambda = \langle \tau, \s_\lambda \rangle$,
satisfy the \Plucker relations \eqref{eq:prelspartitions};
\item $\tau$ is a $\tau$-function of the KP hierarchy.
\end{enumerate}
\end{theorem}

Hence, \eqref{eq:KP} is a convenient way to encode all 
of the \Plucker relations into a single generating function equation.
An alternate proof of \cref{thm:KPplucker} was given by Carrell and Goulden 
\cite[Theorem 5.2]{carrell_goulden10}, based an explicit formula for
$\B(t)\s_\lambda$.
We recall the relevant details of this argument:

\begin{proof}[Proof of \cref{thm:KPplucker}]
By \cite[Corollaries 3.5 and 3.6]{carrell_goulden10},
the left-hand side of~\eqref{eq:KP} is equal to 
$\sum_{\lambda, \mu} b_{\lambda,\mu} \s_\lambda \otimes \s_\mu$, 
where
\begin{equation}
\label{eq:KPlhs}
   b_{\lambda,\mu} = 
   \sum_{\substack{i,j \geq 1, \\ 
    |\lambda^{(-i)}| + |\mu^{(j)}| = |\lambda| + |\mu|+ 1}}
    (-1)^{|\mu|-|\mu^{(j)}|+i+j}
      \Delta^{\lambda^{(-i)}} \Delta^{\mu^{(j)}}  
\,.
\end{equation}
Thus \eqref{eq:KP} holds if and only if $b_{\lambda,\mu} = 0$
for all $\lambda$ and $\mu$, i.e., if and only
if \eqref{eq:prelspartitions} holds.
\end{proof}

We will need a similar equation which encodes only the single-column
\Plucker relations. These turn out to correspond to the 
constant term in the first tensor factor of \eqref{eq:KP}:

\begin{lemma}
\label{lem:singlecolumneqs}
For $\tau \in \widehat \Lambda$, the following are equivalent:
\begin{enumerate}[(a)]
\item the coefficients of $\tau$ in the Schur basis, 
$\Delta^\lambda = \langle \tau, \s_\lambda \rangle$, satisfy all
the single-column \Plucker relations;
\item $\tau$ satisfies the equation
\begin{equation}
\label{eq:scKP}
    [t^{-1}]\, \langle \B(t) \tau , 1 \rangle \cdot \B^\perp(t^{-1}) \tau  = 0
\,.
\end{equation}
\end{enumerate}
\end{lemma}

\begin{proof}
In the notation of \eqref{eq:KPlhs},
the single-column \Plucker relations are the 
equations $b_{0,\mu} =0$.  Since 
$\sum_{\lambda, \mu} b_{\lambda,\mu} \s_\lambda \otimes \s_\mu$ is
the left-hand side of \eqref{eq:KP}, 
the left-hand side of \eqref{eq:scKP} is
$\sum_{\lambda, \mu} b_{\lambda, \mu} \langle 
\s_\lambda, 1 \rangle \cdot \s_\mu = 
\sum_{\mu} b_{0,\mu} \s_\mu$.
The result follows.
\end{proof}

We note that since $\H^\perp(t)1 = 1$, we have
\[
\langle \B(t) \tau , 1 \rangle = 
       \langle \tau , \B^\perp(t) 1 \rangle = 
       \langle \tau , \E(-t^{-1}) \rangle = 
       \sum_{k \geq 0} (-t)^{-k} \langle \tau, \e_k \rangle
\,,
\]
and thus \eqref{eq:scKP} can be rewritten as
\begin{equation}
\label{eq:scKP2}
    \Bresidue \sum_{k \geq 0} (-t)^{-k}
           \langle \tau, \e_k \rangle \tau  = 0
\,.
\end{equation}

\begin{remark}
\label{rmk:taufunction}
Let $V\in\scellnu \subseteq \Gr(d,m)$ have the normalized \Plucker coordinates $(\Delta^\lambda : \lambda \subseteq \nu)$. By \cref{thm:KPplucker}, we have the associated $\tau$-function of the KP hierarchy
\[
\tau := \sum_{\lambda \subseteq \nu}\Delta^\lambda \s_\lambda\in\Lambda.
\]
Then the Wronskian of $V$ is the \defn{exponential specialization} of $\tau$. Namely, define the unital $\CC$-algebra homomorphism $\exspec : \Lambda \to \CC[u]$ as follows \cite[Section 7.8]{stanley24}:
\[
\exspec(\p_1) := u \qquad \text{and} \qquad \exspec(\p_k) := 0 \;\text{ for all } k\ge 2\,.
\]
Then $\exspec(\s_\lambda) = \frac{\numsyt{\lambda}}{|\lambda|!}u^{|\lambda|}$ (cf.\ \cite[Section 7.16]{stanley24}), so $\exspec(\tau) = \Wr(V)$ by \eqref{eq:pluckerwronskian}.
\end{remark}

\subsubsection{Scaled monomial symmetric functions}

The \defn{scaled monomial symmetric functions} (cf.\ \cite{merca15})
are defined to be
$\sm_\lambda := \langle \p_\lambda, \h_\lambda \rangle \m_\lambda$,
i.e., $\sm_\lambda$ is the $\m_\lambda$-term in the
monomial expansion of $\p_\lambda$.  For example, 
$\sm_{1^{n}} = n! \, \m_{1^n}$.
Equivalently, if $\ell(\lambda) = k$,
then
\begin{equation}
\label{eq:sm}
   \sm_\lambda = \sum_{\substack{a_1, \dots, a_k \geq 1\\ \text{(distinct)}}}
      x_{a_1}^{\lambda_1} \dotsm x_{a_k}^{\lambda_k}
\,,
\end{equation}
where the sum is taken over all $k$-tuples of pairwise 
distinct positive integers.

We will sometimes index the scaled monomial symmetric functions using 
compositions instead of partitions.  A \defn{composition}
$\kappa = (\kappa_1, \dots, \kappa_k)$ is an arbitrary tuple of
positive integers.  As with partitions, we write 
$|\kappa| = \kappa_1 + \dots + \kappa_k$, $\ell(\kappa) = k$,
and $\kappa_j = 0$ for $j > k$. 
We write $\kappa \leq \lambda$ to mean $\ell(\kappa) = \ell(\lambda)$
and $\kappa_i \leq \lambda_i$ for all $i$.
Let $\sort(\kappa)$
denote the unique partition 
of the form $(\kappa_{\sigma(1)}, \dots, \kappa_{\sigma(k)})$ for some 
$\sigma \in \symgrp{k}$.  We put $\sm_\kappa := \sm_{\sort(\kappa)}$ for all compositions $\kappa$. Also, if $I \subseteq [k]$, we let $\kappa \setminus \kappa_I$ denote the composition obtained from $\kappa$ by deleting the parts $\kappa_i$ for $i\in I$.

For some purposes, the scaled monomial basis is more natural
than the monomial basis.
For example, we can express $\H^\perp(t)\sm_\kappa$ as follows:
\begin{proposition}
\label{prop:hperpsm}
For any composition $\kappa$ of length $k$, we have
\[
\H^\perp(t) \sm_\kappa = \sm_\kappa + \sum_{i=1}^k t^{\kappa_i}\sm_{\kappa\setminus\kappa_i}\,.
\]
\end{proposition}

\begin{proof}
This follows from \cref{prop:Hperp} by a direct calculation.
\end{proof}

We will also need the power sum expansion of $\sm_\kappa$:

\begin{proposition}[Doubilet {\cite[Theorem 2(i)]{doubilet72}}; cf.\ {\cite[Theorem 2]{merca15}}]
\label{prop:scaledmonomials}
Let $\kappa$ be a composition of length $k$.
For $\sigma \in \symgrp{k}$, 
let $\kappa[\sigma]$ denote the partition
obtained by amalgamating the parts of $\kappa$ that are 
in the same cycle of $\sigma$.
That is, for each cycle $(i_1 \;\; i_2 \; \cdots \; i_s)$ of $\sigma$, we produce
a part of $\kappa[\sigma]$ of size $\kappa_{i_1} + \dots + \kappa_{i_s}$.
Then
\[
    \sm_\kappa = 
    \sum_{\sigma \in \symgrp{k}} \sgn(\sigma) \p_{\kappa[\sigma]}
\,.
\]
\end{proposition}

\subsubsection{Two symmetric-function identities}

We record two identities which
will help us understand the kernel of the operator 
$\Bresidue: \Lambda_\fls \to \Lambda$.
The first identity 
describes the commutation relation between $\B^\perp(t^{-1})$
and the multiplication operator $(1-t^j\p_j)$:

\begin{lemma}
\label{lem:Bp-identity}
For $j \geq 0$, we have 
\[
   \big[(1-t^j\p_j), \B^\perp(t^{-1})\big] = \B^\perp(t^{-1})
\,.
\]  
Equivalently, for any symmetric function $f \in \Lambda$,
\[
     \B^\perp(t^{-1}) (1-t^j \p_j) f
     = - \p_j \cdot \B^\perp(t^{-1}) t^j f
\,.
\]
\end{lemma}

\begin{proof}
By definition, $\B^\perp(t^{-1}) = \E(-t)\H^\perp(t^{-1})$.  Now $\E(-t)$ is a
multiplication operator, and $\H^\perp(t^{-1})$
is a homomorphism, so we have the following general identity:
for all $f,g \in \Lambda$,
\[
\B^\perp(t^{-1}) fg = \big(\B^\perp(t^{-1}) f\big) \cdot 
\big(\H^\perp(t^{-1}) g\big)
\,.
\]
If $g = 1-t^j\p_j$, then by \cref{prop:Hperp}, 
$\H^\perp(t^{-1})g = 1- t^j(t^{-j} + \p_j) = -t^j\p_j$,
so the identity above yields the result.
\end{proof}

Hence for $\phi(t) \in \widehat \Lambda_{\fls}$, we have
$\phi(t) (1-t^j \p_j) \in \ker \Bresidue$ if and only
if $t^j \phi(t) \in \ker \Bresidue$.
The second identity gives an explicit 
family of symmetric functions 
which are in the kernel
of $\Bresidue$:
\begin{lemma}
\label{lem:Bm-identity}
For any composition $\lambda$, consider the
symmetric function
\[
   \phi_\lambda(t) := \sum_{\kappa \leq \lambda} 
   t^{|\kappa| - |\lambda|- \ell(\lambda)} \sm_\kappa
\,,
\]
where the sum is taken over all compositions $\kappa \leq \lambda$.
Then $\Bresidue \phi_\lambda(t) = 0$.
\end{lemma}

\begin{example}
If $\lambda = (3,2)$, the compositions $\kappa \leq \lambda$ are
   $(3,2)$, $(3,1)$, $(2,2)$, $(2,1)$, $(1,2)$, and $(1,1)$.
In this case, \cref{lem:Bm-identity} asserts that
\[
     \Bresidue 
   \big( t^{-2} \sm_{32} + t^{-3} \sm_{31} + t^{-3} \sm_{22} 
    + 2t^{-4} \sm_{21} + t^{-5} \sm_{11} \big) = 0
\,.
\]
\end{example}

\begin{proof}[Proof of \cref{lem:Bm-identity}]
By \eqref{eq:E} and \eqref{eq:sm}, we get that $\E(-t) \phi_\lambda(t)$ equals
\begin{equation}
\label{eq:Ephi0}
   \Bigg(\sum_{\substack{S\subseteq\ZZ_{>0},\\ S \text{ finite}}} (-t)^{|S|} \cdot \prod_{b\in S} x_b\Bigg)
    \Bigg(\sum_{\kappa \leq \lambda} 
       t^{|\kappa| - |\lambda|- \ell(\lambda)}\bigg(
       \sum_{\substack{a_1, \dots, a_{\ell(\lambda)} \geq 1\\ \text{(distinct)}}}
             x_{a_1}^{\kappa_1} \dotsm x_{a_{\ell(\lambda)}}^{\smash{\kappa_{\ell(\lambda)}}}\bigg)\Bigg)
\,.
\end{equation}
We regard \eqref{eq:Ephi0} as a single sum indexed by the appropriate set $\calT_\lambda$ of triples $(S, \kappa, \bolda)$, where $\bolda$ denotes the $\ell(\lambda)$-tuple $(a_1, \dots, a_{\ell(\lambda)})$. Let $\calT'_\lambda \subseteq \calT_\lambda$ be the subset of triples $(S, \kappa, \bolda)$ such that for all $1 \le i \le \ell(\lambda)$, we have either
\begin{equation}\label{Tprimecondition}
\kappa_i=1 \text{ and } a_i\notin S, \quad\text{or}\quad \kappa_i=\lambda_i \text{ and } a_i\in S\,.
\end{equation}

We claim that the sum \eqref{eq:Ephi0} remains unchanged if we sum over $\calT'_\lambda$ instead of $\calT_\lambda$. To prove this, we construct a sign-reversing involution on the terms in the sum indexed by $\calT_\lambda\setminus\calT'_\lambda$. Given $(S, \kappa, \bolda)\in\calT_\lambda\setminus\calT'_\lambda$, let $j$ be the smallest positive integer such \eqref{Tprimecondition} fails to hold at $i=j$. If $a_j\notin S$, then $\kappa_j > 1$; we set $S' = S\cup\{a_j\}$, and let $\kappa'$ be obtained from $\kappa$ by decreasing $\kappa_j$ by $1$. If $a_j\in S$, then $\kappa_j < \lambda_j$; we let $S' = S\setminus\{a_j\}$, and let $\kappa'$ be obtained from $\kappa$ by increasing $\kappa_j$ by $1$. Then $(S, \kappa, \bolda) \mapsto (S', \kappa', \bolda)$ is our desired sign-reversing involution. (For example, if $\lambda = (3,2)$, the involution swaps $(\{2,4,5,7\}, (2,1), (5,3))$ with $(\{2,4,7\}, (3,1), (5,3))$.) Thus
\begin{equation}
\label{eq:Ephi}
   \E(-t) \phi_\lambda(t)
   =
    \sum_{(S, \kappa, \bolda) \in \calT'_\lambda} 
      (-1)^{|S|}t^{|S|+|\kappa|-|\lambda|-\ell(\lambda)}
     x_{a_1}^{\kappa_1} \dotsm 
     x_{a_{\ell(\lambda)}}^{\smash{\kappa_{\ell(\lambda)}}} \cdot \prod_{b\in S} x_b
\,.
\end{equation}

We now reindex the sum in \eqref{eq:Ephi}. Given $(S, \kappa, \bolda) \in \calT'_\lambda$, let $J$ denote the set of $j\in [\ell(\lambda)]$ such that $\kappa_j = \lambda_j$ and $a_j\in S$.  Let $J^c = [\ell(\lambda)]\setminus J$, so that by \eqref{Tprimecondition}, we have $\kappa_i = 1$ and $a_i\notin S$ for all $i\in J^c$. Writing $J = (j_1, \dots, j_s)$ (where $s = |J|$, $j_1< \dots < j_s$), and $\bolda_J = (a_{j_1}, \dots, a_{j_s})$, we can express the term in \eqref{eq:Ephi} indexed by $(S, \kappa, \bolda)$ as
\begin{equation}\label{eq:Ephimid1}
(-1)^{s}x_{a_{j_1}}^{\lambda_{j_1}+1} \dotsm x_{a_{j_s}}^{\lambda_{j_s}+1} \cdot (-1)^{|S|-s}t^{|S|+|\kappa|-|\lambda|-\ell(\lambda)} \cdot \prod_{i\in J^c} x_{a_i} \cdot \prod_{b\in S\setminus \bolda_J} x_b\,.
\end{equation}
Set $\mu = \lambda\setminus\lambda_J$, so that $\ell(\mu) = \ell(\lambda) - s$ and $|\mu| = |\lambda| - |\kappa| + \ell(\mu)$. Also set $r = |S| + \ell(\lambda) - 2s$, so that
\begin{equation}\label{eq:Ephimid2}
(-1)^{|S|-s}t^{|S|+|\kappa|-|\lambda|-\ell(\lambda)} = (-1)^{r-\ell(\mu)}t^{r-|\mu|-\ell(\mu)}\,.
\end{equation}

Now let us fix $J$, $\bolda_J$, and $r$, and sum \eqref{eq:Ephimid1} over all $(S, \kappa, \bolda)\in\calT'_\lambda$ consistent with our choice of $J$, $\bolda_J$, and $r$; we claim the result is
\begin{equation}\label{eq:Ephimid3}
(-1)^{s}x_{a_{j_1}}^{\lambda_{j_1}+1} \dotsm x_{a_{j_s}}^{\lambda_{j_s}+1} \cdot (-1)^{r-\ell(\mu)}t^{r-|\mu|-\ell(\mu)} r(r-1) \cdots (r-\ell(\mu)+1) \,\e_r(\boldx\setminus\boldx_{\bolda_J})\,,
\end{equation}
where $\boldx\setminus\boldx_{\bolda_J}$ means we omit the variables $x_{a_{j_1}}, \dots, x_{a_{j_s}}$. To see this, note that $\prod_{i\in J^c} x_{a_i} \cdot \prod_{b\in S\setminus\bolda_J} x_b$ is an arbitrary square-free monomial of degree $(\ell(\lambda)-s) + (|S|-s) = r$ in the variables $\boldx\setminus\boldx_{\bolda_J}$. The number of times each such monomial appears in the sum is the number of ways to choose (in order) the $\ell(\lambda)-s = \ell(\mu)$ variables appearing in the first product, which is just $r(r-1)\cdots(r-\ell(\mu)+1)$. This, along with \eqref{eq:Ephimid2}, proves the claim.

Now for any composition $\mu$, we define
\begin{equation}\label{eq:Flambda}
\begin{aligned}
   F_\mu(t; \boldx) 
   &:= \sum_{r \ge 0}
      (-1)^{r-\ell(\mu)}t^{r-|\mu|-\ell(\mu)}r(r-1) \cdots (r-\ell(\mu)+1)\,\e_r(\boldx)\,, \\
   G_\mu(t; \boldx)
   &:= F_\mu(t;\boldx) + \sum_{i=1}^{\ell(\mu)} 
     \mu_i t^{-\mu_i-1} F_{\mu \setminus \mu_i}(t;\boldx)\,.
\end{aligned}
\end{equation}
Then summing \eqref{eq:Ephimid3} over all possibilities for $J$, $\bolda_J$, and $r$, \eqref{eq:Ephi} becomes
\begin{equation}
\label{eq:Ephi2}
     \E(-t)\phi_\lambda(t) = 
     \sum_{\substack{
     J = (j_1, \dots, j_s),
      \\ \boldb = (b_1, \dots, b_s)} }
      (-1)^s x_{b_1}^{\lambda_{j_1}+1} \dotsm 
      x_{b_s}^{\lambda_{j_s}+1} \cdot
      F_{\lambda \setminus \lambda_J}(t;\boldx \setminus \boldx_\boldb)
\,,
\end{equation}
where the sum is taken over all subsets $J \subseteq [\ell(\lambda)]$ and all $|J|$-tuples $\boldb$ of distinct positive integers. (In going from  \eqref{eq:Ephimid3} to \eqref{eq:Ephi2}, we set $(b_1, \dots, b_s) = (a_{j_1}, \dots, a_{j_s})$.)

Using \cref{prop:hperpsm}, we have
   $\H^\perp(t^{-1}) \phi_\lambda(t) 
    = \phi_\lambda(t) 
      + \sum_{i=1}^{\ell(\lambda)} 
      \lambda_i t^{-\lambda_i-1} \phi_{\lambda \setminus \lambda_i}(t)
\,,$
and hence
$\B^\perp(t^{-1}) \phi_\lambda(t) = 
    \E(-t)\phi_\lambda(t) + \sum_{i=1}^{\ell(\lambda)}\lambda_it^{-\lambda_i-1}\E(-t)\phi_{\lambda\setminus\lambda_i}(t)\,.$
Combining this with \eqref{eq:Ephi2} gives
\[
   \B^\perp(t^{-1}) \phi_\lambda(t) 
    =  \sum_{\substack{
     J = (j_1, \dots, j_s),
      \\ \boldb = (b_1, \dots, b_s)} }
      (-1)^s x_{b_1}^{\lambda_{j_1}+1} \dotsm 
      x_{b_s}^{\lambda_{j_s}+1}   \cdot
      G_{\lambda \setminus \lambda_J}(t;\boldx \setminus \boldx_\boldb)
\,.
\]
Hence we are done if we can show that $[t^{-1}] G_\mu(t; \boldx) = 0$ for all $\mu$. To see this, we calculate by \eqref{eq:Flambda} that $[t^{-1}]G_\mu(t; \boldx)$ equals $0$ if $\mu=0$, and if $\mu\neq 0$ it equals
\begin{multline*}
(-1)^{|\mu|-1}(|\mu|+\ell(\mu)-1)(|\mu|+\ell(\mu)-2) \cdots|\mu|\,\e_{|\mu|+\ell(\mu)-1}(\boldx) \\
+ \sum_{i=1}^{\ell(\mu)}\mu_i \cdot (-1)^{|\mu|}(|\mu|+\ell(\mu)-1)(|\mu|+\ell(\mu)-2) \cdots (|\mu|+1)\,\e_{|\mu|+\ell(\mu)-1}(\boldx)\,.
\end{multline*}
After factoring out $(-1)^{|\mu|-1}(|\mu|+\ell(\mu)-1)(|\mu|+\ell(\mu)-2) \cdots(|\mu|+1)\,\e_{|\mu|+\ell(\mu)-1}(\boldx)$, the expression above becomes $|\mu| - \sum_{i=1}^{\ell(\mu)}\mu_i$, which equals zero. This finishes the proof.
\end{proof}


\subsection{The Bethe subalgebra of \texorpdfstring{$\CC[\Sn]$}{C[S\textunderscore{}n]}}
\label{sec:bethe}

Fix a positive integer $n$, and let $z_1, \dots, z_n \in \CC$ be
complex numbers.  As before, let $g(u) = (u+z_1) \dotsm (u+z_n)$.
For a subset $X \subseteq [n]$, let $z_X := \prod_{i \in X} z_i$.
We recall the operators $\beta^\lambda(t)$ defined in \eqref{eq:betadef} and $\alpha^\lambda_X$ defined in \eqref{eq:alphadef}. By \eqref{eq:alphatobeta}, we have
\[
     \beta^\lambda(t) = \sum_{\ell=0}^{n-|\lambda|} 
       \beta^\lambda_\ell t^{n-|\lambda|-\ell},
\qquad\text{where }\;\beta^\lambda_{\ell} := 
    \sum_{\substack{X \cap Y = \emptyset, \\
     |X| = |\lambda|, |Y| = \ell}}
    \alpha^\lambda_X z_Y
 \in \CC[\Sn] \;\text{ for } \ell\ge 0
\,.
\]
In particular, $\beta^\lambda = \beta^\lambda(0) = \beta^\lambda_{n-|\lambda|}$.

\subsubsection{Generators}

While most of the operators $\beta^\lambda(t)$ are new, the operators
$\betaminus{k}(t)$ coincide with those introduced by 
Mukhin, Tarasov, and Varchenko in \cite{mukhin_tarasov_varchenko13}.
In the remainder of this section we will only be 
concerned with the 
partitions $\lambda = 1^k$, $k \leq n$.
In this case, for $|X| = k$ the Specht module $\specht{1^k}_X$ is the 
one-dimensional sign representation of $\symgrp{X}$, so 
\[
     \alphaminus{k}_X = \sum_{\sigma \in \symgrp{X}} \sgn(\sigma) \sigma
\,.
\]
Mukhin, Tarasov, and Varchenko define the \defn{Bethe subalgebra} (of \emph{Gaudin type}) of 
$\CC[\Sn]$ to be the subalgebra 
$\bethe \subseteq \CC[\Sn]$ 
generated by the elements $\betaminus{k}_\ell$ for $k, \ell \geq 0$.
Equivalently, $\bethe$ is generated by the elements
$\betaminus{k}(t)$ for $k \geq 0$ and $t \in \CC$.
\begin{remark}
\label{rmk:bethenotation}
In the notation of \cite{mukhin_tarasov_varchenko13}, what we call $\betaminus{k}(t)$ and $\betaminus{k}_\ell$ can be expressed as, respectively, $(-1)^{n-k}\mathit{\Phi}^{[n]}_k(-t)$ and $(-1)^\ell\mathit{\Phi}^{[n]}_{k,\ell}$. Our notation more closely follows \cite{purbhoo}, where our $\betaminus{k}(t)$ and $\betaminus{k}_\ell$ are denoted $\beta^-_{k,n-k}(t)$ and $\beta^-_{k,\ell}$, respectively.
\end{remark}

\begin{remark}\label{rmk:deformation}
When $z_1, \dots, z_n$ are distinct, the \defn{Gaudin elements},
studied in \cite[Section 2]{gaudin76},
are the elements $\zeta_1, \dots \zeta_n \in \bethe$ defined by
\[
\zeta_k := 
\sum_{\substack{1 \le j \le n, \\ j \neq k}}\frac{\trans{j}{k}}{z_k - z_j} 
\,
=
\bigg(\prod_{\substack{1 \le j \le n, \\ j \neq k}}\frac{1}{z_j - z_k}\bigg)\,
\beta^{1^2}(-z_k) 
\;
-
\bigg(\sum_{\substack{1 \le j \le n, \\ j \neq k}}\frac{1}{z_j - z_k}\bigg)\, \Sidentity{n}
\,.
\]
In the limit
\begin{equation}
\label{eq:gtlimit}
\frac{z_{k-1} - z_k}{z_k - z_{k+1}} \to 0 \quad \text{for $2 \le k \le n-1$}\,,
\end{equation}
we can check that $\zeta_k$ (appropriately rescaled) converges to the 
\emph{Jucys--Murphy element}
\[
\sum_{j=1}^{k-1}\trans{j}{k}\in\CC[\Sn]\,.
\]
The subalgebra of $\CC[\Sn]$ generated by the Jucys--Murphy elements for $k=1, \dots, n$ is the well-known \defn{Gelfand--Tsetlin subalgebra} $\GT$, which is semisimple and a maximal commutative subalgebra of $\CC[\Sn]$; see \cite[Sections 1 and 2]{okounkov_vershik96} for details.

Analogously, it is known that Gaudin elements generate $\bethe$ 
for $z_1, \dots, z_n$ distinct;
see \cite[Theorem 4.4]{mukhin_tarasov_varchenko13} 
and cf.\ \cite[Section 3.3]{kirillov16a}.  However, we will not use this fact.
\end{remark}

\begin{theorem}[{Mukhin, Tarasov, and Varchenko \cite{mukhin_tarasov_varchenko13}}]
\label{thm:bethebasics}
The algebra $\bethe$ has the
following properties.
\begin{enumerate}[(i)]
\item\label{bethebasics1}
It is a commutative algebra.
\item\label{bethebasics2}
If $z_1, \dots, z_n \in \CC$ are generic,
it is semisimple and a maximal commutative subalgebra of
$\CC[\Sn]$, and has dimension
$\sum_{\nu \vdash n} \numsyt{\nu}$.
\item\label{bethebasics3}
It contains
the centre of $\CC[\Sn]$.
\item\label{bethebasics4}
It is translation invariant: $\bethe = \translatebethe{t}$ for
all $t \in \CC$.
\end{enumerate}
\end{theorem}

We point out that while \cref{thm:bethebasics} is non-trivial, its proof 
does not require the full machinery of \cite{mukhin_tarasov_varchenko13}; 
in particular it can be proved without using \cref{thm:precise}.  
The proof below gives
an outline of some of the different possible arguments,
their interdependencies, and more detailed references.

\begin{proof}
Part \ref{bethebasics1} is proved in 
\cite[Proposition 2.4]{mukhin_tarasov_varchenko13}, by deducing the result
from a more general commutativity statement in the rational Cherednik algebra.
The result also follows from 
\cite[Theorems 3.1 and 3.2]{mukhin_tarasov_varchenko13}, which shows
that the definition of $\bethe$ is equivalent to a previous definition
formulated in terms of $\gl_n$.
A combinatorial proof in terms of permutations 
is given in \cite[Section 5]{purbhoo}.

For part \ref{bethebasics2}, we use 
\cite[Proposition 2.5]{mukhin_tarasov_varchenko13}, which asserts that
the Bethe algebra $\bethe$ has a degeneration to the Gelfand--Tsetlin algebra
$\GT$.  We note that the proof therein depends on part \ref{bethebasics3}.  
A modification of this proof that does not use \ref{bethebasics3} is
as follows.  Consider the limit of $\bethe$ under \eqref{eq:gtlimit}.
Since the limits of the Gaudin elements generate $\GT$
(see \cref{rmk:deformation}), this limit 
contains $\GT$.  But since $\bethe$ is commutative and $\GT$ is
maximal commutative, the limit is precisely $\GT$.
Now, the set of semisimple maximal commutative
subalgebras is (Zariski)-open 
in the space of commutative subalgebras of $\CC[\Sn]$.
Since $\GT$ is itself a semisimple
maximal commutative subalgebra of $\CC[\Sn]$, 
and $\dim \GT = \sum_{\nu \vdash n} \numsyt{\nu}$,
the existence of this degeneration implies \ref{bethebasics2}.
We point out that \ref{bethebasics2} also follows from 
\cite[Theorem 4.3 parts (i) and (iii)]{mukhin_tarasov_varchenko13}, but 
following this reference would create a circular argument 
in our proof of \cref{thm:precise}.
We also emphasize that using \cite[Theorem 4.3]{mukhin_tarasov_varchenko13} is excessive: it is stated under the hypothesis that $z_1, \dots, z_n$ are distinct, whereas we only need the results of parts (i) and (iii) for \emph{generic} $z_1, \dots, z_n$ (which is much easier to prove).

Part \ref{bethebasics3} follows from \cite[Proposition 2.1]{mukhin_tarasov_varchenko13}.  
The
proof uses \cite[Proposition 3.5]{mukhin_tarasov_varchenko13}, which is
in turn based on results in \cite[Section 2]{mukhin_tarasov_varchenko09b}.
A self-contained exposition of this argument is given in 
\cite[Section 6]{purbhoo}.  We remark that it is also possible to 
prove \ref{bethebasics3} in the course of proving \cref{thm:main}, as
follows.  
First note that for all $z_1, \dots, z_n \in \CC$, we have 
$\bethezero \subseteq \bethe$, since the generators 
of $\bethezero$ are all either zero or independent of $z_1, \dots, z_n$.
Let $p(n)$ denote the number of partitions of $n$.
By \cref{lem:FDOcoords} and \cref{cor:identifyscell},
$\Spec \bethe$ 
has at least one point $V_E$ (given in FDO coordinates)
in each Schubert cell $\scell\nu$ for $\nu \vdash n$; in particular, $\dim \bethe \geq p(n)$ for
all $z_1, \dots, z_n \in \CC$.
On the other hand, $\bethezero$ is contained in the centre of $\CC[\Sn]$, 
which has dimension $p(n)$.
Therefore $\bethezero$ equals the centre of $\CC[\Sn]$, which proves
\ref{bethebasics3}.
We note that the proofs of \cref{lem:FDOcoords} and \cref{cor:identifyscell}
do not rely on \ref{bethebasics3}, but they do rely on \ref{bethebasics2}.
Therefore, in order for this proof of \ref{bethebasics3} to be non-circular, 
we must use the modified proof of 
\cite[Proposition 2.5]{mukhin_tarasov_varchenko13} from the previous 
paragraph.

Finally, part \ref{bethebasics4} is stated 
in \cite[Lemma 2.3]{mukhin_tarasov_varchenko13}, but also follows 
immediately from the definition of $\bethe$.
\end{proof}

\begin{remark}
\label{rmk:parameterdependence}
Many of the properties of $\bethe$ depend on the 
choice of parameters $z_1, \dots, z_n$ (cf.\ \cite{mukhin_tarasov_varchenko13}).
For example, $\bethe$ is not always semisimple, and its dimension
depends discontinuously on $z_1, \dots, z_n$.
When $z_1, \dots, z_n$ are distinct, 
$\dim \bethe = \sum_{\nu \vdash n} \numsyt\nu$, even if $\bethe$
is not semisimple.  However, $\dim \bethe$ can also be
strictly smaller.
For example, as discussed in the proof of \cref{thm:bethebasics},
$\bethezero$ is the centre of $\CC[\Sn]$.
We give a general formula for $\dim \bethe$ in \cref{thm:bethedimensions}.
\end{remark}

If $u$ is a formal indeterminate, then 
$\betaminus{k}(u) \in \CC[\Sn] \otimes \CC[u]$ is an element of the group algebra
of $\Sn$ with coefficients in $\CC[u]$, or equivalently, a polynomial
in $u$ with coefficients in $\CC[\Sn]$.
We combine the elements $\betaminus{k}(u)$
to produce a linear differential operator,
$\calD_n : \CC[\Sn] \otimes \CC(u) \to \CC[\Sn] \otimes \CC(u)$, defined as
\begin{equation}
\label{eq:diffop}
    \calD_n := 
\frac{1}{g(u)} 
    \sum_{k=0}^n (-1)^k \betaminus{k}(u) 
    \du^{n-k}
\,.
\end{equation}

\subsubsection{Eigenspaces}

For any partition $\nu \vdash n$, $\bethe$ acts on the Specht module
$\spechtnu$.  The subalgebra $\bethenu \subseteq \End(\specht{\nu})$ 
is the algebra defined by the image of this action.  We let
$\betaminus{k}_{\ell,\nu},\, \betaminus{k}_\nu(t) \in \bethenu$
denote the image of the generators 
$\betaminus{k}_\ell,\,\betaminus{k}(t)$.
Since $\bethe$ is commutative, so is $\bethenu$.  
Because $\bethe$ contains the centre of $\CC[\Sn]$, which includes
projections onto each $\spechtnu$ (by \cref{prop:projection}), we have a direct product decomposition
\begin{equation}
\label{eq:directproductbethe}
  \bethe \simeq \prod_{\nu \vdash n} \bethenu
\,.
\end{equation}

An \defn{eigenspace}
of $\bethenu$ is a maximal subspace $E \subseteq \specht{\nu}$ such
that every operator $\xi \in \bethenu$ acts as a scalar
$\xi_E$ on $E$; $\xi_E$ is the \defn{eigenvalue} of $\xi$ on $E$.
The eigenspaces of $\bethenu$ correspond naturally to the points of
$\Spec \bethenu$.
The \defn{generalized eigenspace} containing $E$ is the maximal 
$\bethenu$-submodule $\widehat{E} \subseteq \spechtnu$ on which
$\xi-\xi_E$ acts 
nilpotently for all $\xi \in \bethenu$.

For any such eigenspace $E$, we write
$\betaminus{k}_{\ell,E}$ and $\betaminus{k}_E(t)$ for the
eigenvalues of 
$\betaminus{k}_{\ell,\nu}$ and $\betaminus{k}_\nu(t)$, respectively.
Restricting the differential operator $\calD_n$ to $E$, we obtain a 
scalar-valued differential operator
$\calD_E : \CC(u) \to \CC(u)$, given by
\[
    \calD_E = 
\frac{1}{g(u)} 
    \sum_{k=0}^n (-1)^k \betaminus{k}_E(u) 
    \du^{n-k} 
\,.
\]

Recall that for a finite-dimensional vector space $V \subseteq \CC[u]$,
$D_V$ denotes the fundamental differential operator of $V$.
We now state \cref{thm:vague} more precisely (for simplicity,
we work inside the
Grassmannian $\Gr(n, 2n)$):

\begin{theorem}[{Mukhin, Tarasov, and Varchenko \cite[Theorem 4.3(iv)]{mukhin_tarasov_varchenko13}}]
\label{thm:precise}
Let $\nu \vdash n$, and consider the Schubert cell 
$\calX^\nu \subseteq \Gr(n,2n)$.  For every eigenspace 
$E \subseteq \specht{\nu}$ of $\bethenu$, there exists a unique
point $V_E \in \Wr^{-1}(g)$ such that $\calD_E = D_{V_E}$.
This defines
a bijective correspondence between the eigenspaces of $\bethenu$ and the points of $\Wr^{-1}(g)$.
\end{theorem}

\begin{remark}
One advantage of working in $\Gr(n,2n)$ is that 
$\calD_E$ and $D_V$ are both differential operators of order $n$.
If $d \neq n$, and we instead regard $\calX^\nu$ as a Schubert cell
in $\Gr(d,m)$, then differential operators $\calD_E$ and $D_{V_E}$
are no longer of the same order, and therefore cannot be equal.
Nevertheless, \cref{thm:precise} is true more generally
for $\Gr(d,m)$, 
namely, $\calD_E\du^d = D_{V_E} \du^n$. This formulation, however, is awkward to 
work with.
\end{remark}

\subsubsection{Regarding the parameters \texorpdfstring{$z_1, \dots, z_n$}{z\textunderscore{}1, ..., z\textunderscore{}n}}

In the definition of $\bethe$, the parameters $z_1,\dots, z_n$ are
arbitrary fixed complex numbers; specifically, $-z_1, \dots, -z_n$
are the roots the polynomial $g(u) \in \CC[u]$.   This is the
perspective assumed in the statement of \cref{thm:main}.
However, at various points in the proof of
\cref{thm:main}, we will take two other perspectives on the
parameters $z_1, \dots, z_n$, viewing them as generic complex numbers (i.e.\ $(z_1, \dots, z_n)$ is a general
point of $\CC^n$), or as formal indeterminates.

Any polynomial identity that we can prove under                          
any of these three perspectives must also be true under the other two.
This applies only to polynomial identities, and it is not true for 
abstract properties of Bethe algebras
(e.g.\ see \cref{rmk:parameterdependence}),
so some care is required when shifting perspectives.  
Nevertheless, we use this to our advantage.
For example, if
$z_1, \dots, z_n$ are assumed to be generic, then 
$\bethe$ is a maximal commutative subalgebra of $\CC[\Sn]$ 
by \cref{thm:bethebasics}\ref{bethebasics2}. 
Or, if $z_1, \dots, z_n$ are formal indeterminates,
then we can substitute $z_i \mapsto z_i+t$ into any valid equation.
These additional possibilities facilitate certain arguments that
would be invalid if $z_1, \dots, z_n \in \CC$ are always
assumed to be arbitrary and fixed.


\section{Commutativity relations and the translation identity}
\label{sec:commutativitytranslation}

The purpose of this section is to prove parts \ref{main_commutativity} and \ref{main_translation} of \cref{thm:main}, i.e., commutativity relations and the translation identity for the operators $\beta^\lambda(t)$. We begin by formulating a weak form of \cref{thm:main}\ref{main_generates} in \cref{sec:weakform}; we then prove it alongside the translation identity (\cref{sec:translation}) and the commutativity relations (\cref{sec:commutativity}).


\hidelinks
\subsection{A weak form of \texorpdfstring{\fullcref{thm:main}\ref{main_generates}}{Theorem \ref{thm:main}\ref{main_generates}}}
\restorelinks
\label{sec:weakform}
Let $\altbethe$ denote the subalgebra of $\CC[\Sn]$ generated by the
elements $\beta^\lambda$.  Part \ref{main_generates} of \cref{thm:main} asserts
that $\altbethe = \bethe$; however, we will not deduce this until the
end of the proof of \cref{thm:main}.  In the meantime, we will spend the remainder of this section proving the following slightly weaker result, which will be enough to get us through 
until then:

\begin{lemma}  
\label{lem:comparealgebras}
The algebra $\altbethe$ has the
following properties.
\begin{enumerate}[(i)]
\item\label{comparealgebras1}
It is a commutative algebra.
\item\label{comparealgebras2}
If $z_1, \dots, z_n \in \CC$ are generic, it equals $\bethe$.
\item\label{comparealgebras3}
It contains $\bethe$.
\item\label{comparealgebras4}
It is translation invariant: 
$\altbethe = \translatealtbethe{t}$
for all $t \in \CC$. 
\end{enumerate}
\end{lemma}

Note that for any fixed $t \in \CC$, 
$\translatealtbethe{t}$ is generated by the 
elements $\beta^\lambda(t)$.  Thus parts \ref{comparealgebras1} and \ref{comparealgebras4} of 
\cref{lem:comparealgebras} together imply \cref{thm:main}\ref{main_commutativity}.  

In particular, \cref{lem:comparealgebras} tells us 
that $\altbethe$ has all of the same basic
properties as $\bethe$, described in \cref{thm:bethebasics}.
Therefore, once the lemma is proved, we will begin to use the
same notation as we used for $\bethe$.
We write $\altbethenu \subseteq \End(\spechtnu)$
to denote the image of the action $\altbethe$ on $\spechtnu$.  
As a consequence of \cref{lem:comparealgebras}\ref{comparealgebras3}, we 
have the direct product decomposition
\begin{equation}
\label{eq:directproductaltbethe}
    \altbethe \simeq \prod_{\nu \vdash n} \altbethenu
\,.
\end{equation}
We write $\beta^\lambda_{\ell,\nu},\, \beta^\lambda_\nu(t) \in \altbethenu$
for the operators $\beta^\lambda_\ell,\, \beta^\lambda(t)$ acting
on $\spechtnu$.
For each eigenspace $E \subseteq \spechtnu$ of $\altbethenu$, we
write 
$\beta^\lambda_{\ell,E}$ and $\beta^\lambda_E(t)$ 
for the corresponding eigenvalues of 
$\beta^\lambda_{\ell,\nu}$ and $\beta^\lambda_\nu(t)$, respectively.


\subsection{Translation identity}
\label{sec:translation}

We first prove the translation identity, \cref{thm:main}\ref{main_translation}, 
which is required for all parts of \cref{lem:comparealgebras}. Recall the elements $\alpha^\lambda_X$ defined in \eqref{eq:alphadef}.
\begin{lemma}\label{lem:induced}
Let $k\le n$, let $\mu\vdash k$, and let $N$ denote the representation of $\Sn$ induced by the Specht module $\specht{\mu}$ of $\symgrp{k}$ (under the containment $\symgrp{k}\subseteq\Sn$). Let $\chi^N:\Sn\to\CC$ denote the character of $N$. Then
\begin{equation}\label{eq:induced}
\sum_{\sigma\in\Sn}\chi^N(\sigma)\sigma = (n-k)!\sum_{X\in\binom{[n]}{k}}\alpha^\mu_X = \sum_{\substack{\lambda\vdash n, \\ \lambda\supseteq \mu}}\numsyt{\lambda/\mu}\alpha^\lambda_{[n]}\,.
\end{equation}
\end{lemma}

\begin{proof}
By the branching rule \cite[Theorem 2.8.3]{sagan01}, we have
\[
\chi^N = \sum_{\substack{\lambda\vdash n, \\ \lambda\supseteq \mu}}\numsyt{\lambda/\mu}\chi^\lambda\,,
\]
so the first and third expressions in \eqref{eq:induced} are equal. On the other hand, by the definition of the induced representation \cite[Theorem 12]{serre98}, we have
\[
\chi^N(\sigma) = \sum_{\substack{\pi\in C, \\ \pi^{-1}\sigma\pi\in\symgrp{k}}}\chi^\mu(\pi^{-1}\sigma\pi) \quad \text{for all } \sigma\in\Sn\,,
\]
where $C\subseteq\Sn$ is any set of coset representatives of the quotient $\Sn/\symgrp{k}$.
Note that for any $\pi\in\Sn$, we have $\pi\symgrp{k}\pi^{-1} = \symgrp{\pi([k])}$. Conversely, for any $X\in\binom{[n]}{k}$, there exist precisely $(n-k)!$ elements $\pi\in C$ such that $\pi([k]) = X$. Hence
\[
\chi^N(\sigma) = (n-k)!\sum_{\substack{X\in\binom{[n]}{k},\\ \sigma\in\symgrp{X}}}\chi^\mu(\sigma)\,,
\]
where inside the sum, $\chi^\mu$ is a character of $\symgrp{X}$. Therefore the first and second expressions in \eqref{eq:induced} are equal.
\end{proof}

\begin{proof}[Proof of \cref{thm:main}\ref{main_translation}]
We regard $z_1, \dots, z_n$ as formal indeterminates.
By translating
$(z_1, \dots, z_n) \mapsto (z_1+s, \dots, z_n+s)$, it suffices to prove the identity \eqref{eq:translationidentity} when $s = 0$, i.e.,
\begin{equation}
\label{eq:translationzero}
   \beta^\mu(t) = 
     \sum_{\lambda \supseteq \mu} 
      \frac{\numsyt{\lambda/\mu}}{|\lambda/\mu|!} t^{|\lambda/\mu|} 
      \beta^\lambda
\,.
\end{equation}
Let $k := |\mu|$, and note that \eqref{eq:translationzero} is symmetric and square-free in $z_1, \dots, z_n$, and homogeneous of total degree $n-k$ in $z_1, \dots, z_n, t$. Therefore it suffices to prove that the coefficients of $z_{\ell+1}\cdots z_nt^{\ell-k}$ on both sides are equal for all $k \le \ell \le n$. This is the second equality of \eqref{eq:induced} (after replacing $n$ by $\ell$ and rescaling by $\frac{1}{(\ell-k)!}$).
\end{proof}

We now prove parts \ref{comparealgebras4} and \ref{comparealgebras3} of \cref{lem:comparealgebras}.
By definition, the algebra $\translatealtbethe{t}$ is generated by the 
elements $\beta^\lambda(t)$.
From \eqref{eq:translationzero},
we see that $\beta^\mu(t) \in \altbethe$ for all $\mu$.  Conversely, from \eqref{eq:translationidentity} we get
\[
   \beta^\mu = \beta^\mu(t-t) = 
     \sum_{\lambda \supseteq \mu} 
      \frac{\numsyt{\lambda/\mu}}{|\lambda/\mu|!} (-t)^{|\lambda/\mu|} 
      \beta^\lambda(t) 
\,,
\]
implying that $\beta^\mu \in \translatealtbethe{t}$.  This proves part \ref{comparealgebras4}.
In the special case where $\lambda = 1^k$, we have that 
$\betaminus{k}(t) \in \altbethe$.  Since these elements generate $\bethe$, this proves part \ref{comparealgebras3}.


\subsection{Commutativity}
\label{sec:commutativity}

For every partition $\mu$ and $\ell \geq 0$, we define elements 
$\varepsilon^\mu_\ell \in \altbethe$ by
\[
    \varepsilon^\mu_\ell := 
   \sum_{\substack{X \cap Y = \emptyset, \\ |X| = |\mu|, |Y| = \ell}}\;
    \sum_{\substack{\sigma \in \symgrp{X}, \\ \cyc(\sigma) = \mu}} 
     \sigma \, z_Y
\,.
\]
We write 
$\varepsilon^\mu(t) := \sum_{\ell = 0}^n \varepsilon^\mu_\ell t^{n-|\mu|-\ell}$,
and set $\varepsilon^\mu := \varepsilon^\mu(0) = \varepsilon^\mu_{n-|\mu|}$.
These are related to the elements $\beta^\lambda(t)$ by the following
change of basis formulas:
\begin{equation}
\label{eq:changeofbasis}
     \beta^\lambda(t)
     = \sum_{\mu \vdash |\lambda|} \chi^\lambda_\mu \varepsilon^\mu(t)
\qquad \text{and} \qquad
     \varepsilon^\mu(t)
     = \sum_{\lambda \vdash |\mu|} \frac{\chi^\lambda_\mu}{\z_\mu} 
        \beta^\lambda(t)
\,.
\end{equation}
The first formula follows by definition, whence the second formula follows from \eqref{eq:spchangeofbasis}. In particular, the elements $\varepsilon^\mu$ are an alternate set of
generators for $\altbethe$.

The $\varepsilon^\mu$'s are combinatorially simpler than the $\beta^\lambda$'s, in that they have fewer
terms and do not involve characters.
However, even working with this new set of generators, it is difficult 
to prove the commutativity relations 
$\varepsilon^\lambda \varepsilon^\mu = \varepsilon^\mu \varepsilon^\lambda$ 
directly (see \cref {sec:commcomb} for further discussion).
We reduce the problem to a different identity, which follows from combinatorial arguments in \cite{purbhoo}.

\subsubsection{Reduction}
\label{sec:creduction}

\begin{lemma}
\label{lem:commutativityreduction}
The following are equivalent:
\begin{enumerate}[(a)]
\item\label{commutativityreduction_a} for all $z_1, \dots, z_n\in\CC$, the algebra $\altbethe$ is commutative;
\item\label{commutativityreduction_b} for all $z_1, \dots, z_n\in\CC$, the algebra $\altbethe$ commutes with $\bethe$;
\item\label{commutativityreduction_c} for all $k, \ell \geq 0$ and all partitions $\mu$, we have
\begin{equation}
\label{eq:reducedcommutativity}
   \varepsilon^\mu_\ell \betaminus{k}
   = \betaminus{k} \varepsilon^\mu_\ell 
\,.
\end{equation}
\end{enumerate}
\end{lemma}

\begin{proof}
\ref{commutativityreduction_a} $\Rightarrow$ \ref{commutativityreduction_b}: This follows immediately from \cref{lem:comparealgebras}\ref{comparealgebras3}.

\ref{commutativityreduction_b} $\Rightarrow$ \ref{commutativityreduction_a}: Suppose that \ref{commutativityreduction_b} holds. Since commutativity is preserved under limits, it suffices to prove \ref{commutativityreduction_a} when $z_1, \dots, z_n \in \CC$ are generic. Then $\bethe$ is a maximal commutative subalgebra of $\CC[\Sn]$ by \cref{thm:bethebasics}\ref{bethebasics2}.  By
definition, this means that any element of $\CC[\Sn]$ that 
commutes with $\bethe$ must belong to $\bethe$.  Hence $\altbethe \subseteq \bethe$, which implies \ref{commutativityreduction_a}.

\ref{commutativityreduction_b} $\Leftrightarrow$ \ref{commutativityreduction_c}: We will prove this treating
$z_1, \dots, z_n$ as formal indeterminates.  Statement \ref{commutativityreduction_b} means
that every generator of $\altbethe$ commutes with every generator
of $\bethe$, i.e., 
\[
      \varepsilon^\mu \betaminus{k}(t)
     = \betaminus{k}(t) \varepsilon^\mu \qquad \text{for all } k, t, \mu\,.
\]
Since $z_1, \dots, z_n$ are formal indeterminates, 
this is equivalent to the statement obtained if we translate
$(z_1, \dots, z_n) \mapsto (z_1-t, \dots, z_n-t)$:
\[
   \varepsilon^\mu(-t) \betaminus{k}
   = \betaminus{k} \varepsilon^\mu(-t)  \qquad \text{for all } k, t, \mu\,.
\]
Comparing coefficients of $t$ on both sides of the equation above,
this is equivalent to \ref{commutativityreduction_c}.
\end{proof}

\hidelinks
\subsubsection{Bijective proof of the relations \texorpdfstring{\eqref{eq:reducedcommutativity}}{(\ref{eq:reducedcommutativity})}}
\restorelinks
\label{sec:cbijection}

We now explain how \eqref{eq:reducedcommutativity} follows from \cite{purbhoo}; this shows that $\altbethe$ is commutative. 
A \defn{supported permutation} of $[n]$ is formally a pair $(\sigma, X)$, 
which we write as $\sigma_X$, where $\sigma \in \Sn$,
and $X \subseteq [n]$ is a set such that $\sigma$ belongs to the
subgroup $\symgrp{X} \subseteq \Sn$ (equivalently, 
$[n] \setminus X$ is a subset of the fixed points of $\sigma$).
The set $X$ is called the \defn{support} of $\sigma_X$, and $\cyc(\sigma_X)$ denotes the partition of $|X|$ which is the cycle type of $\sigma$ restricted to $X$.
Let $\SP_n$ denote the set of all supported permutations of $[n]$.
With this notation, we can rewrite the generators of $\altbethe$ as follows:
\begin{equation}
\label{eq:betadef3}
    \beta^\lambda = \sum_{\substack{\sigma_X \in \SP_n, \\ |X| = |\lambda|}} 
      \chi^\lambda(\sigma) \sigma \,z_{[n] \setminus X}\,,
\qquad
   \varepsilon^\mu = \sum_{\substack{\sigma_X \in \SP_n, \\ \cyc(\sigma_X) = \mu}}
      \sigma \, z_{[n] \setminus X}
\,.
\end{equation}

We similarly expand both sides of \eqref{eq:reducedcommutativity}.
Let $E^\mu_\ell$ denote the set of pairs $(\sigma_X,Y)$ such
that $\sigma_X\in\SP_n$ with $\cyc(\sigma_X) = \mu$, and $Y\subseteq [n]\setminus X$ with $|Y| = \ell$.
Let $B_k$ denote the set of pairs $(\sigma'_{X'},Y')$ such that $\sigma'_{X'}\in\SP_n$ with $|X'| = k$, and $Y' = [n] \setminus X'$.
Then
\[
     \varepsilon^\mu_\ell = 
    \sum_{(\sigma_X, Y) \in E^\mu_\ell} \sigma \, z_Y
\qquad \text{and} \qquad
     \betaminus{k} = 
    \sum_{(\sigma'_{X'}, Y') \in B_k} \sgn(\sigma') \sigma' z_{Y'}
\,,
\]
and \eqref{eq:reducedcommutativity} is the statement
\begin{equation}
\label{eq:reducedcommutativity2}
    \sum_{(\sigma_X, Y; \sigma'_{X'}, Y') \in E^\mu_\ell \times B_k}
      \sgn(\sigma') \sigma \sigma' z_Y z_{Y'}
 =
    \sum_{(\bar\sigma'_{\bar X'}, \bar Y'; \bar\sigma_{\bar X}, \bar Y) \in B_k \times E^\mu_\ell}
      \sgn(\bar \sigma') \bar \sigma' \bar \sigma \, z_{\bar Y} z_{\bar Y'}
\,.
\end{equation}
To prove \eqref{eq:reducedcommutativity}, we want to match
up each term on the left-hand side of \eqref{eq:reducedcommutativity2}
with an equal term on the right-hand side.  That is, we seek a bijection
\begin{align*}
    E^\mu_\ell \times B_k &\to B_k \times E^\mu_\ell\,, \\
   (\sigma_X, Y; \sigma'_{X'}, Y') &\mapsto 
   (\bar\sigma'_{\bar X'}, \bar Y'; \bar\sigma_{\bar X}, \bar Y)
\end{align*}
such that $\sgn(\sigma') = \sgn(\bar \sigma')$, 
$\sigma \sigma' = \bar \sigma' \bar\sigma$, and 
$z_Yz_{Y'} = z_{\bar Y} z_{\bar Y'}$.
A bijection with precisely these properties is established in 
\cite[Proposition 5.3]{purbhoo}.
Hence \eqref{eq:reducedcommutativity} holds. \qed

\hidelinks
\subsubsection{Proof of \texorpdfstring{\fullcref{lem:comparealgebras}}{Lemma \ref{lem:comparealgebras}}}
\restorelinks

We now complete the proof of \cref{lem:comparealgebras}.
We have already established parts \ref{comparealgebras3} and \ref{comparealgebras4} in 
\cref{sec:translation}, and \ref{comparealgebras1} is proved in 
\cref{sec:cbijection}.  It remains to prove part \ref{comparealgebras2}. The containment $\bethe \subseteq \altbethe$ follows from part \ref{comparealgebras3}. The reverse containment follows just as in the proof of the implication \ref{commutativityreduction_b} $\Rightarrow$ \ref{commutativityreduction_a} in \cref{sec:creduction}; that is, $\altbethe$ is commutative, and $\bethe$ is a maximal commutative subalgebra of $\CC[\Sn]$ when $z_1, \dots, z_n\in\CC$ are generic.\qed

\begin{remark}
It is interesting to note that although commutativity is deduced
by taking limits from the generic case, we cannot similarly
deduce that $\altbethe = \bethe$, even though we know this to be
true generically.  This is because 
equality of algebras is not necessarily preserved 
under taking limits of generators.
In order to show that $\altbethe = \bethe$,
we will argue in \cref{sec:altbetheequalsbethe} that there exist polynomial formulas expressing 
the generators of $\altbethe$ in terms of the generators of $\bethe$.
As these formulas are preserved under taking limits, this will imply
the result.
\end{remark}


\section{\Plucker relations}
\label{sec:pr}

In this section we prove the remaining parts (i.e.\ \ref{main_pluckers}--\ref{main_multiplicity}) of \cref{thm:main}.  The bulk of argument involves showing that the operators $\beta^\lambda(t)$ satisfy the \Plucker relations (\cref{thm:main}\ref{main_pluckers}). We divide this into two parts, which we state now as lemmas:
\begin{lemma}
\label{lem:part1}
The operators $\beta^\lambda$ satisfy all single-column \Plucker relations.
\end{lemma}

\begin{lemma}
\label{lem:part2}
\cref{lem:part1} and the translation identity 
\eqref{eq:translationidentity} together
imply that the operators $\beta^\lambda(t)$ satisfy the \Plucker relations (\cref{thm:main}\ref{main_pluckers}).
\end{lemma}

The two parts are very different in character. The proof of \cref{lem:part1} uses $\tau$-functions of the KP hierarchy, and involves a combinatorial analysis 
of factorizations of permutations, along with several symmetric-function identities.
To prove \cref{lem:part2}, we use the exterior algebra,
and some of the
abstract properties of $\altbethe$ from 
\cref{thm:bethebasics} and
\cref{lem:comparealgebras}.

We prove \cref{lem:part1,lem:part2} in \cref{sec:part1proof,sec:part2}, respectively. We conclude in \cref{sec:finalsteps} by finishing the proof of \cref{thm:main}, by showing that parts \ref{main_generates}--\ref{main_multiplicity} hold.


\hidelinks
\subsection{Proof of \texorpdfstring{\fullcref{lem:part1}}{Lemma \ref{lem:part1}}: single-column \Plucker relations}
\restorelinks
\label{sec:part1proof}

We turn to the proof of \cref{lem:part1}, using $\tau$-functions of the KP hierarchy.

\subsubsection{A \texorpdfstring{$\tau$}{tau}-function of the KP hierarchy}

Define $\tau_n \in \CC[\Sn] \otimes \Lambda$ to be the following symmetric function with coefficients in 
the commutative algebra $\altbethe \subseteq \CC[\Sn]$:
\begin{equation}
\label{eq:taudef}
   \tau_n := \sum_\lambda \beta^\lambda \otimes \s_\lambda
\,,
\end{equation}
where the sum is taken over all partitions $\lambda$.  Since $\beta^\lambda = 0$ for $|\lambda| > n$, this is a finite sum over all $|\lambda| \le n$.

We mention that because the coefficients of $\tau_n$ in the Schur basis are the operators
$\beta^\lambda$, we can use \cref{thm:KPplucker} to rephrase \cref{thm:main}\ref{main_pluckers} as follows:

\begin{theorem}
\label{thm:taufunction}
For all $n \geq 0$ and all $z_1, \dots, z_n \in \CC$,
the symmetric function $\tau_n$ in \eqref{eq:taudef} 
is a $\tau$-function of the KP hierarchy, 
with coefficients in (a commutative subalgebra of) $\CC[\Sn]$.
\end{theorem}

Our approach will instead be to use $\tau_n$ to prove the single-column \Plucker relations. We need the following formula for $\tau_n$ in the power sum basis:
\begin{proposition}
We have
\begin{equation}
\label{eq:tau}
  \tau_n = \sum_\mu \varepsilon^\mu \otimes \p_\mu =
   \sum_{\sigma_X \in \SP_n} \sigma \, z_{[n] \setminus X} \otimes \p_{\cyc(\sigma_X)}
\,.
\end{equation}
\end{proposition}

\begin{proof} 
The first equality follows from \eqref{eq:spchangeofbasis} and \eqref{eq:changeofbasis}.
Then~\eqref{eq:betadef3} gives the second equality.
\end{proof}

By \cref{lem:singlecolumneqs}, the operators $\beta^\lambda$ satisfy the single-column \Plucker relations if and only if equation \eqref{eq:scKP2} is satisfied.
Since $\langle \tau_n, \e_k \rangle = \langle \tau_n, \s_{1^k} \rangle = \betaminus{k}$, in our case
\eqref{eq:scKP2} is the statement that
\[
    \Bresidue  
   \sum_{k=0}^n (-t)^{-k} 
   \betaminus{k}
    \cdot \tau_n = 0
\,.
\]
Using \eqref{eq:betadef3} and \eqref{eq:tau}, the
sum $\sum_{k=0}^n (-t)^{-k} \betaminus{k} \cdot \tau_n$
expands into
\begin{equation}
\label{eq:scoperand}
    \sum_{\sigma_X \in \SP_n} \, \sum_{\pi_Y \in \SP_n}
        \sgn(\sigma) \cdot \sigma \pi \,
      z_{[n] \setminus X} z_{[n] \setminus Y}
       \otimes (-t)^{-|X|} \p_{\cyc(\pi_Y)}
\,.
\end{equation}
Therefore, to prove \cref{lem:part1}, we must show that
\eqref{eq:scoperand}
is in the kernel of the operator $\Bresidue$.

\subsubsection{\texorpdfstring{$Z$}{Z}-factorizations}
\label{sec:zfac}

We will now treat $z_1, \dots, z_n$ as formal indeterminates.
Since the operator $\Bresidue$ is linear, and
treats the first tensor factor of \eqref{eq:scoperand} as 
a scalar, we can prove \cref{lem:part1} coefficient-by-coefficient.
That is, it is necessary and sufficient 
to show that for all $\theta \in \Sn$ and all $j_1, \dots, j_n \in \{0,1,2\}$, 
the $\theta z_1^{j_1} \dotsm z_n^{j_n}$-coefficient of 
\eqref{eq:scoperand} is in $\ker \Bresidue$.
Furthermore, it suffices to prove this for coefficients which
are square-free in the $z_i$'s. 
This is because each non-square-free coefficient is equal to
a square-free coefficient for a smaller value of $n$.  
Specifically, the sum of all terms which contain a factor
of $z_n^2$ is 
\[
    z_n^2 \cdot \sum_{\sigma_X \in \SP_{n-1}} \, \sum_{\pi_Y \in \SP_{n-1}}
         \sgn(\sigma) \cdot \sigma \pi \,
      z_{[n-1] \setminus X} z_{[n-1] \setminus Y}
       \otimes (-t)^{-|X|} \p_{\cyc(\pi_Y)}
\,,
\]
which (apart from the factor of $z_n^2$) is just \eqref{eq:scoperand} 
with $n$ replaced by $n-1$.

Therefore, we fix $\theta \in \Sn$ and $Z \subseteq [n]$, and 
let $C_{\theta,Z}$ denote the 
coefficient of $\theta z_{[n] \setminus Z}$ in \eqref{eq:scoperand}.
Our goal is now to show that $C_{\theta,Z} \in \ker \Bresidue$.

We begin by obtaining a useful formula for $C_{\theta,Z}$. 
Since \eqref{eq:scoperand} is a quadratic expression in
$\CC[\Sn]$, we can write
$C_{\theta,Z}$ as a sum over certain factorizations 
of $\theta$. Namely, define a \defn{$Z$-factorization} 
of $\theta$ to be a pair of supported permutations
$(\sigma_X, \pi_Y)$ such that $X\cup Y = [n]$, $X \cap Y = Z$, and 
$\sigma \pi = \theta$. 
We see that 
\[
    C_{\theta,Z} = 
    \sum_{(\sigma_X, \pi_Y)} \sgn(\sigma) (-t)^{-|X|} \p_{\cyc(\pi_Y)}
\,,
\]
where the sum is taken over all $Z$-factorizations of $\theta$.

It will be more
convenient to write $C_{\theta,Z}$ purely in terms of $\pi$ and $Y$.
We say that $\pi_Y$ is a 
\defn{right $Z$-factor} of $\theta$, if there exists 
$\sigma_X$ such that $(\sigma_X, \pi_Y)$ is a $Z$-factorization of $\theta$.  
Note that if $\sigma_X$ exists, then it is
unique, as we must have $\sigma = \theta \pi^{-1}$ and
$X = [n] \setminus (Y \setminus Z)$.  (However, for some $\pi_Y$, this
construction may not produce a valid supported permutation $\sigma_X$.)
Using $\sgn(\theta) = \sgn(\sigma)\sgn(\pi)$ and $|Z| = |X|+|Y|-n$,
we obtain
\[
    C_{\theta,Z} = 
    \sgn(\theta) 
    \sum_{\pi_Y} \sgn(\pi) (-t)^{|Y|-|Z|-n} \p_{\cyc(\pi_Y)}
\,,
\]
where the sum is taken over all right $Z$-factors $\pi_Y$ of $\theta$.

Finally, for a given $Y \subseteq [n]$, let $\RF(Y)$ denote the set 
of all permutations $\pi \in \symgrp{Y}$
such that $\pi_Y$ is a right $Z$-factor of $\theta$.  With this
notation, we write
\begin{equation}
\label{eq:thetaZcoeff}
    C_{\theta,Z} = 
    \sgn(\theta) 
    \sum_{Y \subseteq [n]} (-t)^{|Y|-|Z|-n} 
    \sum_{\pi \in \RF(Y)} \sgn(\pi) \p_{\cyc(\pi_Y)}
\,.
\end{equation}

\subsubsection{\texorpdfstring{$Z$}{Z}-strips}
Next, we evaluate the inner sum of~\eqref{eq:thetaZcoeff}:
\[
    \sum_{\pi \in \RF(Y)} \sgn(\pi) \p_{\cyc(\pi_Y)}
\,.
\]
To accomplish this, we need a precise understanding of which 
supported permutations
are right $Z$-factors of $\theta$.

Define a \defn{$Z$-strip} of $\theta$ to be a sequence of the form
\[
     \big(a, \theta(a), \theta^2(a), \dots, \theta^{r}(a)\big)\,,
\]
where $a \in Z$, $r\ge 0$, $\theta^{r+1}(a) \in Z$, and 
$\theta(a), \dots, \theta^{r}(a) \notin Z$.
A \defn{$Z$-substrip} of $\theta$
is a nonempty left-substring of a $Z$-strip, i.e.,
\[
\big(
     a, \theta(a), \theta^2(a), \dots, \theta^{r'}(a)
\big) \qquad \text{ for some }\, 0 \le r' \le r\,.
\]
Abusing notation, we will sometimes think
of these as sets: this is reasonable because the elements of any
$Z$-substrip are distinct, and for any set $S \subseteq [n]$,
there is at most one ordering of the elements of $S$ which forms
a $Z$-substrip.

Let $S_1, \dots, S_k$ be the $Z$-strips of $\theta$.  Note that the $S_i$'s are pairwise disjoint and
$k = |Z|$, as each element of $Z$ defines a unique $Z$-strip.
Put $S :=  S_1 \cup \dots \cup S_k$.
Let $\lambda_i := |S_i|$, and write $\lambda =(\lambda_1, \dots, \lambda_k)$.
Thus $\lambda$ is a composition of size $|\lambda| = |S|$.
Let $T := [n] \setminus S$ be the set of
points which are not in any $Z$-strip.  Then $T \cap Z = \emptyset$ and
$\theta(T) = T$.
Write $\theta|_T \in \symgrp{T}$ 
to denote the permutation $\theta$ restricted to $T$.

\begin{example}
\label{ex:strips1}
Let $\theta := (1 \;\; 2 \;\; 3 \;\; 4 \;\; 5)(6 \;\; 7 \;\; 8)(9 \;\; 10)(11)$ and $Z := \{3,5,7\}$.  The
$Z$-strips of $\theta$ are
\[
   34,\ 512, \text{ and } 786\,.
\]
We have $T = \{9,10,11\}$ and $\theta|_T = (9 \;\; 10)(11)$.
\end{example}

We say that a subset $Y \subseteq [n]$ is \defn{valid} for $(\theta, Z)$
if the following conditions hold:
\begin{itemize}
\item $Z \subseteq Y$; 
\item $S_i \cap Y$ is a $Z$-substrip of $\theta$ for
all $i \in [k]$; and
\item $\theta(T \cap Y) = T \cap Y$.
\end{itemize}

\begin{lemma}
\label{lem:RF}
Let $\pi \in \symgrp{Y}$.
\begin{enumerate}[(i)]
\item\label{RF1} If $Y$ is not valid, then $\RF(Y)$ is empty.
\item\label{RF2} If $Y$ is valid, then $\pi \in \RF(Y)$ if and only if the following conditions hold:
\begin{enumerate}[(a)]
\item\label{RF_a} $S_1 \cap Y, \dots, S_k \cap Y$ are the $Z$-strips of $\pi$;
\item\label{RF_b} $\theta|_{T \cap Y} = \pi|_{T \cap Y}$.
\end{enumerate}
\end{enumerate}
\end{lemma}

\begin{proof}
First suppose that $\pi \in \RF(Y)$. We will show that $Y$ is valid (thereby proving part \ref{RF1}) and that \ref{RF_a} and \ref{RF_b} hold (thereby proving the forward direction of part \ref{RF2}). We regard $\pi$ as a permutation of $[n]$ which fixes $[n]\setminus Y$ pointwise. Since $\pi \in \RF(Y)$, we have $Z \subseteq Y$ and that $\theta\pi^{-1}$ fixes $Y\setminus Z$ pointwise. 
That is, for all $a \in [n]$, we have
  $\pi^{-1}(a) = a$ if $a \notin Y$, and
  $\pi^{-1}(a) = \theta^{-1}(a)$ if $a  \in Y\setminus Z$
(and there is no condition if $a \in Z$).

Let $(a_1 \; \cdots \; a_r)$ be a cycle of $\theta|_T$, and set $a_0 := a_r$.
Then for all $i\in [r]$, we have that $\pi^{-1}(a_i)$ equals either 
$a_i$ (if $a_i \notin Y$) or $a_{i-1}$ (if $a_i \in Y$). If we have $a_i \in Y$ for
some $i$, then $\pi^{-1}(a_i) = a_{i-1}$, so $\pi^{-1}(a_{i-1}) = a_{i-2}$ 
since $\pi$ is a bijection, and by induction $\pi$ coincides with $\theta$
on the cycle. Therefore $\theta(T\cap Y) = T\cap Y$ and \ref{RF_b} holds.

Now let $(a_0, a_1, \dots, a_r)$ be a $Z$-substrip of $\theta$ with $a_r \in Y$. Since $a_1, \dots, a_r \notin Z$, we have $\pi^{-1}(a_r) = a_{r-1}$, and a similar induction shows that $\pi^{-1}(a_i) = a_{i-1}$ (and hence $a_i\in Y$) for all $i\in [r]$. Thus $Y$ is valid and \ref{RF_a} holds.

The backward direction of part \ref{RF2} follows using similar arguments.
\end{proof}

Informally, \cref{lem:RF}\ref{RF2} says that
for any given valid $Y$, we can form all permutations 
$\pi \in \RF(Y)$ by assembling the substrips 
$S_1 \cap Y, \dots, S_k \cap Y$ into cycles
in all possible ways.  In addition, $\pi$ must 
include all cycles of $\theta|_{T \cap Y}$. We illustrate this with an example:

\begin{example}
\label{ex:strips2}
For $(\theta, Z)$ from \cref{ex:strips1}, 
consider the valid $Y \hspace*{-1.5pt}=\hspace*{-1.5pt} \{1,\hspace*{-0.1pt} 3,5,6,7,8,9,10\}$.  
The substrips $S_i \cap Y$ are
\[
    3,\ 51, \text{ and } 786\,,
\]
and $\theta|_{T \cap Y} = (9 \;\; 10)$. There are $3!=6$ permutations in $\RF(Y)$, corresponding to the
different ways to assemble the three substrips above into cycles:
\[
\begin{aligned}
    &(3)(5 \;\; 1)(7 \;\; 8 \;\; 6)(9 \;\; 10)\,, \\
    &(3 \;\; 5 \;\; 1)(7 \;\; 8 \;\; 6)(9 \;\; 10)\,, \\
	&(3 \;\; 7 \;\; 8 \;\; 6)(5 \;\; 1)(9 \;\; 10)\,,
\end{aligned}
\qquad \qquad
\begin{aligned}
    &(3)(5 \;\; 1 \;\; 7 \;\; 8 \;\; 6)(9 \;\; 10)\,, \\
    &(3 \;\; 5 \;\; 1 \;\; 7 \;\; 8 \;\; 6)(9 \;\; 10)\,, \\
    &(3 \;\; 7 \;\; 8 \;\; 6 \;\; 5 \;\; 1)(9 \;\; 10)\,.
\end{aligned}
\]
Note that $(9 \;\; 10)$ is a cycle of each of these permutations.
\end{example}

\begin{lemma}
\label{lem:innersum}
Fix a valid $Y$. Let $\kappa_i := |S_i \cap Y|$, and write
$\kappa = (\kappa_1, \dots, \kappa_k)$.  Let $\mu := \cyc(\theta|_{T\cap Y})$.
Then
\begin{equation}
\label{eq:innersum}
    \sum_{\pi \in \RF(Y)} \sgn(\pi) \p_{\cyc(\pi_Y)}
    = (-1)^{|\kappa|-k + |\mu| - \ell(\mu)} \sm_\kappa \p_\mu
\,.
\end{equation}
\end{lemma}

\begin{example}
For $(\theta, Z)$ and $Y$ from \cref{ex:strips1,ex:strips2}, 
equation \eqref{eq:innersum} is
\[
    (-\p_{321} + \p_{33} + \p_{42} + \p_{51} -2\p_{6}) (-\p_2)
     = \sm_{321}\p_2
\,.
\]
The left-hand side is written with $(-1)^{|\mu|-\ell(\mu)}\p_\mu = -\p_2$
factored out; this factor corresponds to the cycle $(9 \;\; 10)$,
which is common to all $\pi \in \RF(Y)$.
\end{example}

\begin{proof}[Proof of \cref{lem:innersum}]
By \cref{lem:RF}\ref{RF2}\ref{RF_b}, we can factor out $\sgn(\theta|_{T\cap Y})\p_{\cyc(\theta|_{T\cap Y})} = (-1)^{|\mu| - \ell(\mu)}\p_\mu$ from the left-hand side of \eqref{eq:innersum}. Cancelling this factor from both sides reduces the proof of \eqref{eq:innersum} to the case when $T\cap Y = \emptyset$. In this case, by \cref{lem:RF}\ref{RF2}\ref{RF_a}, any element $\pi\in\RF(Y)$ is given by arranging the substrips $S_i\cap Y$ into cycles; this corresponds to a permutation $\sigma \in \symgrp{k}$, where we regard $S_i\cap Y$ as the element $i\in [k]$. Since $\sgn(\pi) = (-1)^{|\kappa| - k}\sgn(\sigma)$ and $\cyc(\pi_Y) = \kappa[\sigma]$, the equation follows from \cref{prop:scaledmonomials}.
\end{proof}

\subsubsection{Simplifying \texorpdfstring{$C_{\theta,Z}$}{C(theta,Z)}}

We now use \eqref{eq:innersum} to further simplify our formula
\eqref{eq:thetaZcoeff} for $C_{\theta,Z}$, and complete the proof of \cref{lem:part1}.
To specify a valid $Y$, it is enough to choose a substrip of each $S_i$
and a subset of the cycles of $\theta|_T$.
A substrip of $S_i$ is uniquely determined by its length 
$\kappa_i$, where $1 \leq \kappa_i \leq \lambda_i$.
Thus we have a one-to-one correspondence between valid subsets $Y$ and pairs 
$(\kappa, H)$, where $\kappa$ is a composition such that
$\kappa \leq \lambda$  
and $H$ is a subset of the cycles of $\theta|_T$.  
Abusing notation, we denote the latter condition as $H \subseteq \theta|_T$.
Under this correspondence $|Y| = |\kappa| + |\mu_H|$, where 
$\mu_H$ is the partition listing the lengths of the cycles in $H$.
Thus from \eqref{eq:thetaZcoeff} and \eqref{eq:innersum}, we have
\begin{align*}
C_{\theta,Z} &= 
    \sgn(\theta)
    \sum_{Y \subseteq [n]} (-t)^{|Y|-k-n} 
    \sum_{\pi \in \RF(Y)} \sgn(\pi) \p_{\cyc(\pi_Y)}
\\[4pt]
&=
    \sgn(\theta) 
     \sum_{\kappa \leq \lambda}\,
     \sum_{H \subseteq \theta|_T} 
      (-t)^{|\kappa|+|\mu_H|-k-n}  \cdot
       (-1)^{|\kappa|-k+|\mu_H|-|H|} \sm_\kappa 
     \,\p_{\mu_H}
\\[4pt] &= 
    \sgn(\theta) 
      (-1)^{n}  
     \sum_{\kappa \leq \lambda} 
     t^{|\kappa|-k-n} 
      \sm_\kappa  \cdot
     \sum_{H \subseteq \theta|_T} 
       (-1)^{|H|} 
      t^{|\mu_H|}  
     \p_{\mu_H}
\\
&=
     \sgn(\theta) (-1)^n \sum_{\kappa \leq \lambda}
         t^{|\kappa|-k-n} \sm_\kappa 
         \cdot \prod_{i=1}^s (1- t^{\eta_i}\p_{\eta_i})
\,,
\end{align*}
where $\eta = (\eta_1, \dots, \eta_s) := \cyc(\theta|_T)$.

We conclude the proof of \cref{lem:part1} as follows.  
Recall that our goal is to show that 
$C_{\theta,Z} \in \ker \Bresidue$.
By \cref{lem:Bp-identity},  this is true
if and only if
\[
     t^{|\eta|} \cdot
     \sum_{\kappa \leq \lambda}
         t^{|\kappa|-k-n} \sm_\kappa 
 \in \ker \Bresidue
\,.
\]
Since $|\eta| = n - |\lambda|$, this is precisely the statement of \cref{lem:Bm-identity}. \qed


\hidelinks
\subsection{Proof of \texorpdfstring{\fullcref{lem:part2}}{Lemma \ref{lem:part2}}: general \Plucker relations}
\restorelinks
\label{sec:part2}

We now show that \cref{lem:part1}, when combined with the translation
identity \eqref{eq:translationidentity}, implies that the operators
$\beta^\lambda(t)$ satisfy the \Plucker relations. We extend our convention \eqref{eq:identification} for indexing \Plucker coordinates to the
operators $\beta^\lambda(t)$.
For $I = (i_1, \dots, i_n) \in [2n]^n$, define $\beta_I(t)$
as follows:
if $i_{\sigma(1)} < \dots < i_{\sigma(n)}$ for some $\sigma \in \Sn$, put 
$\beta_I(t) := \sgn(\sigma) \beta^\lambda(t)$, where
$\lambda = (i_{\sigma(n)}-n, \dots, i_{\sigma(1)}-1)$;
if $i_1, \dots, i_n$ are not distinct, put $\beta_I(t) := 0$.
We write $\beta_I := \beta_I(0)$.

We will work with the differential operator $\calD_n$ from \eqref{eq:diffop}; however,
for our purposes, 
it will be more convenient to omit the leading factor of $\frac{1}{g(u)}$.
Hence, we define
\begin{equation}
\label{eq:polyDdef}
     \polyD_n(t) := \sum_{k=0}^n (-1)^k \betaminus{k}(u+t) 
    \du^{n-k} \qquad \text{and} \qquad
    \polyD_n := \polyD_n(0) = g(u)\calD_n
\,.
\end{equation}
We will often restrict the domain of these linear operators from $\CC[u]$ to $\CC_{2n-1}[u]$.

\subsubsection{Reformulation of the translation identity}

For every nonnegative integer $i$ and $t \in \CC$, define the polynomials
\begin{equation}
\label{eq:ei}
e_i(u) := \frac{u^{i-1}}{(i-1)!} \qquad \text{and} \qquad
e_{i,t}(u) := \translatematrix{t}e_i(u) = e_i(u+t) = \frac{(u+t)^{i-1}}{(i-1)!}
\,.
\end{equation}
Thus, $(e_{1,t}, \dots, e_{2n,t})$ is a basis for $\CC_{2n-1}[u]$, and $(e_1, \dots, e_{2n})$ corresponds to the standard
basis for $\CC^{2n}$ under the isomorphism~\eqref{eq:isomorphism}.

For $I = (i_1, \dots, i_n) \in [2n]^n$, write 
\[
   e_{I,t} := e_{i_1,t} \wedge \dots \wedge e_{i_n,t}
 \in
   \exterior{n} \CC_{2n-1}[u]
\,,
\]
and put $e_I := e_{I,0}$.
Following our convention \eqref{eq:identification}, we write $e^\lambda_t := e_{(\lambda_n+1, \lambda_{n-1}+2, \dots, \lambda_1+n),t}$ for any partition $\lambda \subseteq \rectangle = n^n$. By \cref{prop:exteriorplucker} and \eqref{eq:geometrictranslation}, we have
\begin{equation}
\label{eq:basistranslation}
e^\lambda_t = \sum_{\mu \subseteq \lambda} 
      \frac{\numsyt{\lambda/\mu}}{|\lambda/\mu|!}
      t^{|\lambda/\mu|}e^\mu \qquad \text{for every } \lambda\subseteq\rectangle\,.
\end{equation}

For $t\in\CC$, consider the element
$\Omega(t) \in  \exterior{n} \CC_{2n-1}[u] \otimes \altbethe$ defined by
\begin{equation}
\label{eq:omega}
    \Omega(t) := 
 \sum_{\lambda \subseteq \rectangle} e^\lambda \otimes \beta^\lambda(t) =
  \sum_{I \in {[2n] \choose n}} e_I \otimes \beta_I(t)
\,,
\end{equation}
and put $\Omega := \Omega(0)$. The following is a reformulation of the translation identity \eqref{eq:translationidentity}:

\begin{lemma}
\label{lem:exteriortranslation}
For all $s, t \in \CC$, we have
\[
    \Omega(s+t) = \sum_{\lambda\subseteq\rectangle} e^\lambda_t \otimes \beta^\lambda(s)
\,.
\]
\end{lemma}

\begin{proof}
Upon expanding $e^\lambda_t$ as in \eqref{eq:basistranslation}, this becomes equivalent to \eqref{eq:translationidentity}.
\end{proof}

\subsubsection{Derivations on the exterior algebra}

We recall how to extend any linear map $L : V \to W$ to a derivation $L_\#$
on the exterior algebra $\extalg V$:

\begin{lemma}
\label{lem:exterior}
Let $V$ and $W$ be vector spaces, and let $L : V \to W$ be a linear map.
\begin{enumerate}[(i)]
\item\label{exterior1}
There exists a unique linear map
\[
   L_\# : \extalg V \to \extalg V  \otimes  W
\,,
\]
such that for all $k \geq 0$ and $v_1, \dots, v_k \in V$, we have
\begin{equation}
\label{eq:Lsharp}
   L_\#(v_1 \wedge \dots \wedge v_k) = 
    \sum_{i=1}^k (-1)^{k-i} 
   (v_1 \wedge \cdots \wedge \widehat{v}_i \wedge \cdots \wedge v_k) 
   \otimes 
   L(v_i) 
\,.
\end{equation}
\item\label{exterior2} $L_\#$ is a derivation: for all
$\omega \in \exterior{k} V$ and $\omega' \in \exterior{l} V$, we have
\[
    L_\#(\omega \wedge \omega') = \omega \wedge L_\# \omega'
    + (-1)^{kl} \omega' \wedge L_\# \omega
\,.
\]
\item\label{exterior3} We have $\ker(L_\#) = \extalg \ker(L)$.
\end{enumerate}
\end{lemma}

\begin{proof}
The map $V^k \to \exterior{k-1}V \otimes W$, defined by
mapping $(v_1, \dots, v_k)$ to the right-hand side of \eqref{eq:Lsharp}
is $k$-linear and alternating, and therefore
extends to a unique linear map on $\exterior{k} V$.  This proves part \ref{exterior1},
and part \ref{exterior2} follows directly.
For part \ref{exterior3}, it is clear that $\extalg \ker(L) \subseteq \ker(L_\#)$.
For the reverse inclusion, after replacing $W$ by $L(V)$, we may assume 
that $L$ is surjective.  

First note that if $L$ is invertible, 
then $L_\# : \exterior{k} V \to \exterior{k-1} V \otimes W$ is the
zero map if $k=0$ and injective
if $k \geq 1$. The former is true by definition, and the 
latter is true because
\[
   v_1 \wedge \dots \wedge v_{k-1} \otimes w \mapsto 
   \frac{1}{k} v_1 \wedge \dots \wedge v_{k-1} \wedge L^{-1}(w)
\]
defines the left-inverse map.  Thus, when $L$ is invertible, we have 
$\ker(L_\#) = \exterior{0}V$.

In general, there exists a direct sum decomposition 
$V = V' \oplus V''$, with $V' = \ker(L)$.  Thus $L$ restricted to
$V'$ is the zero map, and $L$ restricted to $V''$ is invertible.
Given $\omega \in \extalg V$, we can write
$\omega = \sum_{i=1}^s \omega'_i \wedge \omega''_i$, where
$\omega'_i \in \extalg V'$ and $\omega''_i \in \extalg V''$ for all $i\in [s]$,
and $(\omega'_1, \dots, \omega'_s)$ is linearly independent.
By part \ref{exterior2}, we have
$L_\# \omega = \sum_{i=1}^s \omega'_i \wedge L_\# \omega''_i$.
Finally, if $L_\# \omega = 0$, then $L_\# \omega''_i = 0$ for all $i \in [s]$.
Since $L$ restricted to $V''$ is invertible, this implies that 
$\omega''_i \in \exterior{0} V''$ for all $i \in [s]$.  Hence
$\omega \in \extalg V'$, which proves
$\ker(L_\#) \subseteq \extalg \ker(L)$.
\end{proof}

\begin{corollary}
\label{cor:exterior}
Let $L : V \to W$ be as in \cref{lem:exterior}, and
let $\omega \in \exterior{n} V$.
Suppose that
\[
L_\# \omega = 0 
\qquad\text{and}\qquad
\dim \ker(L) \leq n
\,.
\]
Then 
$\omega = v_1 \wedge \dots \wedge v_n$ for some 
$v_1, \dots, v_n \in \ker(L)$.  Furthermore, if $\omega \neq 0$,
then $\dim \ker(L) = n$ and 
$(v_1, \dots, v_n)$ is a basis for $\ker(L)$.
\end{corollary}

\begin{proof}
The result is trivial if $\omega = 0$, so assume $\omega \neq 0$.
By \cref{lem:exterior}, we have $\omega \in \exterior{n} \ker(L)$.
If $\dim \ker(L) < n$, then $\exterior{n} \ker(L)$ is zero-dimensional,
and cannot contain a nonzero vector.  Therefore we must have 
$\dim \ker(L) =n$, whence $\dim \exterior{n} \ker(L) = 1$, and
every nonzero vector is of the form $v_1 \wedge \dots \wedge v_n$,
where $(v_1, \dots, v_n)$ is a basis for $\ker(L)$.
\end{proof}

\subsubsection{Reduction to single-column \Plucker relations}
\label{sec:screduction}

If $R$ is a unital commutative $\CC$-algebra, 
we extend the definition of $L_\#$
to $R$-linear maps $L : V \otimes R \to W \otimes R$.  In this case,
we obtain an $R$-linear derivation
\begin{equation}\label{eq:derivation}
   L_\# : \extalg V \otimes R 
        \to \extalg V \otimes W \otimes R
\end{equation}
characterized by \eqref{eq:Lsharp}. In particular, taking
$R = \altbethe$, we consider the derivation
\[
   (\polyD_n)_\# : \extalg \CC_{2n-1}[u] \otimes \altbethe \to
     \extalg \CC_{2n-1}[u] \otimes \CC[u] \otimes \altbethe
\,
\]
associated to 
$\polyD_n : \CC_{2n-1}[u] \otimes \altbethe \to \CC[u] \otimes \altbethe$ from \eqref{eq:polyDdef}. (We have switched the order of the tensor factors from \eqref{eq:diffop}, to be consistent with \eqref{eq:derivation}.)

For $t \in \CC$ and $\CC$-vector spaces $V_1$ and $V_2$, 
let $\ev_t : V_1 \otimes \CC[u] \otimes V_2 \to V_1 \otimes V_2$ 
denote the evaluation map 
$f \mapsto f(t)$ on the tensor factor of $\CC[u]$.

\begin{lemma}
\label{lem:translationmagic}
Let $z_1, \dots, z_n$ be formal indeterminates. 
Then the following are equivalent:
\begin{enumerate}[(a)]
\item\label{translationmagic_a} $(\polyD_n)_\# \Omega = 0$;
\item\label{translationmagic_b} $\ev_0 \big((\polyD_n)_\# \Omega\big) = 0$;
\item\label{translationmagic_c} the operators $\beta^\lambda$ satisfy the single-column \Plucker relations for $\Gr(n,2n)$.
\end{enumerate}
\end{lemma}

\begin{proof}
\ref{translationmagic_a} $\Leftrightarrow$ \ref{translationmagic_b}:
Clearly \ref{translationmagic_a} implies \ref{translationmagic_b}. Conversely, suppose that \ref{translationmagic_b} is true.
Then \ref{translationmagic_b} remains true if we perform the
translation $(z_1, \dots, z_n) \mapsto (z_1+t, \dots, z_n+t)$ for 
$t \in \CC$:
\[
    \ev_0 \big((\polyD_n(t))_\# \Omega(t)\big) = 0
\,.
\]
Now perform the change of variables $u \mapsto u-t$ on the equation
above. Under this change of variables, 
$\ev_0 \mapsto \ev_t$, $\betaminus{k}(u+t) \mapsto \betaminus{k}(u)$,
$\du \mapsto \du$, and $\Omega(t) \mapsto \Omega$ (by \cref{lem:exteriortranslation}).
Thus we obtain
\[
    \ev_t \big((\polyD_n)_\# \Omega \big) = 0
\,.
\]
Since this is true for all $t \in \CC$,
we deduce that \ref{translationmagic_a} holds.

\ref{translationmagic_b} $\Leftrightarrow$ \ref{translationmagic_c}:
By direct calculation, we have
\begin{align*}
    \ev_0 \big((\polyD_n)_\# \Omega \big)
    &= 
     \sum_{I \in {[2n] \choose n}} 
      \ev_0 \big((\polyD_n)_\# e_I\big) 
     \cdot \beta_I
\\
    &= 
     \sum_{J \in {[2n] \choose n-1}}
     e_J \otimes \sum_{k=1}^{2n} \ev_0 \big(\polyD_n(e_k) \big)
      \cdot \beta_{J+k} 
\\
    &=	 
     \sum_{J \in {[2n] \choose n-1}}
     e_J \otimes \sum_{k=1}^{n+1} 
      (-1)^{n+1-k} \betaminus{n+1-k}
        \cdot \beta_{J+k} 
\,,
\end{align*}
where $\beta_{J+k}$ is defined analogously to \eqref{eq:addsubscript}.
Therefore, $\ev_0 \big((\polyD_n)_\# \Omega \big) = 0$ if and only if
\[
   \sum_{k=1}^{n+1} (-1)^{n+1-k} \betaminus{n+1-k} \beta_{J+k}  = 0
\,
\qquad \text{for all $J \in \textstyle{[2n] \choose n-1}$}
\,.
\]
These are precisely the single-column \Plucker 
relations for $\Gr(n,2n)$.
\end{proof}

We now complete the proof of \cref{lem:part2}. We want to show that the operators $\beta^\lambda(t)$ satisfy the \Plucker relations.
Recall from \cref{sec:partitions} that
we only need to consider the \Plucker relations for $\Gr(n,2n)$.
Furthermore, by translation and continuity, it is enough to prove the result when $t=0$ and $(z_1, \dots, z_n) \in \CC^n$ is generic. In particular, by \cref{lem:comparealgebras}\ref{comparealgebras2} and \cref{thm:bethebasics}\ref{bethebasics2}, we may assume that $\altbethe$ is semisimple. Hence it suffices to prove that for all partitions $\nu$ and eigenspaces $E \subseteq \spechtnu$ of $\altbethenu$, the eigenvalues $\beta^\lambda_E$ satisfy the \Plucker relations for $\Gr(n,2n)$.

By \cref{lem:part1}, the operators
$\beta^\lambda$ satisfy all single-column \Plucker relations.
By \cref{lem:translationmagic}, we deduce that 
$(\polyD_n)_\# \Omega = 0$.  Thus we have
$(\polyD_E)_\# \Omega_E = 0$, where 
\begin{equation}
\label{eq:deomegae}
   \polyD_E := \sum_{k=0}^{n} (-1)^k \betaminus{k}_E(u) \du^{n-k}
\qquad\text{and}\qquad
   \Omega_E := \sum_{\lambda \subseteq \rectangle} \beta^\lambda_E e^\lambda
\,.
\end{equation}
Since $\polyD_E : \CC_{2n-1}[u] \to \CC[u]$ is a linear differential
operator of order $n$, 
by \cref{prop:Dnullity} we have $\dim \ker (\polyD_E) \leq n$.
Therefore by \cref{cor:exterior}, 
we have $\Omega_E = v_1 \wedge \dots \wedge v_n$ for some
$v_1, \dots, v_n \in \CC_{2n-1}[u]$.
It follows from \cref{prop:exteriorplucker} that the
coefficients $\beta^\lambda_E$ satisfy the \Plucker relations for
$\Gr(n,2n)$. \qed


\subsection{Final steps in the proof of \texorpdfstring{\cref*{thm:main}}{Theorem \ref{thm:main}}}
\label{sec:finalsteps}

We have so far proved parts \ref{main_commutativity}--\ref{main_pluckers} of \cref{thm:main}.  
We now complete the proof by proving
parts \ref{main_generates}--\ref{main_multiplicity}.  Most of part \ref{main_generates} is established by
\cref{lem:comparealgebras}: all that remains is to show
that $\bethe = \altbethe$ in the non-generic case.
We first prove parts \ref{main_eigenspace} and \ref{main_multiplicity}, and then
address the last case of part \ref{main_generates}.
In the process, we also obtain a new proof of \cref{thm:precise}.

\subsubsection{Fibres of the Wronski map and eigenspaces}

We now prove parts \ref{main_eigenspace} and \ref{main_multiplicity} of \cref{thm:main}, but
using $\altbethe$ in place of $\bethe$.  Once we have established part \ref{main_generates},
this will give the results as stated.

By \cref{prop:changegr}, it is enough to prove these
results in $\Gr(n,2n)$.
We have already shown that if $E \subseteq \spechtnu$ is an eigenspace of $\altbethenu$,
the complex numbers $\beta^\lambda_E$
satisfy the \Plucker relations for $\Gr(n,2n)$.  
In order to deduce that
they are the \Plucker coordinates of some point $V_E \in \Gr(n,2n)$,
we need to furthermore check these numbers are not all zero.  
This is implied by the next lemma:

\begin{lemma}
\label{lem:identifyscell}
Let $E \subseteq \spechtnu$ be an eigenspace of $\altbethenu$.  
Then $\beta^\nu_E = \frac{n!}{\numsyt\nu}$,
and $\beta^\lambda_E = 0$ for all $\lambda \not \subseteq \nu$.
\end{lemma}

\begin{proof}
Since $|\nu| = n$, we have $\beta^\nu = \alpha^\nu_{[n]}$.  By
\cref{prop:projection}, $\frac{\numsyt\nu}{n!} \beta^\nu$ 
acts on $\spechtnu$ as the identity operator.
In particular, for every eigenspace $E$, we have
$\beta^\nu_E = \frac{n!}{\numsyt\nu}$.
If $\lambda \not \subseteq \nu$, then 
$\beta^\lambda_\nu = 0$ by \cref{prop:betapsd}\ref{betapsd4}, 
and hence $\beta^\lambda_E = 0$.
\end{proof}

In particular, since $\beta^\nu_E \neq 0$, 
there exists a point $V_E \in \Gr(n,2n)$ with \Plucker coordinates 
$[\beta^\lambda_E: \lambda \subseteq \rectangle]$. Moreover, $V_E$ is contained in the Schubert cell $\scellnu$:

\begin{corollary}
\label{cor:identifyscell}
Let $E \subseteq \spechtnu$ be an eigenspace of $\altbethenu$.
Then $V_E \in \scellnu$, and
$(\beta^\lambda_E : \lambda \subseteq \nu)$ are the normalized 
\Plucker coordinates of $V_E$.
\end{corollary}

\begin{proof}
This follows immediately from \cref{lem:identifyscell} and 
\cref{thm:schubertplucker}.
\end{proof}

Next we show that $V_E \in \Wr^{-1}(g)$ for
$g(u) = (u+z_1)\dotsm (u+z_n)$:

\begin{lemma}
Let $E \subseteq \spechtnu$ be an eigenspace of $\altbethe$.  
Then $\Wr(V_E) = g$.
\end{lemma}

\begin{proof}
By \cref{cor:identifyscell} and \eqref{eq:pluckerwronskian}, we have
\[
   \Wr(V_E) = \sum_{\lambda}
    \frac{\numsyt{\lambda}}{|\lambda|!} \beta^\lambda_E u^{|\lambda|}
\,.
\]
(We can remove the condition $\lambda \subseteq \nu$ in the 
summation,
since $\beta^\lambda_E = 0$ for all $\lambda \not \subseteq \nu$.)
By the translation identity \eqref{eq:translationidentity}, the
right-hand side above equals $\beta^0_E(u)$.  
By definition 
$\beta^0(u) = g(u) \cdot \Snidentity$, and so $\beta^0_E(u) = g(u)$, as required.
\end{proof}

We have shown that $V_E \in \scellnu$ and $\Wr(V_E) = g$, which
are the first two claims of part \ref{main_eigenspace}.
We now argue that the map $E \mapsto V_E$ 
is bijective. First we show that $E \mapsto V_E$ is injective. Let $E$ and $E'$ be eigenspaces such that $[\beta^\lambda_E : \lambda \subseteq \nu] = [\beta^\lambda_{E'} : \lambda \subseteq \nu]$ as projective coordinates. By \cref{cor:identifyscell} these coordinates are normalized, so $\beta^\lambda_E = \beta^\lambda_{E'}$ for all partitions $\lambda$. Since the elements $\beta^\lambda_\nu$ generate $\altbethenu$,
this implies that $\xi_E = \xi_{E'}$ for all 
$\xi \in \altbethenu$. Hence $E = E'$, proving injectivity.

Now we prove that $E\mapsto V_E$ is surjective. By continuity of the fibres of the Wronski map $\Wr: \scellnu \to \monics$ and the
eigenvalues of the operators
$\beta^\lambda_\nu$, it suffices to prove this 
when $z_1, \dots, z_n \in \CC$ are generic.
In this case, \cref{lem:comparealgebras}\ref{comparealgebras2} and \cref{thm:bethebasics}\ref{bethebasics2}
imply that $\altbethenu = \bethenu$ is semisimple and is a maximal commutative
subalgebra of $\End(\spechtnu)$, and hence has $\numsyt\nu = \dim \spechtnu$ 
distinct eigenspaces. Since $\Wr: \scellnu \to \monics$ has degree $\numsyt\nu$, the map $E \mapsto V_E$ is surjective.

This completes the proof of \cref{thm:main}\ref{main_eigenspace}. 
Furthermore, since $\Wr$ is a finite morphism (and hence flat \cite[Exercise III.9.3(a)]{hartshorne77}),
this argument also establishes part \ref{main_multiplicity}, since
both multiplicities of points in $\Wr^{-1}(g)$ and dimensions of
generalized eigenspaces behave additively under taking limits.
\qed

\subsubsection{\texorpdfstring{$\altbethe = \bethe$}{Equality of Bethe subalgebras}}
\label{sec:altbetheequalsbethe}

We now complete the proof of \cref{thm:main}\ref{main_generates}, by
showing that $\altbethe = \bethe$ for all $z_1, \dots, z_n \in \CC$.
We already have $\bethe \subseteq \altbethe$ from 
\cref{lem:comparealgebras}\ref{comparealgebras3}; we now establish the reverse
inclusion.

\begin{lemma}
\label{lem:FDOcoords}
If $E \subseteq \spechtnu$ is an eigenspace of $\altbethenu$, then
the fundamental differential operator of $V_E$ is
\[
    D_{V_E} = \calD_E
        = \frac{1}{\Wr(V_E)} \sum_{k=0}^n \, \sum_{\ell=0}^{n-k} 
           (-1)^k\betaminus{k}_{\ell,E} u^{n-k-\ell} \du^{n-k}
\,.
\]
\end{lemma}

\begin{proof}
Recall from \eqref{eq:deomegae} that 
$(\polyD_E)_\# \Omega_E = 0$ and $\dim \ker (\polyD_E) \leq n$, where $\polyD_E$ is regarded
as a linear map $\polyD_E : \CC_{2n-1}[u] \to \CC[u]$.  
By \cref{lem:identifyscell} we have $\beta^\nu_E \neq 0$, so $\Omega_E \neq 0$. 
By \cref{cor:exterior} we furthermore deduce that 
$\dim \ker \polyD_E = n$ and
$\Omega_E = v_1 \wedge \dots \wedge v_n$, where $(v_1, \dots, v_n)$
is a basis for $\ker \polyD_E$.  By \cref{prop:exteriorplucker}, we have
$V_E = \langle v_1, \dots, v_n\rangle = \ker \polyD_E$.

Since $\calD_E = \frac{1}{\Wr(V_E)} \polyD_E$ we have
$\polyD_E f = 0$ if and only if  $\calD_E f = 0$.
Thus $V_E 
\subseteq \pker \calD_E$,
and \cref{prop:Dnullity} then gives
$V_E = \pker \calD_E$.
Finally, since $\calD_E$ is
monic, uniqueness of the fundamental differential operator
(\cref{prop:FDOunique}) implies that $D_{V_E} = \calD_E$.
\end{proof}

\begin{lemma}
\label{lem:FDOplucker2}
Let $\boldpsi_\nu := (\betaminus{k}_{\ell,\nu} : 0 \le k \le n, 0 \le \ell \le n-k)$,
and let $p^\lambda_\nu$ be the polynomials from 
\cref{lem:FDOplucker}.  Treating $z_1, \dots, z_n$ as formal
indeterminates, we have
\[ 
   \beta^\lambda_\nu = p^\lambda_\nu(\boldpsi_\nu)
\,.
\]
\end{lemma}

\begin{proof}
It suffices to prove the result when $(z_1, \dots, z_n) \in \CC^n$ is generic. In particular, by \cref{lem:comparealgebras}\ref{comparealgebras2} and \cref{thm:bethebasics}\ref{bethebasics2}, we may assume that $\altbethenu$ is semisimple.

For an eigenspace $E \subseteq \spechtnu$ of $\altbethenu$, let
$\boldpsi_E := (\betaminus{k}_{\ell,E} : 0 \le k \le n, 0 \le \ell \le n-k)$.  
\cref{lem:FDOcoords} asserts that $\boldpsi_E$ are the FDO coordinates
of $V_E$.  
Since $(\beta^\lambda_E : \lambda \subseteq \nu)$ 
are the normalized \Plucker coordinates of $V_E$, 
we have $\beta^\lambda_E = p^\lambda_\nu(\boldpsi_E)$, by the
definition of $p^\lambda_\nu$. Since $\altbethenu$ is semisimple, the result follows.
\end{proof}

\cref{lem:FDOplucker2} shows that every generator 
of $\altbethenu$ is given by a polynomial in the generators 
of $\bethenu$.
This proves that $\altbethenu \subseteq \bethenu$ for all 
$z_1, \dots, z_n \in \CC$.
From the direct product decompositions \eqref{eq:directproductbethe}
and \eqref{eq:directproductaltbethe}, 
we deduce that $\altbethe \subseteq \bethe$, as required.
\qed

\begin{remark}
\label{rmk:newproof}
Combining \cref{lem:FDOcoords} with parts \ref{main_generates} and \ref{main_eigenspace} of
\cref{thm:main},
we immediately obtain a new proof of \cref{thm:precise}.
\end{remark}


\section{Discussion and open problems}
\label{sec:discussion}

We conclude the paper by discussing several related results and open problems.


\subsection{Scheme-theoretic statements}
\label{sec:scheme}

We now give the more precise 
scheme-theoretic version of \cref{thm:main}\ref{main_eigenspace}. We consider the case when $g(u) = (u+z_1) \dotsm (u+z_n)$ has distinct roots in \cref{sec:distinct}, and the general case in \cref{sec:nondistinct}.

The \defn{\Plucker relations for $\svarnu$} are the \Plucker relations, 
where we substitute $\Delta^\lambda = 0$ for every 
partition $\lambda \not\subseteq \nu$.
Let $S_\nu := \CC[\Delta^\lambda : \lambda \subseteq \nu]/\calI_\nu$,
where $\calI_\nu$ is the ideal generated by the \Plucker relations
for $\svarnu$.
By \cref{thm:schubertplucker},
the Schubert variety $\svarnu$ is identified with $\Proj S_\nu$.
We identify the Schubert cell $\scellnu$ with
$\Spec S_\nu^\circ$, where 
$S_\nu^\circ := S_\nu / \langle \Delta^\nu - \frac{n!}{\numsyt\nu}\rangle$.
Under this identification the ring elements
$(\Delta^\lambda : \lambda \subseteq \nu)$
are the normalized \Plucker coordinates on $\scellnu$.

By \cref{thm:main}\ref{main_pluckers} and \cref{lem:identifyscell},
the elements $\beta^\lambda_\nu$ satisfy the relations in the
ideal $\calI_\nu + \langle \Delta^\nu - \frac{n!}{\numsyt\nu}\rangle$.
Thus we have a well-defined surjective $\CC$-algebra homomorphism
\[
    \Phi_\nu : S_\nu^\circ \to \bethenu, \quad \Delta^\lambda \mapsto \beta^\lambda_\nu\,,
\]
which induces a closed embedding
\[
      \Phi^*_\nu : \Spec \bethenu \to \scellnu
\,.
\]
\cref{thm:main}\ref{main_eigenspace} says that as sets,
the image of $\Phi^*_\nu$ is $\Wr^{-1}(g)$, the fibre of the
Wronski map $\Wr : \scellnu \to \monics$ over
$g$.  That is, 
$\Spec S_\nu^\circ/\ker(\Phi_\nu) \subseteq \scellnu$ 
and $\Wr^{-1}(g) \subseteq \scellnu$ have the same points.

\subsubsection{Distinct roots}
\label{sec:distinct}

When $g$ has distinct roots, the preceding statement is
also true scheme-theoretically:

\begin{theorem}
\label{thm:distinct}
If $z_1, \dots, z_n \in \CC$ are distinct, then the scheme-theoretic
image of the closed embedding $\Phi^*_\nu$ is the fibre $\Wr^{-1}(g)$. That is, $S_\nu^\circ /\ker(\Phi_\nu)$ is the coordinate ring of $\Wr^{-1}(g)$.
\end{theorem}

\begin{proof}
Let $\boldpsi := (\psi_{k,\ell} : 0 \le k \leq d,\ 0 \le \ell \leq n-k)$ 
denote the FDO coordinates on the Schubert cell $\scellnu$.
Mukhin, Tarasov, and Varchenko \cite[Theorem 4.3]{mukhin_tarasov_varchenko13} prove that the $\CC$-algebra homomorphism
$S_\nu^\circ \to \bethenu, \psi_{k,\ell} \mapsto \betaminus{k}_\ell$ induces a scheme-theoretic
isomorphism $\Spec \bethenu \to \Wr^{-1}(g)$.
By \cref{lem:FDOplucker,lem:FDOplucker2}, this $\CC$-algebra homomorphism
is $\Phi_\nu$.
\end{proof}

\subsubsection{Non-distinct roots}
\label{sec:nondistinct}

We now consider the case when $g$ has repeated roots.
Write $g(u) = (u+z_1)^{\kappa_1} \dotsm (u+z_s)^{\kappa_s}$, where $z_1, \dots, z_s \in \CC$ are distinct and $\kappa := (\kappa_1, \dots, \kappa_s)$ is a composition
of $n$.  Consider the Bethe algebra
\[
   \multibethe{n}
   :=
   \custombethe{n}(
       \underbrace{z_1, \dots, z_1}_{\kappa_1}
        \,,\,
       \underbrace{z_2, \dots, z_2}_{\kappa_2}
       \,,\,\dots\,,\,
             \underbrace{z_s, \dots, z_s}_{\kappa_s})
\,.
\]
Let $\symgrp{\kappa} := \symgrp{\kappa_1} \times \dots \times \symgrp{\kappa_s}
\subseteq \Sn$ be the Young subgroup associated to $\kappa$.  
We write $\boldmu \multipartition \kappa$ to mean that
$\boldmu = (\mu_1, \dots, \mu_s)$ is an $s$-tuple of partitions
such that $\mu_i \vdash \kappa_i$ for $i=1, \dots, s$.
The irreducible representations of $\symgrp{\kappa}$ are of the
form
\[
    \multispecht\boldmu := \specht{\mu_1} \otimes \dots \otimes \specht{\mu_s}
    \qquad \text{ for} \; \boldmu \multipartition \kappa
\,.
\]
By Schur's lemma, the Specht module $\spechtnu$ decomposes as a representation of $\symgrp\kappa$ as
\begin{equation}
\label{eq:spechtdecomposition}
      \spechtnu \simeq \bigoplus_{\boldmu \multipartition \kappa} 
      \Hom_{\symgrp\kappa}(\multispecht{\boldmu}, \spechtnu) \otimes
      \multispecht{\boldmu}
\,.
\end{equation}

It is not hard to check that $\symgrp\kappa$ commutes
with $\multibethe{n}$.  Thus the action of 
$\multibethe{n}$ on $\spechtnu$ respects the decomposition
\eqref{eq:spechtdecomposition}, preserving each summand
$\Hom_{\symgrp\kappa}(\multispecht{\boldmu}, \spechtnu) \otimes
      \multispecht{\boldmu}$, and acting trivially on the second
tensor factor.  That is, $\Hom_{\symgrp\kappa}(\multispecht{\boldmu}, \spechtnu)$
is a module for $\multibethe{n}$.
Let $\multibethe{\nu,\boldmu}$ denote the subalgebra of 
$\End\big(\Hom_{\symgrp\kappa}(\multispecht{\boldmu}, \spechtnu)\big)$ 
generated by the action, and let
$\beta^\lambda_{\nu,\boldmu}, \beta^{\lambda}_{\ell, \nu, \boldmu} \in \multibethe{\nu,\boldmu}$ denote 
the images of $\beta^\lambda, \beta^\lambda_\ell \in \multibethe{n}$.

\begin{proposition}
\label{prop:schubertprojections}
Up to a scalar multiple, the operator
\[
\beta^{\mu_1}_\nu(-z_1) \dotsm \beta^{\mu_s}_\nu(-z_s)
\in \End(\spechtnu)
\]
is the orthogonal projection
onto $\Hom_{\symgrp\kappa}(\multispecht{\boldmu}, \spechtnu) \otimes 
\multispecht{\boldmu}$.
In particular, 
each such orthogonal projection is an element of $\multibethe{\nu}$,
and we have the direct product decomposition
\[
     \multibethe\nu \simeq \prod_{\boldmu \multipartition \kappa} 
   \multibethe{\nu,\boldmu}
\,.
\]
\end{proposition}

\begin{proof}
By \eqref{eq:alphatobeta}, we can write $\beta^{\mu_i}(-z_i)$ as $c_i\frac{\numsyt{\mu_i}}{|\mu_i|!}\alpha^{\mu_i}_{X_i}$ for some nonzero scalar $c_i$, where $X_i := \{\kappa_1 {+} \cdots {+} \kappa_{i-1} {+} 1, \dots, \kappa_1 {+} \cdots {+} \kappa_i\}$. Hence, by \cref{prop:projection}, $\frac{1}{c_i}\beta^{\mu_i}_\nu(-z_i)$ acts as the scalar $\delta_{\lambda_i,\mu_i}$ on the $\boldlambda$-isotypic component of $\spechtnu$. Therefore $\frac{1}{c_1 \cdots c_s}\beta^{\mu_1}_\nu(-z_1) \dotsm \beta^{\mu_s}_\nu(-z_s)$ acts on $\Hom_{\symgrp\kappa}(\multispecht{\boldlambda}, \spechtnu) \otimes \multispecht\boldlambda$ as the scalar $\delta_{\boldlambda,\boldmu}$, as required.
\end{proof}

Since $\multibethe{\nu,\boldmu}$ is a quotient of $\multibethe{\nu}$,
we obtain a surjective $\CC$-algebra homomorphism
\[
   \Phi_{\nu,\boldmu} : S_\nu^\circ \to \multibethe{\nu,\boldmu}, \quad \Delta^\lambda \mapsto \beta^\lambda_{\nu,\boldmu}\,,
\]
which induces a closed embedding
\[
      \Phi^*_{\nu,\boldmu} : \Spec \multibethe{\nu,\boldmu} \to \scellnu
\,.
\]

\begin{theorem}
\label{thm:nondistinct}
The scheme-theoretic image of the closed embedding $\Phi^*_{\nu,\mu}$ 
is the Schubert intersection
\begin{equation}
\label{eq:generalssc}
    \svarnu \cap X_{\mu_1}(z_1) \cap \dots \cap X_{\mu_s}(z_s)
\,.
\end{equation}
The scheme-theoretic image of $\Phi^*_\nu$ is
the union of Schubert intersections
\begin{equation}
\label{eq:specbethe}
    \bigcup_{\boldmu \multipartition \kappa} 
        \svarnu \cap X_{\mu_1}(z_1) \cap \dots \cap X_{\mu_s}(z_s)
\,.
\end{equation}
\end{theorem}

\begin{proof}
The second statement follows from the first, by 
\cref{prop:schubertprojections}.
For the first statement, we recall some additional background from 
\cite{mukhin_tarasov_varchenko09b, mukhin_tarasov_varchenko13}.

To every partition $\lambda$ with $\ell(\lambda) \leq n$, 
we associate a $\gl_n$-module $V^\lambda$, which is an irreducible 
polynomial representation of the Lie algebra $\gl_n= \gl_n(\CC)$.
This representation has (up to scalar multiple) 
a unique singular vector, of weight
$\lambda \in \ZZ^n$.

Now, consider the algebra $\gl_n[t] := \gl_n \otimes \CC[t]$.
Given any $\gl_n$-module $V$ and $w \in \CC$, we can extend
the $\gl_n$-action to a $\gl_n[t]$-action, by
letting $t$ act as multiplication by $w$.  
The resulting $\gl_n[t]$-module is called 
an \emph{evaluation module} of $\gl_n[t]$, denoted by $V(w)$.

The $\gl_n$-Bethe algebra $\fullbethe \subseteq U(\gl_n[t])$
is a commutative subalgebra of the universal enveloping
algebra of $\gl_n[t]$,
which commutes with the subalgebra $U(\gl_n) \subseteq U(\gl_n[t])$.  
The algebra $\fullbethe$ is generated by
the coefficients of the \emph{universal differential operator}.
(We will not need the precise formula for this operator here; 
see \cite[Section 2.7]{mukhin_tarasov_varchenko09b} for complete details.)

We can obtain $\fullbethe$-modules by restricting any
$\gl_n[t]$-module to $\fullbethe$.
Since $\fullbethe$ commutes with $U(\gl_n)$, the $\gl_n$-weight spaces,
the spaces of $\gl_n$-singular vectors, 
and spaces of singular vectors of any weight are
$\fullbethe$-submodules.  
For a $\gl_n[t]$-module $W$, we write
$\fullbethe(W) \subseteq \End(W)$ for the algebra defined by the
action of $\fullbethe$ on $W$.
Similarly, we write $\fullbethesing(W)$ and 
$\fullbetheweight{\lambda}(W)$
for the algebras defined by the action on
the singular vectors in $W$, and the singular vectors of
weight $\lambda$ in $W$, respectively.  
Note that if $\fullbethe$
acts trivially on some vector space $M$, then $\fullbethe(W) \simeq
\fullbethe(W \otimes M)$.  Since $\fullbethe$ and $U(\gl_n)$
commute, this implies that
$\fullbethesing(W) \simeq \fullbethe(W)$ as $\CC$-algebras.

We are mainly concerned with $\gl_n[t]$-modules which are tensor
products of evaluation modules.  In particular,
for $z_1, \dots, z_s \in \CC$ and $\boldmu \multipartition \kappa$, 
consider the $\gl_n[t]$-module
\[
\Vmuz := V^{\mu_1}(z_1) \otimes \dots \otimes V^{\mu_s}(z_s)
\,.
\]
In the algebras $\fullbethe\big(\Vmuz\big) \simeq
\fullbethesing\big(\Vmuz\big)$ and
$\fullbetheweight{\nu}\big(\Vmuz\big)$, 
the universal differential operator takes the form
\[
    \frac{1}{g(u)}\sum_{k=0}^n (-1)^k \glbeta{k,\ell}
      u^{n-k-\ell} \du^{n-k}
\,,
\]
where $\glbeta{k,\ell}$ are generators of the algebra, and
$g(u) = (u+z_1)^{\kappa_1} \dotsm (u+z_s)^{\kappa_s}$.

Mukhin, Tarasov, and Varchenko prove that
$\Spec \fullbetheweight{\nu}\big(\Vmuz\big)$ is scheme-theoretically
identified with the Schubert intersection \eqref{eq:generalssc},
under the $\CC$-algebra homomorphism 
$S^\nu \to \fullbetheweight{\nu}\big(\Vmuz\big)$, 
$\psi_{k,\ell} \mapsto \glbeta{k,\ell}$
\cite[Theorem 5.13]{mukhin_tarasov_varchenko09b}.
Therefore, our task is to show that
$\fullbetheweight{\nu}(\Vmuz) \simeq \multibethe{\nu,\boldmu}$ 
under an isomorphism sending
$\glbeta{k,\ell} \mapsto
\betaminus{k}_{\ell,\nu,\boldmu}$.
The argument in the proof of \cref{thm:distinct} then
shows that $\Spec \multibethe{\nu,\boldmu}$ is scheme-theoretically 
identified with the intersection \eqref{eq:generalssc} 
under $\Phi^*_{\nu,\boldmu}$.

For $z_1, \dots, z_n \in \CC$, consider the
$\gl_n[t]$-module
\begin{equation}
\label{eq:glnt-module}
   W(z_1, \dots, z_n) := \Vone =
     \CC^n(z_1) \otimes \dots \otimes \CC^n(z_n)
\,.
\end{equation}
Note that $\CC[\Sn]$ also acts on $W(z_1, \dots, z_n)$ 
(which, as a $\gl_n$-module, is just $(\CC^n)^{\otimes n}$)
by permuting the tensor factors.
Since $\fullbethe$ commutes with $U(\gl_n)$, by Schur--Weyl duality,
the action of any element of $\fullbethe$ on $W(z_1, \dots, z_n)$ is equivalently given 
by some element of $\CC[\Sn]$.  
Thus Schur--Weyl duality identifies $\fullbethe\big(W(z_1, \dots, z_n)\big)$ 
with
some subalgebra of $\CC[\Sn]$.
Specifically, $\fullbethe\big(W(z_1, \dots, z_n)\big)$ 
is identified with $\bethe \subseteq \CC[\Sn]$,
and the elements $\glbeta{k,\ell} \in 
\fullbethe\big(W(z_1, \dots, z_n)\big)$ 
are identified with $\betaminus{k}_\ell \in \bethe$
\cite[Theorem 3.2]{mukhin_tarasov_varchenko13}.

In particular, for the $\gl_n[t]$-module $W(\boldz_\kappa)$, we have
$\multibethe{n} \simeq \fullbethe\big(W(\boldz_\kappa)\big)$.  
In this case, we can rewrite the right-hand side of \eqref{eq:glnt-module} as
\begin{equation}
\label{eq:tensordecomp}
    (\CC^n)^{\otimes \kappa_1}(z_1) \otimes \dots \otimes 
 (\CC^n)^{\otimes \kappa_s}(z_k)
    \simeq
    \bigoplus_{\boldmu \multipartition \kappa}\,
        \Vmuz \otimes \multispecht{\boldmu}
\,,
\end{equation}
where $\gl_n[t]$ acts trivially on $\multispecht{\boldmu}$ 
and $\symgrp{\kappa}$ acts trivially on $\Vmuz$.
Therefore
\begin{equation}
\label{eq:bethesingdecomp}
    \fullbethesing\big(W(\boldz_\kappa)\big)
    \simeq  \prod_{\boldmu \multipartition \kappa} 
            \fullbethesing\big(\Vmuz\big)
    \simeq  \prod_{\substack{\nu \vdash n, \\ \boldmu \multipartition \kappa}}
       \fullbetheweight{\nu}\big(\Vmuz\big)
\,.
\end{equation}
The first isomorphism in \eqref{eq:bethesingdecomp} is obtained by projecting onto
the $\boldmu$-isotypic component of the 
$\symgrp{\kappa}$-action in the decomposition \eqref{eq:tensordecomp}, 
for each $\boldmu \multipartition \kappa$.
The second isomorphism in \eqref{eq:bethesingdecomp}
is obtained by further projecting onto
the singular vectors of weight $\nu$, for each $\nu \vdash n$;
by Schur--Weyl duality this is
the same as projecting onto the $\spechtnu$-isotypic component of the
$\Sn$-action on $(\CC^n)^{\otimes n}$.  
(These projections are contained in $\fullbethe\big(W(\boldz_\kappa)\big)$
by \cref{prop:schubertprojections}.)  

But this is the same way that we obtain the decomposition
$\multibethe{n} \simeq \prod_{\nu, \boldmu} \multibethe{\nu, \boldmu}$.
Thus, the identification 
$\fullbethesing\big(W(\boldz_\kappa)\big) \simeq 
\fullbethe\big(W(\boldz_\kappa)\big) \simeq \multibethe{n}$
also identifies components 
$\fullbetheweight{\nu}\big(\Vmuz\big) \simeq \multibethe{\nu,\boldmu}$.
Since the latter identification is obtained by projections of the former,
the generators
$\glbeta{k,\ell} \in \fullbetheweight{\nu}\big(\Vmuz\big)$ are
identified with generators $\betaminus{k}_{\ell,\nu,\boldmu} 
\in \multibethe{\nu,\boldmu}$, as required.
\end{proof}

We take this opportunity to state a natural problem:
\begin{problem}\label{prob:natural}
Find explicit formulas for universal \Plucker coordinates in
the $\gl_n$-Bethe algebra, which coincide with the operators
$\beta^\lambda(s)$ on the $\gl_n[t]$-modules $\Vmuz$.
\end{problem}

After a preliminary version of this paper appeared, an answer to \cref{prob:natural} was provided in \cite[Remark 8.3]{karp_mukhin_tarasov25}; it remains open whether it is possible to find a less unwieldy formula.

\subsubsection{Dimensions of Bethe algebras}
\label{sec:dimensions}

We can now calculate the dimensions of the Bethe algebras $\multibethe{\nu}$ and $\multibethe{n}$:
\begin{theorem}
\label{thm:bethedimensions}
For $k \geq 0$, let 
$\ssum_k := \sum_{\lambda \vdash k} \s_\lambda$.  Then 
\begin{equation}
\label{eq:dimensionformula}
\dim \multibethe{\nu} = \langle \s_\nu \,,\,
\ssum_{\kappa_1} \dotsm \ssum_{\kappa_s} \rangle
\qquad
\text{and}
\qquad
\dim \multibethe{n} = \langle \ssum_n \,,\,
\ssum_{\kappa_1} \dotsm \ssum_{\kappa_s} \rangle
\,.
\end{equation}
\end{theorem}

\begin{proof}
By \cref{thm:nondistinct}, 
$\dim \multibethe{\nu, \boldmu}$ is
the length of the Schubert intersection \eqref{eq:generalssc}, as a scheme.
As in \cref{sec:schubert}, this length is
$\langle \s_\nu, \s_{\mu_1} \dotsm \s_{\mu_s}\rangle$ if the
intersection is transverse.  Since the Schubert varieties are smooth 
at the points of intersection, this remains true whenever the 
intersection is proper \cite[Proposition 8.2]{fulton98}.  Hence,
\begin{equation}\label{eq:dimensionformularefined}
\dim \multibethe{\nu, \boldmu} =
\langle \s_\nu\,, \s_{\mu_1} \cdots \s_{\mu_s}\rangle
\,.
\end{equation}
Summing over all $\boldmu \multipartition \kappa$ (and for the second
formula, also over $\nu \vdash n$)
we obtain the formulas \eqref{eq:dimensionformula}.
\end{proof}

We now discuss why $\Spec \multibethe{\nu}$ is (in some cases) scheme-theoretically different from the fibre $\Wr^{-1}(g)$, for $g(u) = (u+z_1)^{\kappa_1} \dotsm (u+z_s)^{\kappa_s}$. 
Note that \eqref{eq:specbethe} is the right-hand side of \eqref{eq:fibre},
which equals $\Wr^{-1}(g)$ set-theoretically, and so $\Spec \multibethe{\nu}$
is always set-theoretically identified with $\Wr^{-1}(g)$.
The Wronski map is a finite morphism of degree $\numsyt\nu$, 
so the fibre $\Wr^{-1}(h)$ is a finite scheme of length $\numsyt\nu$
for all $h \in \monics$.  On the other hand, by definition,
$\Spec \multibethe{\nu}$ is a finite scheme of length $\dim \multibethe{\nu}$, which (depending on $\kappa$) may be strictly less than $\numsyt{\nu}$.
Since \eqref{eq:specbethe} is always a subscheme of $\Wr^{-1}(g)$, 
the two scheme structures coincide if and only if 
$\dim \multibethe{\nu} = \numsyt\nu$.

We can also compare the multiplicities of individual points:

\begin{theorem}
\label{thm:comparemultiplicity}
Let $E \subseteq \spechtnu$ be an eigenspace of 
$\multibethe{\nu}$, and let $\boldmu \multipartition \kappa$
be such that $V_E$ 
belongs to the Schubert intersection \eqref{eq:generalssc}.  
Then the multiplicity of
$V_E$ viewed as a point of $\Wr^{-1}(g)$ equals 
$\dim \multispecht{\boldmu}$ times the multiplicity of $V_E$ 
viewed as a point of \eqref{eq:generalssc}. In particular, the two multiplicities agree if and only if $\dim \multispecht{\boldmu}=1$, i.e., every $\mu_i$ is a single row or column.
\end{theorem}

\begin{proof}
By \cref{thm:main}\ref{main_multiplicity}, the multiplicity of
$V_E$, viewed as a point of the fibre $\Wr^{-1}(g)$, is $\dim \widehat E$.  
Under the decomposition \eqref{eq:spechtdecomposition}, 
we must have $E, \widehat E \subseteq 
\Hom_{\symgrp\kappa}(\multispecht{\boldmu}, \spechtnu) \otimes 
\multispecht{\boldmu}$. Since $\multibethe{\nu}$ acts trivially
on the second tensor factor, we have
$E = E_0 \otimes \multispecht{\boldmu}$ and
$\widehat E = \widehat E_0 \otimes \multispecht{\boldmu}$, for
some subspaces $E_0, \widehat E_0 \subseteq
\Hom_{\symgrp\kappa}(\multispecht{\boldmu}, \spechtnu)$.
Then $E_0$ is an eigenspace of the algebra 
$\multibethe{\nu, \boldmu}$, and $\widehat E_0$ is the corresponding
generalized eigenspace.
By \cref{thm:nondistinct}, the multiplicity of $V_E = V_{E_0}$,
viewed as a point of the Schubert intersection 
\eqref{eq:generalssc}, is $\dim \widehat E_0$.
Hence the two notions of multiplicity differ by a factor of 
$\dim \multispecht{\boldmu}$.
\end{proof}

Informally, in
the fibre $\Wr^{-1}(g)$, there is some non-trivial geometry associated
with the multiplicity spaces $\multispecht{\boldmu}$
of the decomposition
\eqref{eq:spechtdecomposition}, but since 
$\multibethe{\nu}$ acts trivially on these spaces, this is not reflected
in the geometry of $\Spec \multibethe{\nu}$.

We note the following corollary, which was
also effectively observed in \cite{mukhin_tarasov_varchenko13}:
\begin{corollary}[{cf.\ \cite[Remark p.\ 776]{mukhin_tarasov_varchenko13}}]
\label{cor:schemeequality}
The following are equivalent:
\begin{enumerate}[(a)]
\item\label{schemeequality1} the equality \eqref{eq:fibre} holds scheme-theoretically;
\item\label{schemeequality2} $\dim \multibethe{\nu} = \numsyt{\nu}$;
\item\label{schemeequality3} $\nu$ equals $n$ or $1^n$, or $\kappa_i \le 2$ for all $1 \le i \le s$.
\end{enumerate}
\end{corollary}

\begin{proof}
The equivalence of \ref{schemeequality1} and \ref{schemeequality2} is discussed above (before \cref{thm:comparemultiplicity}).
The equivalence of \ref{schemeequality2} and \ref{schemeequality3} follows since both conditions are equivalent to
$\dim \multispecht{\boldmu} = 1$ 
for all $\boldmu \multipartition \kappa$ such that 
$\Hom_{\symgrp\kappa}(\multispecht{\boldmu}, \spechtnu) \neq 0$.
\end{proof}


\subsection{Bases for \texorpdfstring{$V_E$}{V\textunderscore{}E}}
\label{sec:bases}

\cref{thm:main} gives us the points $V_E \in \Wr^{-1}(g)$ in terms
of their \Plucker coordinates.  
We now describe two ways to obtain a basis for $V_E$ (see \cref{thm:betabasis,thm:generalbasis}), and make several related remarks.

\subsubsection{\Plucker-coordinate basis}
\label{sec:basisplucker}

We can obtain a basis for $V_E \in \scellnu \subseteq \Gr(d,m)$
using \cref{prop:pluckerbasis}, corresponding to a matrix representative in
reduced row-echelon form.
Following our convention \eqref{eq:identification}, we index the \Plucker coordinates
$\beta^\lambda_E$ using $\binom{[m]}{d}$ (extended
to $[m]^d$ by the alternating property).  So in this notation, 
$\big(\beta_{I,E}: I \in {m \choose [d]}\big)$ are the normalized \Plucker
coordinates of $V_E$, by \cref{cor:identifyscell}.

\begin{theorem}
\label{thm:betabasis}
Let $E \subseteq \spechtnu$ be an eigenspace of $\bethenu$, and
let $J = (j_1, \dots, j_d) := (\nu_d+1, \dots, \nu_1+d) \in\binom{[m]}{d}$ 
correspond to $\nu$ as in \eqref{eq:identification}. For $i \in [d]$,
define
\begin{equation}
\label{eq:betabasis}
  f_i(u) := \frac{\numsyt\nu}{n!} \sum_{k=1}^{j_i} 
           \beta_{(J-j_i)+k, E} \frac{u^{k-1}}{(k-1)!} 
\,.
\end{equation}
Then $(f_1, \dots, f_d)$ is the unique basis for $V_E$ such that for all $i\in [d]$ and $k \in J\setminus \{j_i\}$, we have
\[ 
     f_i(u) = \frac{u^{j_i-1}}{(j_i-1)!} +\, \text{lower-degree terms}
     \qquad \text{and} \qquad
     [u^{k-1}] f_i(u) = 0
\,.
\]
\end{theorem}

\begin{proof}
Since $\big(\beta_{I,E}: I \in {m \choose [d]}\big)$ are the normalized \Plucker
coordinates of $V_E \in \scellnu$,
the fact that $(f_1, \dots, f_d)$ is a basis for $V_E$ follows from
\cref{prop:pluckerbasis}.
By \cref{thm:schubertplucker} we have $\beta_{(J-j_i)+k,E} =0$
for all $k > j_i$, so we can reduce the upper index of summation 
from $m$ to $j_i$ in the definition of $f_i$.

By definition of $\beta_{(J-j_i)+k}$, we have $[u^{k-1}] f_i(u) = 0$ if
$k \in J\setminus \{j_i\}$.
Since the \Plucker coordinates are normalized, the 
constant factor of $\frac{\numsyt\nu}{n!}$ ensures that 
$\big[\frac{u^{j_i-1}}{(j_i-1)!}\big]f_i(u) = 1$.
\end{proof}

\begin{example}
\label{ex:betabasis}
For $V \in \scell{2} \subseteq \Gr(2,4)$ as in \cref{ex:pluckers}, let us find the basis $(f_1, f_2)$ from \cref{thm:betabasis}. We have $J = (1, 4)$, and the \Plucker coordinates $\beta_{I,E}$ (where $V = V_E$) are given by \eqref{eq:examplepluckers}. We calculate that $f_1(u) = \frac{1}{2}\beta_{(J - 1) + 1, E} = 1$ and
\[
f_2(u) = \scalebox{0.94}{$\displaystyle\frac{1}{2}\Big(\beta_{(J - 4) + 1, E} + \beta_{(J - 4) + 2, E}u + \beta_{(J - 4) + 3, E}\frac{u^2}{2} + \beta_{(J - 4) + 4, E}\frac{u^3}{6}\Big)
= \frac{1}{6}u^3 + \frac{z_1+z_2}{4}u^2 + \frac{z_1z_2}{2}u$}\,.
\]
\end{example}

\subsubsection{A Markov basis exhibiting disconjugacy}
\label{sec:basismarkov}

Let $V \subseteq \RR[u]$ be finite-dimensional vector space of real polynomials, and let $I\subseteq\RR$ be an interval. We say that an ordered basis $(f_1, \dots, f_d)$ for $V$ is a \defn{Markov basis on $I$} if
\[
\Wr(f_1, \dots, f_i)\, \text{ is nonzero on $I$ for $i = 1, \dots, d$}\,.
\]
If $V$ has a Markov basis on $I$ then $V$ is disconjugate on $I$, and the converse holds if $I$ is open or compact. (This follows from work of Markov \cite[Section 1]{markoff04}, Hartman \cite[Proposition 3.1]{hartman69}, and Zielke \cite[Theorem 23.3]{zielke79}; see \cite[Section 4.1]{karp} for further discussion and background.)

Recall that the disconjugacy conjecture (\cref{thm:disconj}) asserts that if $\Wr(V)$ has only real zeros, then $V$ is disconjugate on every interval $I\subseteq\RR$ which avoids the zeros of $\Wr(V)$. Without loss of generality, we may assume that $I$ is open. Then using the $\PGL_2(\RR)$-action, it suffices to consider the case $I = (0,\infty)$. The general theory cited above implies that $V$ has a Markov basis on $I$, but it does not explicitly provide us with one. We observe that \cref{thm:betabasis} provides just such a basis:
\begin{proposition}
\label{prop:disconjbasis}
Let $V \in \scellnu \subseteq \Gr(d,m)$ such that all the zeros of $\Wr(V)$ are real and nonpositive. Then the basis $(f_1, \dots, f_d)$ for $V$ defined in \eqref{eq:betabasis} is a Markov basis on $(0,\infty)$.
\end{proposition}

\begin{proof}
We adopt the notation of \cref{thm:betabasis}. For $i\in [d]$, let $V_i \in \Gr(i,j_i)$ be the span of $(f_1, \dots, f_i)$. Then by construction, we have
\[
\Delta_I(V_i) = \Delta_{I \cup \{j_{i+1}, \dots, j_d\}}(V) \qquad \text{for all } I\in\textstyle\binom{[j_i]}{i}\,.
\]
Since $V$ is totally nonnegative (by \cref{thm:positive}\ref{positive1}), so too is $V_i$. Therefore $\Wr(V_i)$ has nonnegative coefficients by \eqref{eq:pluckerwronskian}, and hence it is nonzero on $(0,\infty)$.
\end{proof}

The proof of \cref{prop:disconjbasis} appears to establish a stronger property of the basis $(f_1, \dots, f_d)$ than claimed: not only is $\Wr(f_1, \dots, f_i)$ nonzero on $(0,\infty)$, it has nonnegative coefficients. In fact, the existence of such a basis is guaranteed by the results of \cite{karp}. However, a genuinely stronger property would be $\Wr(f_1, \dots, f_i)$ having only real roots:
\begin{problem}\label{prob:zerobasis}
Let $V \subseteq \RR[u]$ be a finite-dimensional vector space of polynomials such that all the zeros of $\Wr(V)$ are real and contained in the interval $I \subseteq \RR$. Does there exist a basis $(f_1, \dots, f_d)$ for $V$ such that for $1 \le i \le d$, all the zeros of $\Wr(f_1, \dots, f_i)$ are real and contained in $I$?
\end{problem}

After a preliminary version of this paper appeared, David Speyer sent us the following example which shows that the answer to \cref{prob:zerobasis} is `no' in general.
\begin{example}\label{ex:zerobasis}
Let $d=2$ and set $V = \langle h_1, h_2\rangle \subseteq\RR[u]$, where
\[
h_1(u) = u^4 - \frac{15}{4}u^2 + 4\,, \qquad h_2(u) = u^5 - \frac{15}{16}u^3 + \frac{1}{4}u\,.
\]
We have $\Wr(V) = u^8 - \frac{165}{16}u^6 + \frac{1457}{64}u^4 - \frac{165}{16}u^2 + 1$, whose zeros are all real:
\[
\pm2.7266\cdots,\quad \pm1.5201\cdots,\quad \pm0.6578\cdots,\quad \pm0.3667\cdots.
\]
We claim that every nonzero $f\in V$ has a nonreal zero. This implies that our choice of $V$ answers \cref{prob:zerobasis} in the negative (by taking $i=1$).

To prove our claim, note that $h_1$ has no real zeros, so we are done if $f$ is a scalar multiple of $h_1$. Otherwise, after rescaling $f$ we can write $f = h_2 - ch_1$ for some $c\in\RR$. Since $f$ has degree $5$, it suffices to show that $f$ has at most $3$ real zeros. The zeros of $f$ are precisely the solutions of the equation $\frac{h_2}{h_1} = c$, and we can verify by inspecting \cref{fig:zerobasis} that every horizontal line intersects the graph of $\frac{h_2}{h_1}$ at most $3$ times (keeping in mind that the local extrema of $\frac{h_2}{h_1}$ correspond precisely to the $8$ zeros of $\Wr(V)$).
\begin{figure}[t]
\begin{center}
\begin{tikzpicture}[baseline=(current bounding box.center),xscale=0.54,yscale=0.36]
\tikzstyle{vertex}=[inner sep=0,minimum size=1.5mm,circle,draw=red,fill=red]
\pgfmathsetmacro{\xmin}{-5.3};
\pgfmathsetmacro{\xmax}{5.3};
\pgfmathsetmacro{\ymin}{-8.3};
\pgfmathsetmacro{\ymax}{8.3};
\pgfmathsetmacro{\xtick}{0.3};
\pgfmathsetmacro{\ytick}{0.2};
\draw[latex'-latex',thick](\xmin,0)--(\xmax,0);
\draw[latex'-latex',thick](0,\ymin)--(0,\ymax);
\draw[thick](4,\xtick)--(4,-\xtick);
\draw[thick](-4,\xtick)--(-4,-\xtick);
\draw[thick](\ytick,6)--(-\ytick,6);
\draw[thick](\ytick,-6)--(-\ytick,-6);
\node[inner sep=0]at(4,0)[label={[below=2pt]$4$}]{};
\node[inner sep=0]at(-4,0)[label={[below=2pt]$-4$}]{};
\node[inner sep=0]at(0,6)[label={[left=0pt]$6$}]{};
\node[inner sep=0]at(0,-6)[label={[left=0pt]$-6$}]{};
\node[inner sep=0]at(\xmax,0)[label={[right=-1pt]$u$}]{};
\begin{scope}
\clip(\xmin,\ymin)rectangle(\xmax,\ymax);
\draw[-latex',thick,color=black,domain=0:\xmax-0.1,smooth,variable=\x]plot(\x,{(\x^5-(15/16)*\x^3+(1/4)*\x)/(\x^4-(15/4)*\x^2+4)});
\draw[-latex',thick,color=black,domain=0:\xmax-0.1,smooth,variable=\x]plot(-\x,{-(\x^5-(15/16)*\x^3+(1/4)*\x)/(\x^4-(15/4)*\x^2+4)});
\end{scope}
\node[vertex]at(2.726623608563005,4.217040756420609){};
\node[vertex]at(-2.726623608563005,-4.217040756420609){};
\node[vertex]at(1.520117274329124,7.717747686201651){};
\node[vertex]at(-1.520117274329124,-7.717747686201651){};
\node[vertex]at(0.6578439814397424,0.008098217581243529){};
\node[vertex]at(-0.6578439814397422,-0.008098217581243558){};
\node[vertex]at(0.3667539578471646,0.014820819529629001){};
\node[vertex]at(-0.36675395784716464,-0.014820819529628996){};
\end{tikzpicture}
\hspace*{36pt}
\begin{tikzpicture}[baseline=(current bounding box.center),xscale=2,yscale=80]
\tikzstyle{vertex}=[inner sep=0,minimum size=1.5mm,circle,draw=red,fill=red]
\pgfmathsetmacro{\xmin}{-1};
\pgfmathsetmacro{\xmax}{1};
\pgfmathsetmacro{\ymin}{-0.03};
\pgfmathsetmacro{\ymax}{0.03};
\pgfmathsetmacro{\xtick}{0.00135};
\pgfmathsetmacro{\ytick}{0.054};
\draw[latex'-latex',thick](\xmin,0)--(\xmax,0);
\draw[latex'-latex',thick](0,\ymin)--(0,\ymax);
\draw[thick](0.5,\xtick)--(0.5,-\xtick);
\draw[thick](-0.5,\xtick)--(-0.5,-\xtick);
\draw[thick](\ytick,0.02)--(-\ytick,0.02);
\draw[thick](\ytick,-0.02)--(-\ytick,-0.02);
\node[inner sep=0]at(0.5,0)[label={[below=2pt]$\frac{1}{2}$}]{};
\node[inner sep=0]at(-0.5,0)[label={[below=2pt]$-\frac{1}{2}$}]{};
\node[inner sep=0]at(0,0.02)[label={[left=0pt]$\frac{1}{50}$}]{};
\node[inner sep=0]at(0,-0.02)[label={[left=0pt]$-\frac{1}{50}$}]{};
\node[inner sep=0]at(\xmax,0)[label={[right=-1pt]$u$}]{};
\begin{scope}
\clip(\xmin,\ymin)rectangle(\xmax,\ymax);
\draw[-latex',thick,color=black,domain=0:0.815,smooth,variable=\x]plot(\x,{(\x^5-(15/16)*\x^3+(1/4)*\x)/(\x^4-(15/4)*\x^2+4)});
\draw[-latex',thick,color=black,domain=0:0.815,smooth,variable=\x]plot(-\x,{-(\x^5-(15/16)*\x^3+(1/4)*\x)/(\x^4-(15/4)*\x^2+4)});
\end{scope}
\node[vertex]at(0.6578439814397424,0.008098217581243529){};
\node[vertex]at(-0.6578439814397422,-0.008098217581243558){};
\node[vertex]at(0.3667539578471646,0.014820819529629001){};
\node[vertex]at(-0.36675395784716464,-0.014820819529628996){};
\end{tikzpicture}
\caption{The graph of the rational function $\frac{h_2(u)}{h_1(u)}$ from \cref{ex:zerobasis}, with the $8$ local extrema highlighted. The plot on the right is zoomed in around the origin.}\label{fig:zerobasis}
\end{center}
\end{figure}

(Above we implicitly assumed that $f$ has distinct zeros. In general, we observe that if $z\in\CC$ is a zero of $f$ of multiplicity $d\ge 2$, then taking the derivative of $\frac{h_2}{h_1}$ shows that $z$ is a zero of $\Wr(V)$ of multiplicity at least $d-1$. Since $\Wr(V)$ has distinct zeros, all multiple solutions of $\frac{h_2}{h_1} = c$ have multiplicity $2$ and are zeros of $\Wr(V)$, i.e., they correspond to local extrema of $\frac{h_2}{h_1}$. Hence the graphical argument goes through.)
\end{example}

\subsubsection{A basis independent of the Schubert cell}
\label{sec:basisindependent}

We now give a second basis for $V_E \in \Gr(d,m)$, which does not depend on the ambient Schubert cell $\scellnu$.
For $t \in \CC$, consider the polynomial
\[
   h_t(u) := \sum_{k=d}^{m} \beta^{k-d}(-t) \otimes e_{k,t}(u) =  \sum_{k=d}^{m} \beta^{k-d}(-t) \otimes 
                  \frac{(u+t)^{k-1}}{(k-1)!}
\]
in $\bethe \otimes \CC_{m-1}[u]$, for $e_{i,t}(u)$ as in \eqref{eq:ei}.
The coefficients $\beta^{k}(-t)$ are indexed by single-row partitions.
Similarly, for an eigenspace $E \subseteq \spechtnu$ of $\bethenu$, let
\begin{equation}
\label{eq:generalbasisdef}
   h_{t,E}(u) := \sum_{k=d}^{m} \beta_E^{k-d}(-t)
                  \frac{(u+t)^{k-1}}{(k-1)!}
       \in \CC_{m-1}[u]
\,.
\end{equation}
Viewing $h_t(u)$ as a function of $t$, let 
$h^{(j)}_{t}(u) := \frac{\partial^j}{\partial t^j} h_{t}(u)$ and
$h^{(j)}_{t,E}(u) := \frac{\partial^j}{\partial t^j} h_{t,E}(u)$
denote the $j$th partial derivatives with respect to $t$.

\begin{theorem}
\label{thm:generalbasis}
Let $V_E \in \Gr(d,m)$ correspond to an eigenspace $E \subseteq \spechtnu$ of $\bethenu$.
Then $V_E$ is spanned by the polynomials $h_{t,E}(u)$ for $t \in \CC$.  Furthermore,
\begin{equation}
\label{eq:generalbasis}
    \big(h_{t,E}(u), h^{(1)}_{t,E}(u), \dots, h^{(d-1)}_{t,E}(u)\big)
\end{equation}
is a basis for $V_E$ for every $t \in \CC\setminus\{z_1, \dots, z_n\}$.
\end{theorem}

We consider an example of \cref{thm:generalbasis} before giving its proof.
\begin{example}
\label{ex:generalbasis}
We illustrate \cref{thm:generalbasis} in the case $n=2$, for $\Gr(2,4)$. Calculating as in \cref{ex:main}, we find
\[
h_t(u) = (z_1 - t)(z_2 - t) \otimes (u+t) + (z_1 + z_2 - 2t) \otimes \frac{(u+t)^2}{2} + (1 + \trans{1}{2}) \otimes \frac{(u + t)^3}{6}\,,
\]
where we have identified $\Sidentity{2}$ with $1$ for convenience. The two associated eigenspaces are $E = M^2$ and $E = M^{11}$, whose corresponding elements $V_E \in \Gr(2,4)$ were discussed in \cref{ex:main}. \cref{thm:generalbasis} asserts that $(h_{t,E}(u), h^{(1)}_{t,E}(u))$ is a basis for $V_E$ for any $t\in\CC\setminus\{z_1, z_2\}$. For $E = M^2$, the element $\trans{1}{2} \in \symgrp{2}$ acts as $1$, and so we obtain the following basis for $V_E$:
\begin{align*}
h_{t,E}(u) &= (z_1 - t)(z_2 - t)(u+t) + \frac{z_1 + z_2 - 2t}{2}(u+t)^2 + \frac{1}{3}(u + t)^3\,, \\
h^{(1)}_{t,E}(u) &= \frac{\partial}{\partial t}h_{t,E}(u) = (z_1 - t)(z_2 - t)\,.
\end{align*}
Similarly, for $E = M^{11}$ we obtain the following basis for $V_E$:
\begin{align*}
h_{t,E}(u) &= (z_1 - t)(z_2 - t)(u+t) + \frac{z_1 + z_2 - 2t}{2}(u+t)^2 \,, \\
h^{(1)}_{t,E}(u) &= \frac{\partial}{\partial t}h_{t,E}(u) = (z_1 - t)(z_2 - t) - (u+t)^2\,.
\end{align*}
\end{example}

We now prove \cref{thm:generalbasis}.
The polynomials $h^{(j)}_t(u)$ satisfy a translation property: sending $(z_1, \dots, z_n) \mapsto (z_1-t, \dots, z_n-t)$ and $u\mapsto u+t$ together takes $h^{(j)}_s(u)$ to $h^{(j)}_{s+t}(u)$. Hence it will suffice to work with $h^{(j)}_t(u)$ when $t=0$. Also note that
\begin{equation}
\label{eq:inlinearspan}
h^{(j)}_t(u) \in \langle h_s(u) \mid s \in \CC \rangle \qquad \text{for all $j\ge 0$ and $t\in\CC$}\,.
\end{equation}
\begin{lemma}
\label{lem:leadingterm}
Let $z_1, \dots, z_n \in\CC\setminus\{0\}$. Then for $0 \le j \le d-1$, the term in $h^{(j)}_0(u)$ of minimum $u$-degree is $\beta^0 \otimes \frac{u^{d-j-1}}{(d-j-1)!} = (z_1z_2 \cdots z_n)\Snidentity \otimes \frac{u^{d-j-1}}{(d-j-1)!}$.
\end{lemma}

\begin{proof}
This follows by a direct calculation.
\end{proof}

For the element $\Omega$ defined in \eqref{eq:omega}, we have the following result:
\begin{proposition}
\label{prop:omegawedge}
Suppose that $(d,m) = (n, 2n)$. Then
\[
\Omega \wedge h^{(j)}_t(u) = 0 \qquad \text{ for all $j\ge 0$ and $t \in \CC$}\,.
\]
\end{proposition}

\begin{proof}
It suffices to establish the result when $z_1, \dots, z_n$ are formal indeterminates.
First we prove that $\Omega \wedge h_0(u) = 0$. We calculate that
\begin{equation}
\label{eq:singlerow}
\Omega \wedge h_0(u) = \sum_{J \in {[2n] \choose n}}\,\sum_{k=n}^{2n} e_J \wedge e_k \otimes \beta_J\beta^{k-n} = \sum_{I \in \binom{[2n]}{n+1}}e_I \otimes \sum_{k=n}^{2n} \beta_{I-k}\beta^{k-n}\,,
\end{equation}
where $\beta_{I-k}$ is defined analogously to \eqref{eq:deletesubscript}. This equals $0$ if and only if the operators $\beta^\lambda$ satisfy all the single-row \Plucker relations. These hold by \cref{thm:main}\ref{main_pluckers}.

Now for $t\in\CC$, apply the change of variables $(z_1, \dots, z_n) \mapsto (z_1-t, \dots, z_n-t)$ and $u\mapsto u+t$. This takes $h_0(u)$ to $h_t(u)$, and $\Omega$ remains unchanged by \cref{lem:exteriortranslation}. Hence the equation $\Omega \wedge h_0(u) = 0$ implies $\Omega \wedge h_t(u) = 0$.
By \eqref{eq:inlinearspan}, we deduce 
that $\Omega \wedge h^{(j)}_t(u) = 0$ for all $j\ge 0$.
\end{proof}

\begin{corollary}
\label{cor:contains}
Let $V_E \in \Gr(d,m)$ correspond to an eigenspace $E \subseteq \spechtnu$ of $\bethenu$.
Then $h^{(j)}_{t,E}(u) \in V_E$ for all $j\ge 0$ and $t\in\CC$.
\end{corollary}

\begin{proof}
In the case that $(d,m) = (n,2n)$, this follows from \cref{thm:main}\ref{main_eigenspace} and \cref{prop:omegawedge}. Now we explain why the statement does not depend on the choice of $(d,m)$ such that $\nu \subseteq \rectangle = (m-d)^d$, whence the general result follows. Since $\nu_1 \le m-d$, by \cref{prop:betapsd}\ref{betapsd4} we have $\beta^k_E = 0$ for all $k > m-d$. Then we see from \eqref{eq:generalbasisdef} that $h^{(j)}_{t,E}(u)$ does not depend on $m$. Similarly, sending $d \mapsto d-1$ takes $h^{(j)}_{t,E}(u) \mapsto \du h^{(j)}_{t,E}(u)$, and sending $d\mapsto d+1$ takes $h^{(j)}_{t,E}(u) \mapsto \int_{-t}^uh^{(j)}_{t,E}(s)ds$. Therefore by \cref{prop:changegr}, the statement does not depend on $d$.
\end{proof}

\begin{proof}[Proof of \cref{thm:generalbasis}]
First we show that \eqref{eq:generalbasis} is a basis for $V_E$ for every $t \in \CC\setminus\{z_1, \dots, z_n\}$. By translation, it suffices to consider the case when $t=0$ (and $z_1, \dots, z_n \neq 0$). By \cref{lem:leadingterm}, for all $0 \le j \le d-1$ we have
\begin{equation}
\label{eq:leadingterm}
     h^{(j)}_{0,E}(u) = z_1z_2 \cdots z_n \frac{u^{d-j-1}}{(d-j-1)!} +\, \text{higher-degree terms}\,.
\end{equation}
Since $z_1z_2 \cdots z_n \neq 0$, we deduce that \eqref{eq:generalbasis} is linearly independent. By \cref{cor:contains}, the set \eqref{eq:generalbasis} is contained in $V_E$, so it is a basis for $V_E$.

Now let $W \subseteq \CC_{m-1}[u]$ be the subspace spanned by the polynomials $h_{t,E}(u)$ for $t\in\CC$. It remains to show that $W = V_E$. By \cref{cor:contains}, we have $W \subseteq V_E$. Conversely, by \eqref{eq:inlinearspan}, we have $h^{(j)}_{t,E}(u) \in W$ for all $j\ge 0$ and $t\in\CC$. Since \eqref{eq:generalbasis} spans $V_E$ for any $t \in \CC\setminus\{z_1, \dots, z_n\}$, we get $V_E \subseteq W$.
\end{proof}

\begin{remark}
In \cref{thm:generalbasis}, the assumption $t\neq z_1, \dots, z_n$ is necessary in order for \eqref{eq:generalbasis} to be a basis when $d\ge 2$. Indeed, \cref{thm:generalbasis} implies that
\[
\Wr\big(h_{t,E}(u), h^{(1)}_{t,E}(u), \dots, h^{(d-1)}_{t,E}(u)\big) = c(t) \Wr(V_E) \qquad \text{for some } c(t)\in\CC\,.
\]
Evaluating at $u = -t$ and using \eqref{eq:leadingterm}, we find $c(t) = (-1)^{\binom{d}{2}}(z_1-t)^{d-1} \cdots (z_n-t)^{d-1}$. If $t\in\{z_1, \dots, z_n\}$ then the Wronskian of \eqref{eq:generalbasis} is zero, and so \eqref{eq:generalbasis} is linearly dependent.
\end{remark}

\hidelinks
\subsubsection{Dual proof of \texorpdfstring{\fullcref{thm:main}\ref{main_pluckers}}{Theorem \ref{thm:main}\ref{main_pluckers}}}
\restorelinks

An alternative approach to proving \cref{thm:main}\ref{main_pluckers} takes
\cref{thm:generalbasis} as the definition of $V_E$.  
The argument is essentially dual to the proof in \cref{sec:pr}, in the sense of \cref{sec:duality}. We give a brief sketch of the main ideas.

We work in $\Gr(n,2n)$.
Whereas the proof in \cref{sec:pr} is based on the identity
$(\polyD_n)_\# \Omega = 0$, we instead proceed by showing
that $\Omega \wedge h_t(u) = 0$ for all $t \in \CC$. As in \eqref{eq:singlerow}, this
is equivalent to the operators 
$\beta^\lambda$ satisfying all single-row \Plucker relations.
These hold by
an argument very similar to our proof of the single-column \Plucker relations in \cref{sec:part1proof}. We deduce that $\Omega_E \wedge h_{t, E}(u) = 0$ for every eigenspace $E$.

We now define $V_E \subseteq \CC_{2n-1}[u]$ to be the span of the polynomials
$h_{t,E}(u)$ for $t \in \CC$.
As in the proof of \cref{thm:generalbasis}, we can argue that 
$\dim V_E \geq n$.
Then standard properties of the exterior algebra imply that 
$\dim V_E = n$ and $\Omega_E = v_1 \wedge \dots \wedge v_n$, where 
$(v_1, \dots, v_n)$ is a basis for $V_E$.  Hence by 
\cref{prop:exteriorplucker}, the coefficients of
$\Omega_E$ satisfy the \Plucker relations.  The remainder of the
proof is identical to the one given in \cref{sec:screduction}.

However, using this alternate definition of $V_E$ creates some challenges for the
proof of \cref{thm:main}\ref{main_generates},
because it does not establish
a direct connection with
the defining generators $\betaminus{k}(t)$ of $\bethe$.
(Of course, 
\cref{thm:generalbasis} ensures that the two definitions of
$V_E$ give the same point of $\Gr(n,2n)$.)
This can be overcome using the results of
\cite[Section 7]{purbhoo}. Alternatively, we can show that $h_t(u)$ is in the kernel of the operator $\polyD_n$ from \eqref{eq:polyDdef}, as follows. By translation, it suffices to prove this when $t = 0$:
\[
\sum_{0 \le k,\ell \le n}
(-1)^k \betaminus{k}(u) \beta^{\ell} e_{k+\ell}(u) = 0\,.
\]
This turns out to be precisely the equation obtained by taking the coefficient of $e_{[n+1]}$ in $\Omega \wedge h_t(u) = 0$, and then setting $t=u$ and translating by $u$.  
This implies \cref{lem:FDOcoords} and hence the other
statements in \cref{sec:altbetheequalsbethe}.


\subsection{Bethe algebras under geometric transformations}
\label{sec:transformations}

In this section, we discuss two examples of natural geometric transformations,
and how they manifest in the Bethe algebras $\bethe$ and $\bethenu$.

\subsubsection{Grassmann duality}
\label{sec:duality}

The Grassmannians $\Gr(d,m)$ and $\Gr(m-d,m)$ are dual to each other. We explain how to set up this duality to be compatible with the Wronski map and translation; cf.\ \cite[Section 2.4]{karp} (which uses less natural conventions) and \cite[Section 7]{purbhoo}.

Define the non-degenerate bilinear pairing $(\cdot,\cdot)$ on $\CC^m$ by
\[
(a,b) := \sum_{j=1}^m (-1)^{j-1}a_jb_{m+1-j}\,.
\]
Given $V\in\Gr(d,m)$, its \defn{dual} is the subspace
\[
\conjugate{V} := \{a\in\CC^m : (a,b) = 0 \text{ for all } b\in V\} \in \Gr(m-d,m)\,.
\]
Then $V\mapsto \conjugate{V}$ defines an isomorphism $\Gr(d,m) \to \Gr(m-d,m)$. Also, given a partition $\lambda$, we let $\conjugate{\lambda}$ denote its \defn{conjugate}, whose diagram is the transpose of that of $\lambda$.
\begin{proposition}
\label{prop:duality}
Let $V\in\Gr(d,m)$.
\begin{enumerate}[(i)]
\item\label{duality_pluckers} Taking duals preserves \Plucker coordinates: $\Delta^\lambda(V) = \Delta^{\conjugate{\lambda}}(\conjugate{V})$ for all partitions $\lambda$.
\item\label{duality_wronskian} Taking duals preserves Wronskians: $\Wr(V) = \Wr(\conjugate{V})$.
\item\label{duality_translation} Taking duals commutes with translation: $\conjugate{V(t)} = \conjugate{V}(t)$ for all $t\in\CC$.
\item\label{duality_schubert} Taking duals acts on Schubert varieties: $\conjugate{(X_\lambda(w))} = X_{\conjugate{\lambda}}(w)$ and $\conjugate{(\svarnu)} = \svar{\conjugate{\nu}}$.
\end{enumerate}
\end{proposition}

\begin{proof}
Part \ref{duality_pluckers} follows from \cite[Lemma 1.11(ii)]{karp17}. Then parts \ref{duality_wronskian} and \ref{duality_translation} follow from \eqref{eq:pluckerwronskian} and \eqref{eq:geometrictranslation}, respectively, using the fact that $\numsyt{\lambda/\mu} = \numsyt{\conjugate{\lambda}/\conjugate{\mu}}$ for all $\mu \subseteq \lambda$. For part \ref{duality_schubert}, by translation, it suffices to prove the first equality when $w=0$; this case follows from \cref{rmk:0schubertconditions} and \cref{prop:duality}\ref{duality_pluckers}. The second equality then follows by taking $w\to\infty$ with $\lambda = \nu^\vee$.
\end{proof}

This notion of duality is also compatible with the Bethe algebra. Namely, we have an involutive $\CC$-algebra automorphism $\star : \CC[\Sn] \to \CC[\Sn]$ given by
\[
\star\sigma = \sgn(\sigma)\sigma \qquad \text{for all } \sigma\in\Sn\,.
\]
As shown in \cite[Section 7]{purbhoo},
the involution $\star$ restricts to an automorphism of the Bethe
algebra $\bethe$. Indeed, we have the following result:
\begin{proposition}
\label{prop:dualityautomorphism}
The map $\star$ restricted to $\bethe$ is an involutive algebra automorphism, sending $\beta^\lambda(t) \mapsto \beta^{\conjugate{\lambda}}(t)$ for all partitions $\lambda$.
\end{proposition}

\begin{proof}
We only need to check that $\star\beta^\lambda(t) = \beta^{\conjugate{\lambda}}(t)$ for every partition $\lambda$. This follows from the fact that $\sgn(\sigma)\chi^\lambda(\sigma) = \chi^{\conjugate{\lambda}}(\sigma)$ for all $\sigma\in\Sn$.
\end{proof}

\subsubsection{\texorpdfstring{$\PGL_2(\CC)$}{PGL\textunderscore{}2(C)}-invariance}
\label{sec:invariance}

We have previously seen, in \cref{thm:bethebasics}\ref{bethebasics4}
and \cref{lem:comparealgebras}\ref{comparealgebras4}, that
the Bethe algebra $\bethe$ is translation invariant, i.e.,
$\bethe = \translatebethe{t}$ for all $t \in \CC$.
Also, since 
$\beta^\lambda$ is homogeneous of degree $n-|\lambda|$ in
the parameters $z_1, \dots, z_n$, it is clear that $\bethe$ 
is scaling invariant, i.e.,
$\bethe = \custombethe{n}(s z_1, \dots, s z_n)$
for all $s \neq 0$.  These invariance identities respect the
direct product decomposition \eqref{eq:directproductbethe}, i.e.,
we also have $\bethenu = \translatebethenu{t}$ and 
$\bethenu = \custombethe{\nu}(s z_1, \dots, s z_n)$.

The two types of invariance above correspond to the translation and
scaling actions on the Schubert variety $\svarnu$.
If $E \subseteq \spechtnu$ is an eigenspace of $\bethenu$, then
$(\beta^\lambda_E(t) : \lambda \subseteq \nu)$ are the normalized 
\Plucker coordinates of $V_E(t) = \translatematrix{t} V_E$,
and 
$(s^{n-|\lambda|}\beta^\lambda_E : \lambda \subseteq \nu)$ are the 
normalized \Plucker coordinates of 
$\{f(s^{-1}u) \mid f(u) \in V_E\} 
= \scalematrix{s} V_E$.

In the case where $\nu = \rectangle = (m-d)^d$ is a rectangle, 
we have the larger group $\PGL_2(\CC)$ acting on $\svarrect = \Gr(d,m)$.
This suggests that when $\nu$ is a rectangle, 
there may be a more general invariance statement
for $\bethenu$ corresponding to the
$\PGL_2(\CC)$-action.  We now show that this is the case.

We first consider the element 
$\inversionmatrix \in \PGL_2(\CC)$,
which acts on $\PP^1$ by inversion: $w \mapsto w^{-1}$.
Let $\beta^{\lambda, \inv}$ denote the element 
$\beta^\lambda$ with parameters $z_1, \dots, z_n \neq 0$ replaced by
their inverses $z_1^{-1}, \dots, z_n^{-1}$.  Hence the elements 
$\beta^{\lambda, \inv}$ are generators of 
$\custombethe{n}(z_1^{-1}, \dots, z_n^{-1})$.

\begin{theorem}
\label{thm:inversioninvariance}
Suppose $\nu = \rectangle$ is a rectangle, and 
$z_1, \dots, z_n \in \CC \setminus \{0\}$.
\begin{enumerate}[(i)]
\item \label{inversioninvariance1}
For every partition $\lambda \subseteq \nu$, we have
\[
\frac{\numsyt{\lambda}}{|\lambda|!} \cdot
 \beta^{\lambda, \inv}_\nu
= z_1^{-1} \dotsm z_n^{-1} \cdot
\frac{\numsyt{\lambda^\vee}}{|\lambda^\vee|!} \cdot
\beta^{\lambda^\vee}_\nu
\,,
\]
where $\lambda^\vee$ denotes the complement of $\lambda$ 
in the rectangle $\nu$.
\item \label{inversioninvariance2}
If $E \subseteq \spechtnu$ is an eigenspace of $\bethenu$, then 
$(\beta^{\lambda, \inv}_E : \lambda \subseteq \nu)$
are the normalized \Plucker coordinates of 
$\inversionmatrix V_E$.
\item \label{inversioninvariance3}
$\bethenu = \custombethe{\nu}(z_1^{-1}, \dots, z_n^{-1})$\,.
\end{enumerate}
\end{theorem}

\begin{proof}
Let $X \subseteq [n]$ and $Y := [n] \setminus X$.  As in the proof
of \cref{prop:betapsd}, let
$\alpha^\lambda_{X,\nu}$ denote the
operator $\alpha^\lambda_X$ acting on $\spechtnu$. We claim that as representations of $\symgrp{X} \times \symgrp{Y}$, we have the decomposition
\begin{equation}
\label{eq:rectangledecomposition}
    \spechtnu \simeq \bigoplus_{\lambda \vdash |X|}
     \specht{\lambda}_X \otimes \specht{\lambda^\vee}_Y
\,.
\end{equation}
Indeed, by Frobenius reciprocity, \eqref{eq:rectangledecomposition} is equivalent to the fact that the \emph{Littlewood--Richardson coefficient} $c^{\nu}_{\lambda,\mu}$ (cf.\ \cite[Section 7.15]{stanley24}) equals $1$ if $\mu = \lambda^\vee$, and $0$ otherwise. This follows from, e.g., \cite[(9.11)]{fulton97}.

Therefore, by \cref{prop:projection}, 
$\frac{\numsyt{\lambda}}{|\lambda|!} 
\alpha^\lambda_{X,\nu}$ and
$\frac{\numsyt{\lambda^\vee}}{|\lambda^\vee|!} 
\alpha^{\lambda^\vee}_{Y,\nu}$ are both equal to the orthogonal projection
onto the summand 
$\specht{\lambda}_X \otimes \specht{\lambda^\vee}_Y$.  In particular,
we have 
\[
   \frac{\numsyt{\lambda}}{|\lambda|!} \cdot
\alpha^\lambda_{X,\nu} = 
\frac{\numsyt{\lambda^\vee}}{|\lambda^\vee|!}  \cdot
\alpha^{\lambda^\vee}_{[n] \setminus X,\nu}
\qquad\text{for all $X \subseteq [n]$}\,.
\]
Part \ref{inversioninvariance1} now
follows directly from \eqref{eq:alphatobeta}.

If $V \in \Gr(d,m)$ has
\Plucker coordinates $[\Delta^\lambda : \lambda \subseteq \rectangle]$,
then a direct calculation using \eqref{eq:hookformula} shows that 
$\big[\frac{|\lambda|!}{\numsyt{\lambda}}\cdot
\frac{\numsyt{\lambda^\vee}}{|\lambda^\vee|!}
\cdot \Delta^{\lambda^\vee} : \lambda \subseteq \rectangle\big]$
are the \Plucker coordinates of
$\inversionmatrix V \in \Gr(d,m)$.
Part \ref{inversioninvariance2} therefore follows from 
part \ref{inversioninvariance1} and \cref{cor:identifyscell}.

Finally, part \ref{inversioninvariance3} follows from 
part \ref{inversioninvariance1} and \cref{thm:main}\ref{main_generates}.
\end{proof}

As a corollary, we obtain the following
more general $\PGL_2(\CC)$-invariance result:

\begin{corollary}
\label{cor:invariance}
Suppose $\nu = \rectangle$ is a rectangle, $z_1, \dots, z_n \in \CC$, 
and $\phi \in \PGL_2(\CC)$ is such that 
$\phi(z_i) \neq \infty$ for $i=1, \dots, n$.  Then
\[
   \bethenu = \custombethe{\nu}(\phi(z_1), \dots, \phi(z_n))
\,.
\]
\end{corollary}

\begin{proof}
This follows from the preceding invariance statements, and the fact that
the matrices
$\translatematrix{t}$, $\scalematrix{s}$, and $\inversionmatrix$
generate $\PGL_2(\CC)$.
\end{proof}


\subsection{Combinatorics of commutativity}
\label{sec:commcomb}

For $\theta \in \Sn$ and $Z \subseteq [n]$, 
recall from \cref{sec:zfac} that a $Z$-factorization of $\theta$
is a pair of supported permutations
$(\sigma_X, \pi_Y)$ such that $X\cup Y = [n]$, $X \cap Y = Z$, and 
$\sigma \pi = \theta$.
Hence a $Z$-factorization of $\theta$ is a factorization of $\theta$,
together with some additional data relating to fixed points of
the factors.  The commutativity relation 
$\varepsilon^\lambda \varepsilon^\mu
= \varepsilon^\mu \varepsilon^\lambda$
(cf.\ \cref{sec:commutativity})
can be reformulated as a non-trivial combinatorial identity:

\begin{theorem}
For $\theta \in \Sn$, $Z \subseteq [n]$, and partitions $\lambda, \mu$, 
let $\mathrm{Fac}_{\theta,Z}(\lambda,\mu)$ denote
the set of $Z$-factorizations $(\sigma_X, \pi_Y)$ of $\theta$
such that $\cyc(\sigma_X) = \lambda$ and $\cyc(\pi_Y) = \mu$.
Then
\begin{equation}
\label{eq:Zfcyc}
   \#\mathrm{Fac}_{\theta,Z}(\lambda, \mu) 
       = \#\mathrm{Fac}_{\theta,Z}(\mu,\lambda)
\,.
\end{equation}
\end{theorem}

\begin{proof}
Compare coefficients of $\theta z_{[n]\setminus Z}$ on both sides
of $\varepsilon^\lambda \varepsilon^\mu 
= \varepsilon^\mu \varepsilon^\lambda$.
\end{proof}

A direct proof of \eqref{eq:Zfcyc} would imply the commutativity relations \eqref{eq:commutativity},
without going through the reduction in \cref{sec:creduction}. We pose this as an open problem:
\begin{problem}
\label{prob:bijection}
Give a bijective proof of \eqref{eq:Zfcyc}.
\end{problem}

The identity \eqref{eq:reducedcommutativity} is much easier to prove
bijectively than \eqref{eq:Zfcyc}, because we only have to keep track of the cycle type of 
one of the two factors.


\subsection{A \texorpdfstring{$\tau$}{tau}-function with coefficients in \texorpdfstring{$\CC[\symgrp{\infty}]$}{C[S\textunderscore{}infinity]}}
\label{sec:tauinfinity}

Since $[n-1] \subseteq [n]$, we have natural inclusions
\[
   \symgrp{1} \hookrightarrow \symgrp{2} \hookrightarrow
    \symgrp{3} \hookrightarrow \cdots
\,.
\]
The group $\symgrp{\infty}$ is defined to be the direct limit of this
sequence of inclusions.   Equivalently, $\symgrp{\infty}$ is the subgroup of
permutations of $\{1,2,3, \dots\}$ that fix all but finitely many
positive integers.  We say that $\sigma_X$ is a 
\defn{supported permutation of $\symgrp{\infty}$} if
$X$ is a \emph{finite} subset of $\{1,2,3, \dots\}$, and 
$\sigma \in \symgrp{X}$.
Hence $\SP_\infty := \bigcup_{n=0}^\infty \SP_n$ is the set of all 
supported permutations of $\symgrp\infty$.

Equation~\eqref{eq:KP} is unaffected by rescaling $\tau$ by a constant.
If $z_1, \dots, z_n \in \CC\setminus \{0\}$,
we can rescale the $\tau$-function $\tau_n$ from \eqref{eq:tau} by 
$z_1^{-1} \dotsm z_n^{-1}$:
\[
    z_1^{-1} \dotsm z_n^{-1} \tau_n = 
     \sum_{\sigma_X \in \SP_n} \sigma \, z_{X}^{-1} \otimes \p_{\cyc(\sigma_X)}.
\]
This is still a $\tau$-function of the KP hierarchy, which extends formally to $n=\infty$:

\begin{theorem}
\label{thm:tauinfinity}
Let $z_1, z_2, \dots$ be formal indeterminates.  The series
\[
   \tau_\infty := 
     \sum_{\sigma_X \in \SP_\infty} \sigma \, z_{X}^{-1} \otimes \p_{\cyc(\sigma_X)}
\]
is a $\tau$-function of the KP hierarchy, with coefficients
in (a commutative subalgebra of) 
$\CC[\symgrp\infty][[z_1^{-1}, z_2^{-1}, \dots]]$.
\end{theorem}

\begin{proof}
As $z_1, z_2, \dots$ are formal, the result is true 
if and only if it is true
whenever all but finitely many $z_i^{-1}$ are set to zero.  Hence
the result is equivalent to \cref{thm:taufunction}.
\end{proof}

Recall from \cref{thm:KPplucker} that if $\tau\in\Lambda$ is a $\tau$-function of the KP hierarchy, then it defines a point in $\Gr(d,m)$ whenever $d$ and $m-d$ are sufficiently large. This may no longer hold when $\tau \in \widehat\Lambda$; rather, $\tau$ defines a point in the infinite-dimensional \emph{Sato Grassmannian}. The Wronskian of a point in the Sato Grassmannian does not have a determinantal definition. However, we can define it as a formal series as in \eqref{eq:pluckerwronskian}, or equivalently, as the exponential specialization $\exspec(\tau)$ as in \cref{rmk:taufunction}. It may be interesting to study $\tau_\infty$ and $\exspec(\tau_\infty)$ as analytic objects, and connect them to the inverse Wronski problem for spaces of analytic functions (rather than just polynomials):
\begin{problem}
Find an analytic version of \cref{thm:tauinfinity} when $(z_1, z_2, \dots)$ is an infinite sequence of nonzero complex numbers satisfying an appropriate convergence condition.
\end{problem}


\subsection{Higher-degree positivity for \Plucker coordinates}
\label{sec:asw}

In the case of $\Gr(1,m)$, \cref{thm:positive}\ref{positive1} can be reformulated as the following (rather obvious) statement: if $f\in\CC[u]$ is a polynomial of degree at most $m-1$ whose zeros are all are real and nonpositive, then (up to rescaling) $f$ has nonnegative real coefficients. This is a much weaker statement than the following result, which gives a complete characterization of polynomials with nonpositive real zeros:
\begin{theorem}[{Aissen, Schoenberg, and Whitney \cite{aissen_schoenberg_whitney52}}]
\label{thm:asw}
Let $f(u) = \sum_{i=0}^d a_iu^i \in\CC[u]$ be a polynomial with at least one nonnegative real coefficient. Then all the zeros of $f$ are real and nonpositive if and only if the infinite Toeplitz matrix
\begin{equation}
\label{eq:toeplitz}
\begin{pmatrix}
a_0 & 0 & 0 & \cdots \\
a_1 & a_0 & 0 & \cdots \\
a_2 & a_1 & a_0 & \cdots \\
\vdots & \vdots & \vdots & \ddots
\end{pmatrix}
\end{equation}
is totally nonnegative (i.e.\ all its finite minors are real and nonnegative).
\end{theorem}

Note that nonnegativity of the $1\times 1$ minors of \eqref{eq:toeplitz} is equivalent to the sequence of coefficients $(a_0, \dots, a_d)$ being nonnegative, and nonnegativity of the $2\times 2$ minors implies that the sequence is log-concave. We pose the problem of simultaneously generalizing \cref{thm:positive}\ref{positive1} (from the case of linear inequalities to higher-order inequalities) and \cref{thm:asw} (from $\Gr(1,m)$ to $\Gr(d,m)$):
\begin{problem}
\label{prob:asw}
Give a characterization of the subset of elements $V\in\Gr(d,m)$ such that all zeros of $\Wr(V)$ are real and nonpositive, in terms of polynomial inequalities in the \Plucker coordinates.
\end{problem}


\subsection{Conjectures in real Schubert calculus}
\label{sec:openconjectures}

We briefly discuss the more general forms of the Shapiro--Shapiro 
conjecture and the secant conjecture, as well as the total reality conjecture for convex curves.

\subsubsection{General form of the Shapiro--Shapiro conjecture}
\label{sec:generalssc}

The more general form of the Shapiro--Shapiro conjecture is the following
statement:

\begin{theorem}[{Mukhin, Tarasov, and Varchenko 
\cite[Corollary 6.3]{mukhin_tarasov_varchenko09b}}]
\label{thm:generalssc}
Let $\nu \vdash n$ and $\boldmu \multipartition \kappa$, where
$\kappa = (\kappa_1, \dots, \kappa_s)$ is a composition of $n$.
Also let $z_1, \dots, z_s$ be distinct real numbers.  Then the Schubert
intersection \eqref{eq:generalssc} is real and scheme-theoretically
reduced.
\end{theorem}

The proof below is a reformulation of the argument in
\cite{mukhin_tarasov_varchenko09b}:

\begin{proof}
By \cref{prop:betapsd}\ref{betapsd1}, 
the operators 
$\beta^\lambda_\nu \in \multibethe{\nu}$ are self-adjoint, and hence the
eigenvalues of $\beta^\lambda_{\nu,\boldmu}$ are real.  This implies
that the points of the intersection \eqref{eq:generalssc} are real.  By 
\cref{thm:nondistinct}, the intersection
is scheme-theoretically reduced if and only if $\multibethe{\nu,\boldmu}$
is semisimple.  This again follows from the fact that the generators
are self-adjoint.
\end{proof}

The reality statement can also be deduced from \cref{thm:ssc},
since the Schubert intersection
\eqref{eq:generalssc} is contained in the fibre of the Wronskian
\eqref{eq:fibre}, which is a limit of fibres $\Wr^{-1}(g)$ where
$g$ has distinct real roots.  
The reducedness, however, is more subtle, and does not
readily follow by limiting arguments.

We also obtain the following generalization of \cref{thm:positive}:
\begin{corollary}[Positive Shapiro--Shapiro conjecture]
\label{cor:generalpositivenu}
Let $z_1, \dots, z_s \in \CC$ be distinct.
\begin{enumerate}[(i)]
\item\label{generalpositivenu1} If $z_1, \dots, z_s \in [0,\infty)$, then all points of the Schubert intersection~\eqref{eq:generalssc} are totally nonnegative.
\item\label{generalpositivenu2} If $z_1, \dots, z_s \in (0,\infty)$, then all points of the Schubert intersection~\eqref{eq:generalssc} are totally positive in $\scellnu$.
\end{enumerate}
\end{corollary}

\begin{proof}
This follows from \eqref{eq:fibre} and the proof of \cref{thm:positive} in \cref{sec:conjectureproofs}.
\end{proof}

\subsubsection{General form of the secant conjecture}
\label{sec:generalsecant}

Similarly, there is the general form of the secant conjecture, as formulated
by Sottile around 2003 (see \cref{sec:secantconjecture} for a discussion of the history).
For an interval $I \subseteq \RR$, we say that a complete flag 
$F_\bullet : F_0 \subsetneq \dots \subsetneq F_m$ in $\CC^m$
is a \defn{generalized secant flag} to the moment curve $\gamma$ along $I$ if each 
subspace $F_i$ is a generalized secant to $\gamma$ along $I$.

\begin{conjecture}[Secant conjecture, general form]
\label{conj:generalsecant}
Let $\nu \vdash n$ and $\boldmu \multipartition \kappa$, where
$\kappa = (\kappa_1, \dots, \kappa_s)$ is a composition of $n$.
Let $I_1, \dots, I_s  \subseteq \RR$ be pairwise disjoint real intervals.
If $F^{(1)}_\bullet, \dots , F^{(s)}_\bullet$ are generalized secant
flags to $\gamma$ along $I_1, \dots , I_s$, respectively, 
then the Schubert intersection
\eqref{eq:generalschubertintersection} is real and scheme-theoretically
reduced.
\end{conjecture}

\cref{thm:secant} addresses the divisor case of \cref{conj:generalsecant}, i.e., where $\mu_i = 1$ for all $i$ in \eqref{eq:generalschubertintersection}. (We mention that the proof can be generalized in a straightforward way to handle the case where for each $i$, we have $\mu_i = 1$ or $F^{(i)}_\bullet$ is an osculating flag.)
Unfortunately, in general, neither the 
reality nor the reducedness statements of \cref{conj:generalsecant}
follow readily from \cref{thm:secant}.  
In particular, the Schubert intersections described in 
\cref{conj:generalsecant} cannot be realized as limits of
the Schubert intersections in \cref{thm:secant}, since the intervals $I_i$ are not allowed to overlap. Also, \cref{conj:generalsecant} does not appear to follow easily from Theorems~\ref{thm:disconj} or \ref{thm:positive}.
For now, therefore, the general form of the secant conjecture remains open.

\subsubsection{Total reality conjecture for convex curves}
\label{sec:totalreality}

Recall that the Shapiro--Shapiro conjecture (\cref{thm:ssc}) asserts that the solutions to the Schubert problem \eqref{eq:schubertproblem} are all real when each $W_i\in\Gr(m-d,m)$ is an osculating plane to the moment curve $\gamma$ at a real point. There is a generalization of this conjecture due to Boris Shapiro in the 1990's (cf.\ \cite[Section 1]{sedykh_shapiro05}), called the \defn{total reality conjecture for convex curves}. Namely, a real continuous curve $\rho : \RR \to \RR_{m-1}[u]$ is called \defn{(strictly) convex} if any choice of $m$ distinct points along $\rho$ are linearly independent. By continuity, this is equivalent to the condition that for all $r\ge m$ and $w_1 < \cdots < w_r$ in $\RR$,
\[
\begin{pmatrix}
\vline & & \vline \\
\rho(w_1) & \cdots & \rho(w_r) \\
\vline & & \vline
\end{pmatrix} \;\text{ represents a totally positive element of } \Gr(m,r)\,,
\]
i.e., its $m\times m$ minors all have the same (nonzero) sign. The fact that the moment curve $\gamma$ is convex follows from Vandermonde's determinantal formula.

\begin{conjecture}[Total reality conjecture for convex curves]
\label{conj:totalreality}
Let $\rho: \RR \to \RR_{m-1}[u]$ be a convex curve, and let $z_1, \dots, z_{d(m-d)}$ be distinct real numbers. For $i = 1, \dots, d(m-d)$, let $W_i\in\Gr(m-d,m)$ be the osculating $(m-d)$-plane to $\rho$ at $z_i$. Then there are exactly $\numsyt{\rectangle}$ distinct solutions to the Schubert problem \eqref{eq:schubertproblem} (with $\nu = \rectangle$), and all solutions are real.
\end{conjecture}

In analogy with the Shapiro--Shapiro conjecture, there are several natural extensions of \cref{conj:totalreality}, such as to an arbitrary Schubert cell $\scellnu$, to the case that $W_1, \dots, W_n$ are generalized secants to $\rho$ along disjoint real intervals, and to Schubert conditions of arbitrary codimension.

\cref{conj:totalreality} was long thought to be false, due to a counterexample which was only recently found to be erroneous \cite{sedykh_shapiro05,shapiro_shapiro22}. It is true for $\Gr(2,4)$, which follows from work of Arkani-Hamed, Lam, and Spradlin on scattering amplitudes \cite[Section 4]{arkani-hamed_lam_spradlin21}; cf.\ \cite[Theorem 2]{shapiro_shapiro22}. It is not immediately clear how to apply our techniques to address \cref{conj:totalreality}, because the connection between the Schubert problem \eqref{eq:schubertproblem} and the Wronski map is particular to the case that $\rho$ is the moment curve $\gamma$.

It would be interesting to explain the role played by total positivity. Namely, in the total reality conjecture, positivity is used to define the Schubert conditions; and in \cref{thm:positive,thm:positivesecant}, positivity is a property of the solutions of the Schubert problem. The fact that positivity appears in both places appears for now to be a coincidence.


\footnotesize
\makeatletter
\renewcommand\@openbib@code{\itemsep-2pt}
\makeatother
\bibliographystyle{alpha}
\bibliography{ref}

\end{document}